\documentclass[a4paper,11pt]{amsart}

\usepackage{amssymb}
\usepackage{amsrefs}

\usepackage{comment}

\usepackage[hmargin=3.5cm,foot=0.7cm]{geometry}
\usepackage[utf8x]{inputenc}

\usepackage[english]{babel}

\usepackage{tikz}
\usetikzlibrary{matrix,arrows}

\usepackage{tikz-cd}
\usepackage{enumitem}
\usepackage{bm}

\BibSpec{article}{%
  +{}{\PrintAuthors}  		{author}
  +{,}{ \textit}     		{title}
  +{,}{ }             		{journal}
  +{}{ \textbf}       		{volume}
  +{}{ \parenthesize} 		{date}
  +{,}{ }      	      		{conference}
  +{,}{ }      	      		{book}
  +{,}{ }            		{pages}
  +{,}{ }            	 	{note}
  +{,}{ }            	 	{status}
  +{,}{  \texttt } {eprint}
  +{.}{}              {transition}
}
\BibSpec{book}{%
  +{}{\PrintAuthors}  {author}
  +{,}{ \textit}      {title}
  +{,}{ }	      {note}
  +{,}{ }             {publisher}
  +{,}{ }             {place}
  +{,}{ }             {date}
  +{.}{}              {transition}
}

\newtheorem{theorem}{Theorem}[section]
\newtheorem*{theorem*}{Theorem}

\newcounter{mt}

\newtheorem{maintheorem}[mt]{Theorem}

\theoremstyle{plain}
\newtheorem{corollary}[theorem]{Corollary}
\newtheorem{lemma}[theorem]{Lemma}
\newtheorem{proposition}[theorem]{Proposition}

\newtheorem*{lemma*}{lemma}
\newtheorem*{question*}{Question}
\theoremstyle{definition}
\newtheorem{definition}[theorem]{Definition}

\newtheorem{remark}[theorem]{Remark}

\newcommand{\CC}{\mathbb{C}}
\newcommand{\C}{\mathbb{C}}

\newcommand{\RR}{\mathbb{R}}

\newcommand{\calC}{\mathcal{C}}

\newcommand{\calK}{\mathcal{K}}
\newcommand{\calL}{\mathcal{L}}

\newcommand{\calR}{\mathcal{R}}

\newcommand{\calV}{\mathcal{V}}

\newcommand{\KLK}{\calK \calL\calK}

\newcommand{\dg}{d_\nabla}
\newcommand{\bw}{{\bigwedge}}
\newcommand{\largewedge}{\mbox{\Large $\wedge$}}

\newcommand \restrict[2]{\left. #1\right|_{#2}}

\newcommand \contr[1]{\iota_{#1}}

\newcommand{\inv}{^{-1}}
\newcommand{\spann}{\mathrm{span}}

\DeclareMathOperator{\dblR}{\mathsf R}
\DeclareMathOperator{\dblg}{\mathsf g}
\DeclareMathOperator{\dblG}{\mathsf G}
\DeclareMathOperator{\dblC}{\mathsf C}
\DeclareMathOperator{\CP}{\C P}
\DeclareMathOperator{\Val}{Val}

\DeclareMathOperator{\pd}{pd}
\DeclareMathOperator{\vol}{vol}
\DeclareMathOperator{\id}{id}

\DeclareMathOperator{\glob}{glob}
\DeclareMathOperator{\Grass}{Gr}
\DeclareMathOperator{\Sym}{Sym}

\DeclareMathOperator{\SL}{SL}
\DeclareMathOperator{\SO}{SO}
\DeclareMathOperator{\U}{U}

\DeclareMathOperator{\nc}{nc}
\DeclareMathOperator{\Curv}{Curv}

\makeatletter
\DeclareRobustCommand{\pder}[1]{%
  \@ifnextchar\bgroup{\@pder{#1}}{\@pder{}{#1}}}
\newcommand{\@pder}[2]{\frac{\partial#1}{\partial#2}}
\makeatother

\makeatletter
\@namedef{subjclassname@2020}{\textup{2020} Mathematics Subject Classification}
\makeatother

\begin{document}

\title{The Weyl tube theorem for K\"ahler manifolds}
\author{Andreas Bernig}
\email{bernig@math.uni-frankfurt.de}
\address{Institut f\"ur Mathematik, Goethe-Universit\"at Frankfurt,
Robert-Mayer-Str. 10, 60054 Frankfurt, Germany}

\author{Joseph H.G. Fu}
\email{joefu@uga.edu}
\address{Department of Mathematics, University of Georgia, Athens, GA 30602, USA}

\author{Gil Solanes}
\email{solanes@mat.uab.cat}
\address{Departament de Matem\`atiques, Universitat Aut\`onoma de Barcelona, 08193 Bellaterra, Spain, and Centre de Recerca Matem\`atica, Campus de Bellaterra, 08193 Bellaterra, Spain}

\author{Thomas Wannerer}
\email{thomas.wannerer@uni-jena.de}
\address{Friedrich-Schiller-Universit\"at Jena, Fakult\"at f\"ur Mathematik \& Informatik, Institut für Mathematik, Ernst-Abbe-Platz 2, 07743 Jena, Germany}

\thanks{A.B. was supported by DFG grant BE 2484/5-2. 
	J.H.G.F. was supported by a Simons Collaboration Grant.
	G.S. was supported by the Serra Húnter Programme,  MICIU/AEI/FEDER grant PGC2018-095998-B-I00, and  MICINN/AEI María de Maeztu grant CEX2020-001084-M.
T.W. was supported by DFG grant WA 3510/3-1.}

\begin{abstract} 
As sharpened in terms of Alesker's  theory of valuations on manifolds, a classic theorem of Weyl asserts that the coefficients of the tube polynomial of an isometrically embedded riemannian manifold $M \hookrightarrow \mathbb R^n$ constitute a canonical finite dimensional  subalgebra  $\mathcal {L K}(M)$ of the algebra $\mathcal{V} (M)$ of all smooth valuations on $M$, isomorphic to the algebra of  valuations on Euclidean space that are invariant under rigid motions. We construct an analogous, larger, canonical subalgebra 
$\mathcal{KLK}(M)\subset \mathcal{V}(M)$ for K\"ahler manifolds $M$: i) if $\dim M = n $,  then $\mathcal{KLK}(M)\simeq \mathrm{Val}^{\mathrm{U}(n)}$,  the algebra of valuations on $\mathbb{C}^n$ invariant under the holomorphic isometry group, and ii) if $M\hookrightarrow \tilde M$ is a K\"ahler embedding, then the  restriction map $\mathcal{V}(\tilde M) \to \mathcal{V}(M)$ induces a surjection $\mathcal{KLK}(\tilde M)\to \mathcal{KLK}(M)$. This answers a question posed by Alesker in 2010 and gives a structural explanation for some previously known, but mysterious phenomena in hermitian integral geometry.
\end{abstract}


\maketitle

\tableofcontents

\section{Introduction} 

\subsection{Lipschitz-Killing curvatures} 

Given a smooth compact submanifold $M\subset \RR^n$,  Weyl  \cite{weyl_tubes} expressed the volume of a tubular neighborhood about $M$ of small radius $r$ as a polynomial of degree $n$ in $r$. Weyl's deeper point, however, was that the coefficients are given by integrals of invariants of the curvature tensor of $M$ with 
respect to the induced metric.   These integrals, now known as {\it total Lipschitz-Killing curvatures}, 
are extensions of the {\it intrinsic volumes} of convex geometry, 
 the same quantities that arise in the kinematic formulas of Blaschke, Chern, and Santal\'o \cite{santalo76}. Examples of total Lipschitz-Killing curvatures are the Euler characteristic, the total scalar curvature and the volume  of $M$.

The (total) Lipschitz-Killing curvatures of riemannian manifolds appear in various contexts. Let us mention just some of them. Cheeger-M\"uller-Schrader proved that the Lipschitz-Killing curvatures of a riemannian manifold can be obtained as limits of the Lipschitz-Killing curvatures of a sequence of piecewise linear manifolds converging to the manifold in an appropriate sense. The Lipschitz-Killing curvatures of piecewise linear manifolds are of a more combinatorial nature, and have been used by physicists in {\it Regge calculus}, an attempt to discretize important quantities in general relativity such as the total scalar curvature. 

Weyl's tube formula was the main ingredient in the proof by Allendoerfer and Weil of the Chern-Gauss-Bonnet theorem. The heat equation approach to the Chern-Gauss-Bonnet theorem by Atiyah, Patodi, and Singer is also related to Lipschitz-Killing curvatures. Recall that the short time asymptotics of the heat kernel associated to the Laplacian acting on differential forms is given by complicated polynomials in the curvature tensor and its derivatives. McKean and Singer conjectured that some particular linear combination of these polynomials equals the integrand in the Chern-Gauss-Bonnet theorem. In particular, this would imply a ``fantastic cancellation" of all terms involving derivatives. The conjecture was shown by Patodi and led to a heat equation proof of the Chern-Gauss-Bonnet theorem and of the more general Atiyah-Singer index theorem \cite{atiyah_bott_patodi}. Some years later, Donnelly \cite{donnelly} observed that there are other linear combinations of the heat kernel coefficients that do not contain derivatives of the curvature tensor; and these combinations give precisely the Lipschitz-Killing curvatures. In particular, the total Lipschitz-Killing curvatures are spectral invariants.

More recently, the theory of valuations, introduced by Alesker in the early 2000s, has enriched the context for these classical facts: the functionals $\mu_i^M$ that yield the tube coefficients may be viewed as {\it smooth
valuations}---finitely additive measures---canonically associated to the riemannian structure of $M$. They  are called the {\em intrinsic volumes} on $M$. Remarkably, Chern's most general version of the Gauss-Bonnet theorem \cite{chern45} shows that $\mu_0^M=\chi$ is a smooth valuation (without using this terminology of course).

A new and fascinating development is the extension of Lipschitz-Killing curvatures to Alexandrov spaces. Alesker \cite{alesker_alexandrov_spaces} stated precise and far-reaching conjectures about the limits of the Lipschitz-Killing curvatures of a sequence of manifolds with a uniform lower bound on the sectional curvature converging to an Alexandrov space. Some special cases of the conjecture are known (roughly for volume and total scalar curvature). For recent results in this direction we refer to \cites{alesker_riemannian_submersions, alesker_katz_prosanov}.

The main feature of  the space $\calV(M)$ of smooth valuations on $M$ is the existence of  a natural commutative product. From the perspective of Alesker's theory  the intrinsic volumes are distinguished by two properties: invariance under (i) isometric embeddings and (ii) the formation of Alesker products. The resulting  subalgebra $\mathcal{LK}(M)\subset \calV(M)$ is canonically isomorphic to $\RR[t]/(t^{n +1})$, where $n = \dim M$, and  is called the {\it Lipschitz-Killing algebra.} Up to scale, the powers $t^i$ coincide with the intrinsic volumes $\mu_i^M$. In the case $M = \RR^n$, this subalgebra coincides with the algebra $\Val^{\mathrm O(n)}$ of valuations
invariant under the euclidean group. In this context a natural grading exists, with $\deg t =1$. 

The existence of the product structure is related to the existence of kinematic formulas on isotropic spaces. The algebras of isometry invariant valuations on euclidean space, the euclidean sphere and hyperbolic space coincide with the Lipschitz-Killing algebra and are thus mutually isomorphic. This can be translated into the fact that the kinematic formulas on the real space forms are isomorphic. 

Bernig, Faifman, and Solanes \cites{bernig_faifman_solanes, bernig_faifman_solanes_part2,bernig_faifman_solanes_part3} have recently constructed intrinsic volumes with analogous properties for pseudo-riemannian manifolds.

\subsection{The K\"ahler case}

A natural question is to what extent the above can be generalized in the complex setting. To be more precise, at Oberwolfach in 2010, Alesker asked whether the classical picture admits a K\"ahler analogue: is there a canonical subalgebra of valuations for K\"ahler manifolds $M$, analogous to the Lipschitz-Killing algebra for the riemannian case?

Although tubes around complex submanifolds of flat hermitian space $\C^n$ or of complex projective space $\CP^n$ do contain information about Chern classes \cite{shifrin81,bernig_fu_solanes}, it is no surprise that they do not reflect the complex geometry. We are therefore forced to look at this problem from the valuations point of view. 

In the flat hermitian case $\C^n$, we have a complete understanding of the space of invariant smooth valuations. The entry point was the determination \cite{alesker03_un, fu06} of the structure of the algebra $\Val^{\U(n)}$ of smooth valuations on $\CC^n$ invariant under the group of holomorphic isometries: in this case 
\begin{equation}\label{eq:unitaryAlg} \Val^{\U(n)}\simeq \RR[s,t]/(f_{n+1},f_{n+2}),
\end{equation}	
 where $\deg s =2$ and the relations are 
homogeneous of degrees $n+1, n+2$. Thus the monomials $s^it^j$ furnish a natural spanning set for $\Val^{\U(n)}$. Up to a factor, $t$ corresponds to the first intrinsic volume, which can be defined by averaging the  Euler characteristic of intersections with real hyperplanes. The other generator, $s$, is defined by averaging intersections with complex hyperplanes. 

Another useful basis,  more natural from a geometric viewpoint, are the {\it hermitian intrinsic volumes} $\mu_{k,p}$ from \cite{bernig_fu_hig}. The isomorphism \eqref{eq:unitaryAlg} can be translated into explicit kinematic formulas for the pair $(\C^n,\U(n) \ltimes \CC^n)$, see \cite{bernig_fu_hig}.

On complex projective space $\CP^n$, $t$ may be defined as the first intrinsic volume (with respect to the riemannian structure provided by the Fubini-Study metric), while $s$ can be defined using intersections with complex hyperplanes inside $\CP^n$. The algebra generated by these two valuations equals the algebra of isometry invariant valuations on $\CP^n$. Surprisingly, and opposed to what happens in the real case, the map $t \mapsto t, s \mapsto s$ from  $\Val^{\U(n)}$  to $\mathcal V(\CP^n)^{\mathrm{Isom}}$ is not an isomorphism. Even more surprisingly, as observed in \cite{bernig_fu_solanes}, there are nevertheless some isomorphisms between  $\Val^{\U(n)}$ and $\mathcal V(\CP^n)^{\mathrm{Isom}}$. In particular, the kinematic formulas are isomorphic. 

Previously, there was no satisfactory explanation of this phenomenon. Inspired by Weyl's tube approach and Alesker's refinement in the riemannian setting, we clarify the situation by providing a vast generalization of the notion of hermitian intrinsic volumes to K\"ahler manifolds.

Our main result gives a positive answer to Alesker's question. 
 
\begin{maintheorem}\label{mainthm_global}
For any K\"ahler manifold $M$ of complex dimension $n$, there is a canonical subalgebra $\KLK(M)\subset \calV(M)$, isomorphic
to $\Val^{\U(n)}$, with the property that if $M' \hookrightarrow M$ is a K\"ahler embedding then the natural 
restriction map $\calV(M) \to \calV(M')$ restricts to a natural surjection $\KLK(M) \to \KLK(M')$ such that the following diagram commutes:
	\begin{center}
		\begin{tikzpicture}
			\matrix (m) [matrix of math nodes,row sep=3em,column sep=4em,minimum width=2em]
			{
			\Val^{\U(n)}  &  \Val^{\U(n')}\\
			\KLK(M)   &  \KLK(M')\\};
			\path[-stealth]
			(m-1-1) edge node [right] {$\cong$}  (m-2-1)
			(m-1-1) edge (m-1-2)
			(m-1-2) edge node [right] {$\cong$}  (m-2-2) 
			(m-2-1) edge (m-2-2);
		\end{tikzpicture}
	\end{center}
where the upper horizontal map is the restriction map induced by the inclusion $\C^{n'} \hookrightarrow \C^n$.
\end{maintheorem}

In geometric terms, the isomorphism with $\Val^{\U(n)}$ maps the hermitian intrinsic volume $\mu_{k,p}$ to the hermitian intrinsic volume $\mu_{k,p}^M$ that we will construct (cf. Theorem \ref{thm:alg_isom} below).

Since the isometries of complex projective (or complex hyperbolic) space are K\"ahler embeddings, the elements of $\mathcal{KLK}(\CP^n)$ are isometry invariant. Our main theorem thus provides a new perspective on invariant valuations and integral geometry of $\CP^n$. 

\begin{corollary} \label{cor_canonical_isomorphism}
The algebras of isometry invariant smooth valuations on $n$-dimensional complex projective and on $n$-dimensional complex hyperbolic spaces are canonically isomorphic to $\Val^{\U(n)}$. 
\end{corollary}

Several isomorphisms between these algebras, algebraically simple but geometrically unmotivated, were previously discovered in \cite{bernig_fu_solanes}. The isomorphism constructed in the present paper, however, is canonical in the sense that it arises from the K\"ahler structures on the underlying spaces. In terms of the valuations $t$ and $s$ described above, the canonical isomorphism is given by $t \mapsto t$ and $s \mapsto s/(1-\lambda s)$ (see Prop. \ref{prop:sigma}). This particular isomorphism was a key ingredient in the recent paper \cite{bernig_fu_conjecture}.

\subsection{Local setting }

The kinematic formulas due to Chern, Blaschke, and Santal\'o were generalized by Federer in two directions. First, instead of taking compact convex sets, he considered sets of positive reach. Note that compact convex bodies as well as compact submanifolds, possibly with boundary, are sets of positive reach. Secondly, he localized the kinematic formulas by looking at so called {\it curvature measures}, which are localizations of the intrinsic volumes. In the euclidean case, there is a one-one correspondence between intrinsic volumes and curvature measures. Consequently, Federer's local kinematic formulas look formally the same as the previously known global kinematic formulas. The same happens on other real space forms, where each invariant valuation admits a unique localization. We thus find that on all real space forms of the same dimension, the local kinematic formulas are isomorphic. 

In \cite{fu_wannerer}, a sequence of curvature measures naturally associated to riemannian manifolds was constructed. The main observation was that the space of such curvature measures is a {\it universal module} over the Lipschitz-Killing algebra, i.e. the structure coefficients do not depend on the riemannian manifold. In the case of real space forms, the module structure is related to a version of kinematic formulas called semi-local kinematic formulas. Hence such formulas on all real space forms of the same dimension are  isomorphic.

In the complex case, it follows easily from Howard's transfer principle (which applies more generally to homogeneous manifolds \cite{howard93}) that local kinematic formulas on complex space forms of the same dimension are formally identical. To write them down explicitly is highly non trivial. The reason is that, in contrast to the real case, the globalization map is not injective, so the local kinematic formulas cannot be derived from the global ones. The problem was solved in \cite{bernig_fu_solanes}, see also \cite{bernig_fu_solanes_proceedings} for a shorter way of describing such formulas.

In view of our first main theorem \ref{mainthm_global} and of the results concerning riemannian curvature measures mentioned above it is natural to ask for a K\"ahler version. This is our second main theorem. 
\begin{maintheorem} \label{mainthm_local}
For any K\"ahler manifold $M$ of complex dimension $n$, there is a canonical $\KLK(M)$-module $\widetilde{\KLK}(M) \subset \calC(M)$, isomorphic
to the $\Val^{\U(n)}$-module $\Curv^{\U(n)}$, with the property that if $M' \hookrightarrow M$ is a K\"ahler embedding then the natural 
restriction map $\calC(M) \to \calC(M')$  restricts to a natural surjection $\widetilde{\KLK}(M) \to \widetilde{\KLK}(M')$ such that the following diagram commutes:
	\begin{center}
		\begin{tikzpicture}
			\matrix (m) [matrix of math nodes,row sep=3em,column sep=4em,minimum width=2em]
			{
				\Curv^{\U(n)}  &  \Curv^{\U(n')}\\
				\widetilde{KLK}(M)   &  \widetilde{KLK}(M')\\};
			\path[-stealth]
			(m-1-1) edge node [right] {$\cong$}  (m-2-1)
			(m-1-1) edge (m-1-2)
			(m-1-2) edge node [right] {$\cong$}  (m-2-2) 
			(m-2-1) edge (m-2-2);
		\end{tikzpicture}
	\end{center}
	where the upper horizontal map is the restriction map induced by the inclusion $\C^{n'} \hookrightarrow \C^n$.
	
%
\end{maintheorem}
 
Applied to complex projective and complex hyperbolic space the theorem yields
\begin{corollary}
There is a canonical module isomorphism between the modules of invariant curvature measures on complex space forms of the same dimension, compatible with the canonical isomorphism from Corollary \ref{cor_canonical_isomorphism}. 
\end{corollary}

This is a new insight in hermitian integral geometry and a nice extra property of the canonical valuation isomorphism. It follows from these results that, quite unexpectedly, there are isomorphisms for all integral geometric structures on complex space forms: global, semi-local and local kinematic formulas; algebra structure on valuations, module structure on curvature measures.

\subsection{Our approach}

Our strategy for proving Theorems \ref{mainthm_global} and \ref{mainthm_local} is  a hybrid of two different approaches to the construction of the Lipschitz-Killing algebra in the riemannian case.  In the first,
the Nash theorem is used to embed a given riemannian manifold $M$ in some $\RR^N$. The valuations 
giving the tube coefficients are then restrictions of the intrinsic volumes $\mu_i$ of $\RR^N$, which by Weyl's 
theorem depend only on the riemannian structure of $M$. The resulting restrictions $\mu_i^M \in \calV(M)$ then span $\mathcal {LK}(M)$: fundamental results
of Alesker show that this span is closed under Alesker multiplication. 

That approach has the advantage of being simple and direct. On the other hand, it fails to show how the $\mu_i^M $ and their algebraic properties arise from the riemannian structure of $M$. A second approach, remedying this defect, was given in \cite{fu_wannerer}, in which the $\mu_i^M $ 
are constructed from canonical differential forms associated to the connection and the curvature tensor of $M$.
Certain combinations of these forms were shown to be subject to an algebraic model of the fundamental formula of \cite{alesker_bernig} for the Alesker product. Using invariant theory, the possible outcomes
were shown to be highly constrained, so that the actual values could be computed from a few simple templates.

Attempting to extend these approaches to  the K\"ahler case we consider here, we meet two obstacles: i) there can be 
no K\"ahler analogue of the Nash theorem, so the first approach is impossible; and ii) the relevant space of invariant differential forms is large and unwieldy, to the point where we were unable to mimic the second 
approach above in the new setting. Happily, however, we are able to combine weak versions of these two devices to bring the proof
to fruition. First, via the formalism of double forms we construct algebraic models for the differential forms 
underlying the $\mu_{k,p}^M$. Although our understanding of the invariant theory is insufficient to conclude
that the Alesker product of two such valuations lies in their span, we are nonetheless able to deduce that 
the underlying differential form of such a product depends polynomially on the curvature tensor $\dblR$ of $M$ (Proposition
\ref{prop:polynomiality prime}). Second, we point out (Theorem \ref{thm_pointwise_lemma}) that we do have 
an extremely weak K\"ahler isometric embedding theorem: for $\dblR$ lying in some open subset of the space of algebraic curvature 
tensors at a single point $x \in M$, there is a local embedding of $M$ into some $\CC^N$ for which the curvature 
tensor at $x$ of the K\"ahler metric induced on $M$ by that of $\CC^N$ is equal to $\dblR$. Since the  $\mu_{k,p}^M$
behave well under restrictions (this will be shown in Subsection \ref{subsec_restrictions}), the polynomial dependence on $\dblR$
of the Alesker product of hermitian intrinsic volumes implies that the multiplicative structure is as claimed.

To conclude this introduction, let us point out the importance of the formalism of double forms--- introduced originally by de Rham \cite{deRham} and subsequently used to brilliant effect in Gray's work on integral geometry \cite{gray69}, \cite{gray_book}---  in this paper. Section~\ref{sec_double_forma_main} clarifies the construction of the canonical curvature measures of \cite {fu_wannerer} in these terms. In fact they are polynomials in the 
canonical $(1,1)$ forms $\dblg$ (the metric) and $\bm \omega$ (the connection), as well as the 
$(2,2)$-form $\dblR$ representing curvature. Section~\ref{sec_hermitian_double}
 then  constructs similar objects in the K\"ahler case, where the resulting algebraic structure is much richer, due
 to the appearance of the  canonical (2,2)-form $\dblG$ and the additional symmetries enjoyed by K\"ahler
 curvature tensors $\dblR$. As in \cite{fu_wannerer}, the fact (Proposition \ref{prop_rumin_in_cartan_calculus}) that the Rumin differential of
 these elements is again expressible in this language is key.



\section{Background on valuations and curvature measures}\label{section_background}

In this section we give a brief account of the
 theory of \emph{valuations on smooth manifolds} $M$, introduced by Alesker \cites{alesker_val_man1, alesker_val_man2, alesker_val_man4, alesker_val_man3}, in the case (sufficient to the purposes of the present paper) where $M$ is an oriented riemannian manifold. 
Such valuations are  functions defined for sufficiently smooth subsets $A \subset M$, determined by (pairs of) differential forms on the tangent sphere bundle $SM$ in two different ways. The first way is more direct and intuitive,
as integrals over the normal cycle of $A$; the second way is more obscure but has the advantage of being unique. The two 
are formally related via the Rumin differential, which will play a key role.  The
main technical point is Proposition \ref{prop_double_fibration}, which gives a universal formal model for the Alesker
product of two valuations in the latter terms. After that we recap some of the main features of  invariant valuations 
on riemannian manifolds in general and  complex space forms in particular.

\subsection{Valuations on manifolds}

Given an oriented riemannian manifold $M$ of dimension $n$, we let $SM$ be its sphere bundle. { The contact form on $SM$ is denoted by $\alpha$.} A sufficiently regular set  $X \subset M$ (e.g. a compact differentiable polyhedron \cite{alesker_intgeo}, or a compact subset with positive reach \cite{federer59}, or a compact subanalytic set \cite{fu94}, { or a compact WDC set \cites{pokorny_rataj,fu_pokorny_rataj}}) determines a {\it normal cycle} $\nc(X)$, which is a current in $SM$. The normal cycle is finitely additive for such subsets $X,Y\subset M$  in  general position:
\begin{equation}\label{eq:nc additive}
\nc(X\cup Y ) = \nc(X ) + \nc(Y) -\nc(X\cap Y).
\end{equation}

A \emph{smooth valuation on $M$} is a functional of the form 
\begin{equation} \label{eq_def_smooth_val}
\mu(X)=\int_X \phi+\int_{\nc(X)} \omega, 
\end{equation}
where $\phi \in \Omega^n(M)$ { is the {\it inner term} and $\omega \in \Omega^{n-1}(SM)$ is the {\it boundary term}}. We denote it  by $\mu=[[\phi,\omega]]$. 
A \emph{smooth curvature measure on $M$} is a functional of the form
\begin{equation}\label{eq_def_smooth_cm}
\Phi(X,B)=\int_{X \cap B} \phi+\int_{\nc(X) \cap \pi^{-1}B} \omega,
\end{equation}
for $B \subset M$ a Borel subset. Such a curvature measure is written as $[\phi,\omega]$. These objects
inherit finite additivity from \eqref{eq:nc additive}.

The space of smooth valuations is denoted by $\mathcal{V}(M)$, the space of smooth curvature measures by $\mathcal{C}(M)$. The \emph{globalization map} $\glob\colon\mathcal{C}(M) \to \mathcal V(M)$ is the map given by $\glob(\Phi)= [\Phi]:=\Phi(\bullet,M)$.  The space $\mathcal{V}(M)$ admits the structure of an algebra \cites{alesker_bernig,alesker_val_man3}, while $\mathcal{C}(M)$ is a module over $\mathcal{V}(M)$ \cite{bernig_fu_solanes}.  If $e: M \subset N$ is an  embedding of riemannian manifolds, there are obvious restriction maps $e_{\calV}^*: \mathcal{V}(N) \to \mathcal{V}(M), e_{\calC}^*:\mathcal{C}(N) \to \mathcal{C}(M)$ (cf. \cite{alesker_intgeo,bernig_fu_solanes}). These maps respect the multiplicative structure:
$$
e_{\calV}^*(\mu \cdot \nu) = e_{\calV}^*\mu \cdot e_{\calV}^*\nu, \quad e_{\calV}^*\mu \cdot e_{\calC}^*\Phi = e_{\calC}^*(\mu \cdot \Phi).
$$

The kernel of the map $(\phi,\omega) \mapsto [[\phi,\omega]]$ was described in \cite{bernig_broecker07} as the space of all pairs $(\phi,\omega)$ such that $D\omega+\pi^*\phi=0,\pi_*\omega=0$. Here $D$ is the Rumin operator \cite{rumin94} and $\pi^*,\pi_*$ are pull-back and push-forward of differential forms with respect to the projection $\pi:SM \to M$. The kernel of the map $(\phi,\omega) \mapsto [\phi,\omega]$ is given by the  pairs $(0,\omega)$ where $\omega$ belongs to the ideal generated by $\alpha$ and $d\alpha$, see \cite[Proposition 2.4]{bernig_faifman_solanes_part2}.
Set 
\begin{equation}
\label{eq:zeta_tau}
\zeta:=\pi_*\omega \in C^\infty(M),\quad  \tau:=D\omega+\pi^*\phi \in \Omega^n(SM).
\end{equation}
It is easy to check that $\alpha \wedge \tau=0, d\tau=0, \pi_*\tau=d\zeta$. Conversely, if $(\zeta,\tau)$ is a pair with $\zeta \in C^\infty(M), \tau \in \Omega^n(SM)$ such that $\alpha \wedge \tau=0, d\tau=0, \pi_*\tau=d\zeta$, then there exist $\phi,\omega$ { satisfying \eqref{eq:zeta_tau}}. We may therefore describe smooth valuations on $M$ uniquely by such pairs $(\zeta,\tau)$. In this case we write 
\begin{equation}\label{eq:def_tau_zeta}
\mu=\{\{\zeta,\tau\}\},\quad \mbox{and}\quad \zeta=\zeta(\mu),\quad \tau=\tau(\mu).
\end{equation}

The product of smooth valuations admits a description in these terms \cite{alesker_bernig}*{Eqs. (48),(49)}, which
for our purposes may be distilled as follows. Fix $x \in M$. Then for
each $\xi \in S_xM$, the Levi-Civita connection induces a decomposition 
\begin{equation}\label{eq:decompHV}
T_\xi SM =  H_\xi \oplus T_\xi (S_xM) 
\end{equation}
into horizontal and  vertical subspaces, where the { derivative of the } projection map $\pi:SM \to M$ { restricts to an} isomorphism
$ H_\xi\simeq T_xM$.  These decompositions in turn induce an isomorphism
\begin{equation}\label{eq:transfer_forms}
\restrict{\Omega^\bullet (SM)}{S_xM} \simeq  \bw^\bullet T_x^*M  \otimes\Omega^\bullet(S_xM). 
\end{equation}

Similarly,  for a fixed $\xi \in S_xM$,
\begin{equation}
\bw^\bullet T_\xi^* SM \simeq \bw^\bullet (T_xM \oplus  \xi^\perp )^*,
\end{equation}
where $T_\xi (S_xM )\simeq \xi^\perp  \subset T_xM$. Thus,  distinguishing a unit vector $e_0 \in \RR^n$, and an isometry $\sigma:\RR^n \to T_xM$ with 
$\sigma e_0 = \xi$ yields 
 isomorphisms
\begin{align}
\bar\sigma^*&\colon\restrict{\Omega^\bullet (SM)}{S_xM} \to  \bw^\bullet (\RR^n)^*  \otimes  \Omega^\bullet(S^{n-1})\label{eq:sigma_x}\\
\bar\sigma_\xi^*&\colon \bw^\bullet T_\xi^* SM \to \bw^\bullet (\RR^n \oplus e_0^\perp  )^*.\label{eq:sigma_xi}
\end{align}

\begin{proposition}\label{prop_double_fibration}
 For any isometry $\sigma\colon \RR^n\to T_xM$ with $\sigma e_0 = \xi$  as above, there exist bilinear maps
\begin{displaymath}
\Gamma: \left[ \bw^\bullet (\RR^n)^*  \otimes \Omega^\bullet (S^{n-1})\right]^2 \to \bw^{\bullet} (\RR^n\oplus e_0^\perp )^*
\end{displaymath}
and 
\begin{displaymath}
\mathrm{GT} \colon \Omega^\bullet (SM) ^2 \to \Omega^\bullet (SM),
\end{displaymath}
making the diagram 
\begin{center}\begin{tikzpicture}
	\matrix (m) [matrix of math nodes,row sep=3em,column sep=3em,minimum width=2em]
	{
		\left[\restrict{\Omega^\bullet (SM)}{S_xM}\right]^2    & \restrict{\Omega^\bullet (SM)}{S_xM}   & \bw^\bullet  T^*_\xi SM \\
	      \left[ \bw^\bullet (\RR^n)^*  \otimes \Omega^\bullet (S^{n-1})\right]^2  &  & \bw^\bullet (\RR^n \oplus e_0^\perp )^*  \\ };
	\path[-stealth]
	(m-1-1) edge node [right] {$\bar \sigma^* \oplus \bar \sigma^*$}  (m-2-1)
	(m-1-1) edge node [above] {$\mathrm{GT}$} (m-1-2)
	(m-1-2) edge node [above] {$|_\xi$} (m-1-3)
	(m-2-1) edge node [above] {$\Gamma$} (m-2-3)
	(m-1-3) edge node [right] {$\bar \sigma_\xi^*$}  (m-2-3);
	\end{tikzpicture}
\end{center}
commutative and such  that the following hold:
\begin{enumerate}
	\item  If $\mu_i=\{\{\zeta_i,\tau_i\}\}, i=1,2$ are smooth valuations, then $\mu_1 \cdot \mu_2=\{\{\zeta,\tau\}\}$ with  
	\begin{align}
	\zeta & = \zeta_1 \cdot \zeta_2, \label{eq_ab_formula2}\\
	\tau & = \mathrm{GT}(\tau_1 , \tau_2)+\pi^*\zeta_1 \cdot \tau_2+\pi^*\zeta_2 \cdot \tau_1. \label{eq_ab_formula1}
	\end{align}
	\item 
If $\mu=\{\{ \zeta, \tau\}\}$ is a smooth valuation and $[\phi,\omega]$ is a curvature measure, 
then $\mu\cdot [\phi,\omega]= [\phi',\omega']$ with
\begin{align*}
\phi' & =\pi_*(\omega \wedge s^*\tau) +  \zeta \cdot \phi, \\
\omega' & = \mathrm{GT}(\omega , \tau)+\pi^*\zeta \cdot \omega,
\end{align*}
where $s\colon SM\to SM$ denotes the antipodal map.
\end{enumerate}
\end{proposition}

 The notation refers to the fact that  $\mathrm{GT}$ is a Gelfand transform with respect to a certain double fibration, described in the proof below.

\begin{proof}
We use the description of the Alesker-Bernig formula given in \cite{fu_alesker_product}*{Theorem 5.2} together with the notation introduced there.
In particular, as in  \cite{fu_alesker_product}*{Section 5} consider the fiber bundle   $\Sigma \to M$ with the projections $\xi,\eta,\zeta\colon \Sigma\to SM$, given by
\begin{displaymath}
\Sigma:= \{(\xi,\eta,\zeta)\in SM^3: \pi\xi =\pi\eta =\pi\zeta, \xi\neq \pm \eta, \zeta \text{ lies in the open geodesic segment } \overline{\xi\eta}\}.
\end{displaymath}
At each point $(\xi,\eta,\zeta)$ of $\Sigma$ the Levi-Civita connection of $M$ defines a splitting 
\begin{equation} \label{eq:splittingSigma} 
T_{\xi,\eta,\zeta} \Sigma = H^\Sigma \oplus V^\Sigma
\end{equation}
into horizontal and vertical tangent vectors,  where $H^\Sigma \simeq T_xM$ and $x\in M$ is the common projection of $\xi,\eta,\zeta$. The vertical subspace splits  further   as follows. First note  that the riemannian metric of $M$ induces an inner product on $V^\Sigma$.   Let  $\Sigma_\zeta\subset \Sigma_x$ denote a fiber of $\zeta\colon \Sigma\to SM$. 
Thus $V^\Sigma=   V' \oplus V''$, where $V''$ is the tangent space to the fiber $\Sigma_\zeta$ and $V'\simeq T_\zeta (S_xM)$ is its orthogonal complement inside $V^\Sigma$. 

Recall that also the tangent spaces to $SM$ split into horizontal and vertical vectors, 
\begin{displaymath}
T_\zeta SM = H\oplus V.
\end{displaymath} 
 Since the differential $\zeta_*\colon V^\Sigma \to V$ is onto with kernel $V''$, $\zeta_*:V'\to V$ is an isomorphism.

Consider  the model fiber 
\begin{displaymath}
\Sigma_0 = \{ (\xi_0,\eta_0,\zeta_0)\in (S^{n-1})^3 \colon \xi_0\neq \pm \eta_0 \text{ and } \zeta_0 \text{ lies in the open geodesic segment } \overline{\xi_0\eta_0}\}
\end{displaymath}
with the projections $\xi_0,\eta_0\colon \Sigma_0\to S^{n-1}$. Fixing $\zeta_0=e_0 \in \RR^n$, the tangent space to  $\Sigma_{0}$ splits 
as $ T_{\xi_0,\eta_0,e_0} \Sigma_{0} \simeq V'_0 \oplus V''_0$, where $V'_0 \simeq e_0^\perp= T_{e_0}S^{n-1}$ and $V_0'' \simeq T_{\xi_0,\eta_0} \Sigma_{0,e_0}$, mirroring the splitting above. For any isometry $\sigma\colon \RR^n\to T_xM$ with $\sigma e_0 = \zeta$ , we obtain from \eqref{eq:splittingSigma} induced isomorphisms
 \begin{gather*}   T_{\xi,\eta,\zeta} \Sigma \simeq \RR^n \oplus T_{\sigma^{-1} \xi, \sigma^{-1} \eta ,e_0} \Sigma_0 \simeq  \RR^n\oplus e_0^\perp\oplus V''_0\\
T_{\xi} SM \simeq \RR^n  \oplus T_{\sigma^{-1} \xi} S^{n-1},
\end{gather*}
both of which we denote by $\sigma_*$. Consequently, the following diagram commutes
\begin{center}\begin{tikzpicture}
  \matrix (m) [matrix of math nodes,row sep=3em,column sep=4em,minimum width=2em]
{
	T_{\xi,\eta,\zeta} \Sigma  &  \RR^n \oplus T_{\sigma^{-1} \xi, \sigma^{-1} \eta ,e_0} \Sigma_0 \\
	T_{\xi} SM   &\RR^n \otimes T_{\sigma^{-1} \xi} S^{n-1}\\};
\path[-stealth]
(m-1-1) edge node [right] {$\xi_*$}  (m-2-1)
edge node [above] {$\sigma_*$} (m-1-2)
(m-1-2) edge node [right] {$\id_{\RR^n}  \oplus \xi_{0*}$}  (m-2-2) 
(m-2-1) edge node [above] {$\sigma_*$}(m-2-2);
\end{tikzpicture}
\end{center}
 yielding a simple description of the differential $\xi_*$ in terms of the model space.
Likewise, the  projection to top-degree forms on the fiber $\Sigma_\zeta$ admits a description in terms of the model space
\begin{center}\begin{tikzpicture}
	\matrix (m) [matrix of math nodes,row sep=3em,column sep=4em,minimum width=2em]
	{
		\bw^\bullet T^*_{\xi,\eta,\zeta} \Sigma  &  \bw^\bullet (\RR^n \oplus T_{\sigma^{-1} \xi, \sigma^{-1} \eta ,e_0} \Sigma_0)^* \\
		\bw^\bullet (H\oplus V)^* \otimes \bigwedge^n (V'')^*    &\bw^\bullet (\RR^n \oplus e_0^\perp )^* \otimes\bigwedge^n   T^*_{\sigma^{-1} \xi,\sigma^{-1}\eta} \Sigma_{0,e_0}\\};
	\path[-stealth]
	(m-1-1) edge   (m-2-1)
	edge node [above] {$\sigma_*$} (m-1-2)
	(m-1-2) edge  (m-2-2) 
	(m-2-1) edge node [above] {$\sigma_*$}(m-2-2);
	\end{tikzpicture}
\end{center}

Combining these two operations with integration over the fiber we obtain the commutative diagram 
\begin{center}\begin{tikzpicture}
	\matrix (m) [matrix of math nodes,row sep=3em,column sep=4em,minimum width=2em]
	{
		(\restrict{\Omega^\bullet (SM)}{S_xM} )^{\otimes 2}  &  ( \bw^\bullet (\RR^n)^* \otimes  \Omega^\bullet(S^{n-1}) )^{\otimes 2}  \\
	  \restrict{\Omega^\bullet(\Sigma)}{\Sigma_\zeta}   & \bw^\bullet (\RR^n)^*  \otimes \restrict{\Omega^\bullet ( \Sigma_0)}{\Sigma_{0,e_0}}\\
	  \bw^\bullet  T^*_\zeta SM   \otimes \Omega^n(\Sigma_\zeta) &   \bw^\bullet (\RR^n \oplus e_0^\perp )^*  \otimes \Omega^n(\Sigma_{0,e_0}) \\
	   \bw^\bullet  T^*_\zeta SM  &  \bw^\bullet (\RR^n \oplus e_0^\perp )^*  \\ };
	\path[-stealth]
	(m-1-1) edge node [right] {$\xi^* \wedge \eta^*$}  (m-2-1)
	edge node [above] {} (m-1-2)
	(m-1-2) edge node [right] {$ (\id_{\RR^n}\times \xi_0)^* \wedge  (\id_{\RR^n}\times\eta_0)^*$}  (m-2-2) 
	(m-2-1) edge (m-2-2)
	(m-2-2) edge (m-3-2)
	(m-3-2) edge node [right] {$\int_{\Sigma_{0,e_0}}$} (m-4-2)
	(m-2-1) edge (m-3-1)
	(m-3-1) edge node [right] {$\int_{\Sigma_\zeta}$} (m-4-1)
	(m-3-1) edge (m-3-2)
	(m-4-1) edge (m-4-2);
	\end{tikzpicture}
\end{center}
where the horizontal arrows are isomorphisms induced by $\sigma$. Composed with $\id \otimes s^*$, where $s$ is the antipodal map, the left column defines $\mathrm{GT}$ and the  right column defines  $\Gamma$.  
\end{proof}

\subsubsection{Translation invariant valuations}
If $M$ is a finite dimensional real vector space $V$, then any finite union of 
 compact convex bodies admits a normal cycle, and any two such finite unions $X,Y$ satisfy \eqref{eq:nc additive}. 
 We denote the space of {\it translation invariant} smooth valuations on $V$ by $\Val=\Val(V)\subset \calV(V)$. In fact it is a graded 
 algebra, with multiplication inherited from $\calV(V)$ and  grading 
\begin{displaymath}
\Val=\bigoplus_{\substack{k=0,\ldots,n\\\epsilon=\pm}} \Val_k^\epsilon,
\end{displaymath}
where $n=\dim V$ and where $\Val_k^\epsilon$ is the subspace of $k$-homogeneous valuations of parity $\epsilon$. (Note that there exists a more primitive, and historically prior, notion of valuation in this context, in which the smoothness condition is relaxed to continuity.)

By $\Curv\subset \calC(V)$ we denote the space of smooth and translation invariant curvature measures (meaning that $\Phi(K+v,B+v)=\Phi(K,B)$ for all $v \in V$). It admits a grading by degree 
\begin{displaymath}
\Curv=\bigoplus_{k=0,\ldots,n} \Curv_k,
\end{displaymath}
where $\deg \Phi = k$ iff $\Phi(tK,tB)=t^k \Phi(K,B), t>0$. $\Curv$ is a module over $\Val$.

A famous theorem of Hadwiger implies that the vector space $\Val^{\mathrm{SO}(n)}$ of $\mathrm{SO}(n)$-invariant elements in $\Val$ is spanned by the intrinsic volumes $\mu_0,\ldots,\mu_n$ (in fact Hadwiger's theorem yields the same conclusion even in the broader context mentioned above, where smoothness is replaced by continuity). More generally, if $G \subset \mathrm{O}(n)$ is a subgroup that acts transitively on the unit sphere $S^{n-1}$, then the subspace of $G$-invariant elements $\Val^G$ is a finite dimensional graded subalgebra of $\Val$.

In the 
 present paper, we are particularly interested in the case $G=\U(m) \subset \mathrm{O}(2m)$. 
 In this case \cite{alesker03_un,fu06} 
 \begin{displaymath}
\Val^{\U(m)} \cong \mathbb{R}[t,s]/(f_{m+1},f_{m+2}),
\end{displaymath}  
where $\log(1+tx+sx^2)=\sum_{k=1}^\infty f_k(t,s)x^k$. The valuation $t$ is  (up to scale) the mean width, or the first intrinsic volume, while $s\in \Val_2$ is (up to scale) the average area of the projections to complex lines. 
In particular,
\begin{displaymath}
\dim \Val_k^{\mathrm{U}(m)}=1+ \min\left(\left\lfloor\frac{k}{2}\right\rfloor,\left\lfloor\frac{2m-k}{2}\right\rfloor\right).
\end{displaymath}

The space of $\U(m)$-invariant and translation invariant curvature measures was studied in \cite{park02}, \cite{bernig_fu_hig} and \cite{bernig_fu_solanes}. It admits a natural basis $\Delta_{k,q},N_{k,q}$ such that $[\Delta_{k,q}]=\mu_{k,q}, [N_{k,q}]=0$; and another natural basis $B_{k,q},\Gamma_{k,q}$ with $[B_{k,q}]=[\Gamma_{k,q}]=\mu_{k,q}$. The dimensions are given by
\begin{displaymath}
\dim \Curv_k^{\mathrm U(m)}=\min(k,2m-k{ -1})+1.
\end{displaymath}

\subsection {Complex space forms}
 Let $(\CP^n_\lambda,G_\lambda)$ denote the complex space form of complex dimension $n$ and holomorphic curvature $4\lambda$ (depending on the sign of $\lambda$ this will be a rescaling of complex hyperbolic space, hermitian space, or complex projective space); the group $G_\lambda$ is the isometry group if $\lambda \neq 0$, or the unitary isometry group $\mathrm{U}(n) \ltimes \mathbb C^n$ if $\lambda=0$. 
The paper \cite{bernig_fu_solanes} displayed several non-canonical isomorphisms
\begin{displaymath}
\mathcal{V}(\CP^n_\lambda)^{G_\lambda} \cong \Val^{\U(n)}
\end{displaymath}
for all $\lambda$. The present paper gives a canonical one.

Relevant bases for $\mathcal{V}(\CP^n_\lambda)^{G_\lambda}$ are the monomials in  certain canonical generators $ t,s$; the valuations $\mu_{k,q}^\lambda$; and the valuations $\tau_{k,q}^\lambda$. We refer to \cite{bernig_fu_solanes} for the definitions.

{ Given a point $x$ in a smooth manifold $M$, we can view \eqref{eq:transfer_forms} as an isomorphism
\begin{equation}\label{eq:bar_tau}
 \bar\tau_x\colon \restrict{\Omega^\bullet(SM)}{S_xM}\to  \bw^\bullet T_x^*M\otimes \Omega^\bullet(ST_xM)\simeq\Omega^\bullet(ST_xM)^{T_xM}
\end{equation}into the space of translation invariant forms of the sphere bundle $ST_xM$. This induces the {\em  transfer map} (cf. \cite{bernig_fu_solanes}*{Prop. 2.5})
\begin{equation}\label{eq_transfer_map}
 \tau_x\colon \mathcal C(M)\to \Curv(T_xM),
\end{equation}
given by $\tau_x([\phi,\omega])=[\phi_x,\bar\tau_x(\omega|_{S_xM})]$.

It follows \cite{bernig_fu_solanes}*{Prop. 2.22}} that there is a canonical identification of $G_\lambda$-invariant smooth curvature measures on $\CP^n_\lambda$ and $\mathrm{Curv}^{\U(n)}$. In particular, the curvature measures $\Delta_{k,q},N_{k,q},B_{k,q},\Gamma_{k,q}$ from the flat case induce curvature measures in the curved case, denoted by the same letters. 
For future reference we  note from \cite{bernig_fu_solanes}
\begin{equation}
[B_{k,p}]=\mu_{k,p}^\lambda,\qquad 
[\Delta_{k,q}]  =\mu_{k,q}^\lambda-\lambda \frac{q+1}{\pi} \mu_{k+2,q+1}^\lambda \label{eq_glob_delta}\\
\end{equation}
as well as
\begin{align}
\tau_{k,q}^\lambda & = \sum_{i=q}^{\left\lfloor \frac{k}{2}\right\rfloor} \binom{i}{q} \mu_{k,i}^\lambda \label{eq_def_tau} \\
\tau_{k,q}^\lambda & = (1-\lambda s) \frac{\pi^k}{\omega_k (k-2q)! (2q)!}v^{\frac{k}{2}-q}u^q,  \label{eq_tau_in_uv}
\end{align}where $v:=t^2(1-\lambda s), u:=4s-v$, and 
\begin{displaymath}
 \omega_n:=\frac{\pi^{n/2}}{\Gamma\left(\frac{n}{2}+1\right)}
\end{displaymath}
is the volume of the $n$-dimensional unit ball. We will also use the notation $s_n=(n+1)\omega_{n+1}$ for the volume of the $n$-dimensional unit sphere and the formulas 
\begin{align}
\omega_{2n}  = \frac{\pi^n}{n!}, \qquad
 \label{eq:volume_odd}\omega_{2n+1} = \frac{2^{2n+1}\pi^n n!}{(2n+1)!}.
\end{align}
Recall also the  recursion relation
\begin{equation}\label{eq:volume_ball_rec}
\omega_{n+2l} = \frac{(2\pi)^l}{(n+2l)(n+2l-2)\cdots (n+2)} \omega_n.
\end{equation}

\subsection{Riemannian valuations and curvature measures}  
Let us now summarize some results from \cite{fu_wannerer}. 

On a riemannian manifold $M$ of dimension $n$, certain natural smooth curvature measures $C_{k,p}, 0 \leq 2p \leq k \leq n$ are constructed within Cartan's calculus, i.e. using solder, connection and curvature forms on the frame bundle. The \emph{Lipschitz-Killing curvature measures} $\Lambda_k,0\leq k\leq n$ are given as special linear combinations of the $C_{k,p}$. We revisit these expressions below, using
	Gray's notion of {\it double forms}. The globalizations of the Lipschitz-Killing curvature measures yield the \emph{intrinsic volumes} $\mu_k=\frac{\pi^k}{\omega_k k!}[\Lambda_k]$, { whose linear span constitutes} an algebra under the Alesker product, called the \emph{Lipschitz-Killing algebra} $\mathcal{LK}(M) \cong \mathbb R[t]/(t^{n+1})$. The main point of this construction is invariance: the action of the Lipschitz-Killing algebra on the curvature measures $C_{k,p}$ is the same for every riemannian manifold. Thus this action may be computed from easy manifolds (like round spheres).

The Lipschitz-Killing curvature measures are characterized among all linear combinations of the $C_{k,p}$ by the property that  for an isometric immersion $e:M \to N$ of riemannian manifolds we have $e_{\mathcal C}^* \Lambda_k^N=\Lambda_k^M$.


\section{Double forms}\label{sec_double_forma_main}

In this section we describe the algebra of double forms, whose utility in integral geometry was first highlighted by Gray \cite{gray_book} (in fact they furnish a convenient formal language for the tensorial aspects of riemannian geometry more generally).
A novel element of our treatment  is the inclusion of double forms corresponding to the boundary
terms  of curvature measures and valuations, which were not considered by Gray. 

\subsection{Double forms associated with a smooth map}
Several of the constructions in this section are taken from \cite{gray69}. 

\subsubsection{Double exterior algebras}
Given finite dimensional real vector spaces $V,W$, we consider the bigraded {\it double exterior algebra}
\begin{displaymath}
\bigwedge\nolimits^{\bullet,\bullet}(W,V):=\bigwedge\nolimits^\bullet W^* \otimes \bigwedge\nolimits^\bullet V^* =\bigoplus_{p,q} \left( \bigwedge\nolimits ^p W^* \right)\otimes \left(\bigwedge\nolimits ^q V^*\right)
\end{displaymath}
with respect to the product 
\begin{displaymath}
(\theta_1 \otimes \psi_1)\wedge (\theta_2 \otimes \psi_2) := (\theta_1 \wedge \theta_2) \otimes(\psi_1\wedge \psi_2).
\end{displaymath}

The elements (double forms) of the summands are said to be of {\it type} $(p,q)$. If $ \omega_1$ is of type $(p_1,q_1)$ and $\bm\omega_2$  of type $(p_2,q_2)$, then  clearly
\begin{displaymath}
 \bm \omega_1 \wedge \bm\omega_2=(-1)^{p_1p_2+q_1q_2}\bm\omega_2 \wedge \bm\omega_1.
 \end{displaymath}
If $\bm\omega$ is of type $(p,q)$ and $X_1,\dots,X_p \in W, \ Y_1,\dots Y_q \in V$ then we denote the evaluation
\begin{displaymath}
\bm\omega( X_1,\dots,X_p ;\ Y_1,\dots Y_q):=\bm \omega( X_1,\dots,X_p )( Y_1,\dots Y_q) \in \RR.
\end{displaymath}

Given linear maps $V\to V', W \to W'$, there is a pullback map
\begin{displaymath}
\bw^{\bullet,\bullet}(W',V')\to \bw^{\bullet,\bullet}(W,V).
\end{displaymath}
{ Given linear maps $l\colon V\to W$ and $m\colon W\to V$}  there is an operation $\bm\omega \mapsto \bm\omega'$ 
{ of degree $(1,-1)$}
\begin{align*}
\bw^{p,q}(V,W) &\to \bw^{p+1,q-1}(V,W) ,\\
\bm\omega'(X_1,\ldots,X_{p+1};Y_2,\ldots,Y_q)&:=\sum_{j=1}^{p+1} (-1)^{j+1}\bm \omega(X_1,\ldots,\widehat{X_j},\ldots,X_{p+1};{ l}X_j,Y_2,\ldots,Y_q)
\end{align*}
and  an operation $\bm\omega \mapsto \bm\omega^\vee$ from forms of type $(p,q)$ to forms of type $(q,p)$:
\begin{align*}
\bw^{p,q}(V,W) &\to \bw^{q,p}(V,W) ,\\
\bm\omega^\vee(X_1,\dots,X_q;Y_1,\dots,Y_p)&: =\bm\omega({ m}Y_1,\dots,{ m}Y_p;{ l}X_1,\dots,{ l}X_q).
\end{align*}
Clearly
\begin{displaymath}
 (\bm\omega \wedge \bm\phi)'=\bm\omega' \wedge \bm\phi+(-1)^{p+q}\bm\omega \wedge \bm\phi'
\end{displaymath}
and 
\begin{displaymath}
 (\bm \omega\wedge \bm \phi)^\vee = \bm\omega^\vee \wedge \bm\phi^\vee.
\end{displaymath}
\subsubsection{Double forms} { Let $M_1, M_2$ be smooth manifolds, and $f: M_1\to M_2$ a smooth 
map. Define the algebra $\Omega^{\bullet,\bullet}(f)$ to be the space of smooth sections  of the bundle $\bigwedge\nolimits^\bullet T^*M_1 \otimes f^*(\bigwedge\nolimits^\bullet T^*M_2) $ over
$M_1$, i.e. with fibers 
\begin{displaymath}
\bigwedge\nolimits^{\bullet,\bullet}(T_xM_1, T_{f(x)}M_2), \quad x \in M_1.
\end{displaymath}
Given another smooth manifold $M_3$ and a smooth map $g:M_3\to M_1$, there is a natural pullback map
\begin{displaymath}
g^*: \Omega^{\bullet,\bullet}(f)\to \Omega^{\bullet,\bullet}(f\circ g)
\end{displaymath}
given explicitly by
\begin{displaymath}
(g^*\bm \omega)_z(Z_1,\dots,Z_k;Y_1,\dots,Y_l) := \omega_{g(z)} (dg(Z_1),\dots,dg(Z_k);Y_1,\dots,Y_l) 
\end{displaymath}
for $\bm \omega \in  \Omega^{k,l}(f), \ Z_1,\dots, Z_k \in T_z M_3, \ Y_1,\dots ,Y_l \in T_{f\circ g(z)} M_2$.

Suppose now that the bundle $TM_2$ is endowed with a connection $\nabla$. Consider the pullback {bundle}  $f^*(TM_2)$ over $M_1$, together with the pullback
connection $f^*\nabla$, characterized by
 \begin{equation} \label{eq_charactezitation_pull_back_connection}
(f^*\nabla)_X f^*Y=f^*(\nabla_{df (X)} Y) ,\qquad X\in T_xM_1                                                                                                                           
\end{equation}
for any vector field $Y$, considered as a section of $TM_2$. Consider also the canonical extension of $f^* \nabla$ to $f^*(\bw^\bullet T^*M_2)$, given by  
\begin{equation}\label{eq_f*nabla}
(( f^*\nabla)_X s) (Z_1,\dots,Z_k)= X(s(Z_1,\dots,Z_k)) -\sum_{j=1}^k s(Z_1,\dots,(f^*\nabla)_{X} Z_j,\dots,Z_k),
\end{equation}
where $s\in \Omega^{0,k}(f)$ (i.e. $s$ is a section of $f^*\bw^k T^*M_2$), and the  $Z_j$ are smooth sections of $f^*TM_2$.

Thus we may consider 
$$\dg s:= (f^*\nabla)s,\qquad s\in \Omega^{0,k}(f)
$$ to be an element of $\Omega^{1,k}(f)$.
Now define the exterior covariant derivative $d_\nabla$ as the linear operator $\Omega^{\bullet,\bullet}(f) \to \Omega^{\bullet,\bullet}(f) $ of degree $(1,0)$ given by
\begin{equation}\label{eq:DG_wedge}
d_\nabla (\omega\otimes s)=d\omega\otimes s +(-1)^p \omega\wedge (f^*\nabla) s,
\end{equation}
where $\omega\in\Omega^p(M_1)$ and $s$ is a section of $f^*{\bigwedge}^\bullet T^*M_2$. 
Explicitly, given $\bm\eta \in \Omega^{p,q}(f)$, and smooth vector fields $X_0,\dots,X_p$ on $M_1$,
\begin{align*}(\dg \bm\eta)(X_0,\ldots, X_p)= \sum_{i=0}^p & (-1)^i (f^*\nabla)_{X_i}( \bm\eta(X_0,\ldots, \widehat{X_i}, \ldots, X_p)) \\
& + \sum_{i<j} (-1)^{i+j} \bm\eta([X_i,X_j], X_0,\ldots, \widehat{X_i},\ldots,\widehat{X_j},\ldots,   X_p),
\end{align*}
where we regard the various $\bm\eta(X_{i_1},\dots,X_{i_p})$ 
as smooth sections of $f^*{\bigwedge}^q T^*M_2$, and the terms of the first sum
may be unpacked as in \eqref{eq_f*nabla}.

\begin{lemma}\label{lem_dnabla}
\begin{enumerate}
\item For double forms $\bm\varphi\in \Omega^{p,q}(f), \bm\theta\in \Omega^{r,s}(f)$:
\begin{equation*}
\dg (\bm\varphi\wedge \bm\theta)= \dg\bm \varphi\wedge \bm\theta +(-1)^p \bm\varphi\wedge \dg\bm \theta.
\end{equation*}
\item If $g:M_3\to M_1$ as above, then  $\dg g^*\bm \eta= g^*\dg \bm \eta$ with respect to the pullback $g^*:\Omega^{\bullet,\bullet}(f)\to  \Omega^{\bullet,\bullet}(g\circ f) $.
\item { If $M_2$ is oriented riemannian and $\nabla$ is the Levi-Civita connection, then 
\begin{equation*}
\dg (\eta \otimes \vol)=d\eta \otimes \vol
\end{equation*}
for any $\eta \in \Omega^\bullet(M_1)$.}
\end{enumerate}
\end{lemma}
\proof
(1)  is straightforward using
\begin{displaymath}
 (f^*\nabla)(s_1\wedge s_2)=(f^*\nabla s_1)\wedge s_2 + s_1\wedge (f^*\nabla s_2).
\end{displaymath}

(2) follows at once from the definitions, and the corresponding property of the usual
exterior derivative.

(3) is immediate
since the 
volume form $\vol = \vol_{M_2} $ is parallel.
\endproof

{\bf Remark.} The derivative $\dg$ is denoted by $D$ in \cite{gray69}, a symbol we reserve for the Rumin differential.

\subsection{Double forms on a riemannian manifold}

We specialize to the case where  $M_2 = M$ is a smooth oriented riemannian manifold of dimension $n$ with Levi-Civita connection $\nabla$, $M_1:= SM$ is its unit tangent  bundle, and $f$ is the projection $\pi:SM\to M$. A {\it spherical double form } on $M$ is an element of
\begin{displaymath}
 \Omega^{\bullet,\bullet}(SM,M):=  \Omega^{\bullet,\bullet}(\pi).
\end{displaymath}
Define similarly the space 
\begin{displaymath}
	\Omega^{\bullet,\bullet}(M) := \Omega^{\bullet,\bullet}(\id_M)
\end{displaymath}
of {\em double forms} on $M$. The pullback $\pi^*$ embeds the latter
space into the former. {  Using the maps $l= d\pi\colon T_\xi SM\to T_xM$ and the horizontal lift $m=(d\pi|_{H_\xi})^{-1}\colon T_xM{\longrightarrow} T_\xi SM$, we define the operations $'$ and ${}^\vee$ on $\Omega^{\bullet,\bullet}(SM,M)$.

The riemannian metric of $M$ is a double form $\dblg\in \Omega^{1,1}(M)$, with $\dblg'=0$. The curvature tensor 
 $\dblR=\dblR^M \in  \Omega^{2,2}(M)$, where as usual
\begin{displaymath}
 \dblR(W,X;Y,Z)=\dblg(\nabla_{[W,X]}Y-\nabla_W\nabla_X Y+\nabla_X\nabla_W Y;Z).
\end{displaymath} 
The first Bianchi identity may be stated as
\begin{displaymath}
\dblR' =0. 
\end{displaymath}

Given $X \in T_\xi SM$, let 
$$X = X^H + X^V
$$
 be its decomposition into horizontal 
and vertical components with respect to the Levi-Civita connection. Identifying $X^V \in T_\xi(S_{\pi \xi} M)$ with the corresponding vector in $T_{\pi \xi }M$, we define the {\em connection double form} $\bm\omega \in
\Omega^{1,1}(SM,M)$  by 
\begin{displaymath}
 \bm\omega(X;Y)=\dblg( X^V;Y)= \dblg( (\pi^*\nabla)_X \xi; Y),
\end{displaymath}
where in the last expression $\xi$ denotes the tautological section of the bundle $\pi^* TM$.

Similarly,  the canonical contact form 
\begin{displaymath}
 \bm\alpha(X)=\dblg(X^H;\xi)
\end{displaymath}
defines an element in $\Omega^{1,0}(SM,M)$.

We also define  spherical double forms $\dblR_0,\widehat\dblR$ of respective types $(2,1),(2,2)$ by
\begin{equation}\label{eq_def_R0}
(\dblR_0)_\xi(X_1,X_2;Y)=\dblR_x({d\pi}X_1,{d\pi}X_2;Y,\xi),\qquad\widehat\dblR=\dblR+ \bm\alpha^\vee \wedge\dblR_0.
\end{equation}
 In particular $\widehat\dblR_\xi(\cdot,\cdot;\cdot,\xi)=0$.
 It is easily checked that 
\begin{equation}\label{eq_primes}
\quad (\bm\alpha^\vee)' =\bm \alpha, \quad  \dblR_0'  =0, \quad\widehat  \dblR' {= \bm \alpha \wedge  \dblR_0},
\end{equation}
and
\begin{equation} \label{eq_rvee_prime}
( \dblR_0^\vee)'= \dblR_0.
\end{equation}

\medskip We  will make frequent use of the natural identification 
\begin{displaymath}
\Omega^\bullet(SM)\simeq\Omega^{\bullet,0}(SM,M)\quad{\rm via}\quad
\phi\leftrightarrow\phi \otimes 1.
\end{displaymath}
Prominently, $\bm \alpha$ may be thought of as an element of either space.

It will be useful also to identify $\Omega^{\bullet,n}(SM,M)$ with the space of differential forms on $ SM$ with values in the orientation line bundle $\Omega^n(M)$.  In the main thrust  of the present paper, $M$ will be a K\"ahler manifold,
hence endowed with a canonical volume form, so that 
 \begin{equation}\label{eq_omegan}
 \Omega^{\bullet}(SM)\simeq \Omega^{\bullet,n}(SM) \quad {\rm via} \quad \phi\leftrightarrow\phi \otimes \vol_M.
 \end{equation}
 Typically we will make this identification only after the formal calculations in the algebra of
double forms have taken place, and  we will alert the reader on these occasions in order to minimize confusion.
We remark, however, that integration over the normal cycle $\nc(A)$ of a subset  $A\subset M$, which will be a central construction for us, has the following property:
if $M$  is riemannian, but not oriented (or even orientable), one may  pair $\nc(A)$ with elements of $\Omega^{n-1,n}(SM)$ but not with elements of $\Omega^{n-1}(SM)$.

Of course all of the above canonical double forms   may be expressed in terms of the classical language of moving frames. For $\eta $ lying in an open set $U\subset SM$, let $e_0= \eta,\dots,e_{n-1}$ be a orthonormal frame for $T_{\pi \eta} M$.  In particular, each $e_j$ defines a section of $\pi^* TM$, which we suppose smooth.  Define $\theta_i, \omega_{i,j} \in\Omega^1(U)$, and $\Omega_{i,j}\in \Omega^2(U)$, by 
\begin{align}\label{eq_def_solder}
 \theta_i(X)&:= \dblg(e_i;d \pi X),\\
 \omega_{i,j}(X)&:= \dblg(e_i; (\pi^*\nabla)_Xe_j),\label{eq_def_connection}\\
  \Omega_{i,j}(X,Y)&:= \dblR(e_i,e_j;d\pi X,d\pi Y),
\end{align}
for $X,Y \in T_{\eta}SM$. The corresponding structure equations are then 
\begin{align}
\label{eq:structure_equation1}
d\theta_i&=-\sum_{j=0}^{n-1} \omega_{i,j} \wedge \theta_j, \quad \omega_{j,i}=-\omega_{i,j},\\
 d\omega_{i,j} & =-\sum_{k=0}^{n-1} \omega_{i,k} \wedge \omega_{k,j}+ \Omega_{i,j}, \quad \Omega_{j,i}=-\Omega_{i,j} \label{eq:structure_equation2}\\
 d\Omega_{i,j} & =\sum_{k=0}^{n-1} (\Omega_{i,k} \wedge \omega_{k,j}-\omega_{i,k} \wedge \Omega_{k,j}). \label{eq:structure_equation3}
\end{align}

 Note that the $\theta_i$ may also be viewed as local sections of $\pi^* T^*M$. In these terms,
\begin{align*}
  \dblg & = \sum_{i=0}^{n-1} \theta_i  \otimes \theta_i,    &&   \dblR = \frac12 \sum_{i,j=0}^{n-1} \Omega_{i,j}  \otimes (\theta_i \wedge \theta_j),   \\
  \bm\omega & = \sum_{i=0}^{n-1} \omega_{i,0} \otimes  \theta_i,     &&  \dblR_0 = \sum_{i=0}^{n-1} \Omega_{i,0}  \otimes  \theta_i,  \\
\bm\alpha & = \theta_0\otimes 1, &&   \widehat  \dblR = \frac12 \sum_{i,j=1}^{n-1} \Omega_{i,j} \otimes( \theta_i \wedge \theta_j).
\end{align*}

\subsection{Contractions}

Gray \cite{gray_book} introduced  a {\it contraction operator} $C$ on the space of double forms, 
\begin{displaymath} 
C(\bm \eta) (X_1,\ldots, X_{p-1};Y_1,\ldots, Y_{q-1})= \sum_{i=1}^n \bm \eta( X_1,\ldots, X_{p- 1}, e_i;Y_1,\ldots, Y_{q- 1}, e_i)
\end{displaymath}
for $\bm \eta $ of type $(p,q)$ on a euclidean space  $V$,
where $e_1,\ldots,e_n$  is any orthonormal basis of $V$. We need the following simple lemma. 

\begin{lemma}\label{lemma:contraction}
	If $\bm \eta$ is a double form of type $(p,p)$ on a euclidean space 
	$V$ of dimension $n$, then 
	\begin{displaymath}
	\frac{(n-p)!}{p!} C^p(\bm\eta) = (\bm\eta \wedge \dblg^{n-p}) (e_1,\ldots, e_n;e_1,\ldots, e_n)
	\end{displaymath}
	for any orthonormal basis  $e_1,\ldots, e_n$ of $V$.
\end{lemma}

\begin{proof}
From the definition of the wedge product of double forms it follows that 
\begin{align*}
	(\bm\eta  & \wedge \dblg^{n-p}) (e_1,\ldots, e_n;e_1,\ldots, e_n)\\
	& =\sum_{\rho,\sigma\in \mathcal{S}_{p,n-p}}  \varepsilon_\rho \varepsilon_\sigma\bm \eta(e_{\rho_1}, \ldots, e_{\rho_p};e_{\sigma_1}, \ldots, e_{\sigma_p}) \dblg^{n-p}(e_{\rho_{p+1}}, \ldots, e_{\rho_{n}};e_{\sigma_{p+1}}, \ldots, e_{\sigma_n}), 
\end{align*} 
where $\epsilon_\rho$ is the sign of the permutation $\rho$ and $\mathcal S_{p,n-p}$ denotes the set of $(p,n-p)$ shuffles, i.e. permutations of $\{1,\dots,n\}$ that are monotonically increasing on $\{1,\dots,p\}$ and  on $\{p+1,\dots,n\}$.
Now 
	\begin{displaymath}
	\dblg^{n-p}(e_{\rho_{p+1}}, \ldots, e_{\rho_{n}};e_{\sigma_{p+1}}, \ldots, e_{\sigma_n}) 
	= \begin{cases}
	(n-p)! & \text{if }  \{\rho_{p+1},\ldots,\rho_n\}=\{\sigma_{p+1},\ldots,\sigma_n\};\\
	0 & \text{otherwise}.
	\end{cases} 
	\end{displaymath}
	
Hence 
	\begin{align*}
	(\bm\eta \wedge \dblg^{n-p})  (e_1,\ldots, e_n;e_1,\ldots, e_n)   & =(n-p)!\sum_{1\leq i_1<\cdots< i_p\leq n } \bm\eta(e_{i_1}, \ldots, e_{i_p};e_{i_1}, \ldots, e_{i_p})\\
	& =\frac{(n-p)!}{p!}\sum_{i_1,\ldots, i_p=1 }^n \bm\eta(e_{i_1}, \ldots, e_{i_p};e_{i_1}, \ldots, e_{i_p})\\
	& = \frac{(n-p)!}{p!} C^p(\bm\eta).
	\end{align*}
\end{proof}

\subsection{Exterior derivatives of riemannian double forms}

\begin{lemma}  Let $\theta_i,\omega_{i,j}$ be defined  by \eqref{eq_def_solder} and \eqref{eq_def_connection}.
Then 
\begin{equation}\label{eq:d_solder}
d_\nabla (1\otimes \theta_i) = - \sum_{j=0}^{n-1}\omega_{i,j} \otimes \theta_j.
 \end{equation}
\end{lemma}

\begin{proof}

From the defining relations \eqref{eq:DG_wedge}, \eqref{eq_f*nabla}, \eqref{eq_def_connection}
\begin{align*}d_\nabla (1\otimes\theta_i) (X;Y)&= (\pi^*\nabla)_X \theta_i(Y)\\
&=X(\theta_i(Y))-\theta_i((\pi^*\nabla)_X Y)\\
&=X(\theta_i(Y))-\sum_j \theta_i\big(X(\theta_j(Y))e_j +\theta_j(Y)(\pi^*\nabla)_X e_j\big)\\
&=-\sum_j\theta_j(Y)\omega_{i,j}(X),
\end{align*}
as stated.\end{proof}

\begin{proposition} \label{prop_diff_basic_forms}  The exterior derivatives of $\dblg,\dblR\in \Omega^{\bullet,\bullet}(M)$ and $\bm\alpha^\vee,\bm\omega\in \Omega^{\bullet,\bullet}(SM,M)$ are
	\begin{equation} \label{eq_gray_diff_basic_forms}
	\dg \dblg=0, \quad \dg \dblR=0, \quad \dg \bm\alpha = \bm \omega', 
	\quad \dg \bm\alpha^\vee=\bm\omega, \quad \dg \bm\omega=\dblR_0.
\end{equation}
\end{proposition}
\begin{proof}
	Straightforward.
	Let us prove for instance the last one:
	\begin{align*}
	\dg\bm\omega
	& =\sum_{i=0}^{n-1}d\omega_{i,0}\otimes  \theta_i -\omega_{i,0}\wedge d_\nabla(1\otimes \theta_i)\\
	&=-\sum_{i,j=0}^{n-1}\omega_{i,j}\wedge \omega_{j,0}\otimes  \theta_i +\sum_{i=0}^{n-1}\Omega_{i,0}  \otimes \theta_i +\sum_{i,j=0}^{n-1}(\omega_{i,0}\wedge \omega_{i,j})  \otimes \theta_i\\
	&=\sum_{i=0}^{n-1}\Omega_{i,0}\otimes \theta_i\\
	&  =	\dblR_0.
	\end{align*}
\end{proof}

\subsection{Lipschitz-Killing curvature measures}  
Let us now express in terms of double forms the Lipschitz-Killing curvature measures $\Lambda_k$  defined in \cite{fu_wannerer} and \cite{bernig_fu_solanes}.
Recall that these objects associate elements $\Lambda_k^M \in \calC(M)$ to any smooth riemannian manifold $M$,
with the property that if $N \subset M$ is a smooth submanifold of dimension $k$ then $\Lambda_k^M(N,\cdot)=$ the restriction of $k$-dimensional volume to $N$.

Given  a riemannian manifold $M^n$,  define elements of $\Omega^{ n-k-1,n-k-1}(SM,M)$ by
\begin{align*}
 \bm \psi_{-1}:=&0=:\tilde{ \bm \psi}_{-1},\\
\bm \psi_{n-k-1}& = \sum_{l=0}^{\lfloor \frac{n-k-1}{2}\rfloor}  d_{n,k,l} \dblR^l \wedge\bm  \omega^{n-2l-k-1} \\
\tilde {\bm \psi}_{n-k-1} & =  \sum_{l=0}^{\lfloor \frac{n-k-1}{2}\rfloor} (n-2l-k)d_{n,k,l} \dblR^l \wedge \bm  \omega^{n-2l-k-1}, \quad  k=0,\ldots,n-1,
\end{align*}
where 
\begin{align} \label{eq:def_c_tilde}
 d_{n,k,l} &= \frac{\pi^k}{(k+1)!\omega_k} \frac{1}{2^l} \binom{\frac{k}{2}+l}{l} \frac{\omega_{k+2l}}{\pi^{k+2l}(n-2l-k)! \omega_{n-2l-k}}
 \\
 \notag &= \left[(2\pi)^l l! (k+1)! (n-2l-k)! \omega_{n-2l-k}\right]^{-1}.
\end{align}
These constants will be taken to be zero whenever any index $n,k,l$ is a half-integer.

The spherical double form $\tilde{\bm\psi}_{n-k-1}$ will be used later in Section \ref{sec_hermitian_double}.
For later use we record the relation 
\begin{equation} \label{eq_recursion_tilde_c}
 (n-2l-k-1)d_{n,k,l}=(l+1)d_{n,k,l+1}
\end{equation}
and its immediate consequence
\begin{displaymath}
 d_{n,k,l}=\frac{(n-k-1)!!}{ l!(n-2l-k-1)!!} d_{n,k,0}, \quad  d_{n,k,0} =( (k+1)! (n-k)!\omega_{n-k})\inv,
\end{displaymath}
where as usual $n!!=n(n-2)!!$, ${ (-1)}!!=1$, and $0!!=1$. 

Also for $m<n$
\begin{equation}\label{eq_d_relation_1}
d_{2n,k,l} \frac 1{(2\pi)^e e!} \frac {\omega_{2n-2l-k} (2n-2l-k)!}{\omega_{2m-2l-2e-k} (2m-2l-2e-k)!}
= d_{2m,k,l+e} \binom {l+e} e
\end{equation}
and
\begin{equation}\label{eq_d_relation_2}
d_{2n,k,l} \frac{(2n-2l-k)!}{(2m-2l-k+1)!} \frac{2^{m-l-\frac{k}{2}+1} \omega_{2n-2l-k}}{\omega_{2m-2l-k+1}}  
=  \frac{\pi^k}{(k+1)!\omega_k} \frac{1}{\left(\frac{k}{2}\right)!\left(m-\frac{k}{2}\right)!} \frac{1}{2^{m-\frac{k}{2}} \pi^m } \binom{m-\frac k 2} l.
\end{equation}

Let 
\begin{displaymath}
 \bm \kappa_k=\begin{cases}
\frac{1}{k!\left(\frac{n-k} 2\right)!(2\pi)^{\frac{n-k} 2}} \dblR^{\frac{n-k} 2} \wedge \dblg^{k}, & n-k {\rm \  even}\\
0 &{\rm otherwise.}
\end{cases}
\end{displaymath} 

 \begin{lemma}\label{lem:LKF CMs}
The  Lipschitz-Killing curvature measures 
 are given by
\begin{displaymath}
\Lambda_k^M=\left[\bm \kappa_k , (k+1)\bm  \alpha^\vee\wedge  \dblg^k\wedge \bm \psi_{n-k-1}\right]
{ =(k+1) \left[d_{n,k,\frac{n-k}2} \dblR^{\frac{n-k} 2} \wedge \dblg^{k} , \bm  \alpha^\vee\wedge  \dblg^k\wedge \bm \psi_{n-k-1}\right],}
\end{displaymath}
$k=0,\dots,n$,  under the identifications 
$\Omega^{n,n}(M,M) \simeq \Omega^n(M),  \ \Omega^{n-1,n}(SM,M) \simeq \Omega^{n-1}(SM)$ of \eqref{eq_omegan}.
\end{lemma}

\begin{proof}
This is a direct translation of the Definitions~3.1 and 3.10 from \cite {fu_wannerer}.
\end{proof}

The following computation will be useful in Section \ref{sec_hermitian_double}.
\begin{lemma}\label{lem_d_nabla_psi}
\begin{equation*}
  \dg (\bm\alpha^\vee \wedge \bm\psi_{n-k-1})
  ={ d_{n,k,\frac{n-k-2}{2}}\bm\alpha^\vee \wedge \dblR^{\frac{n-k-2}{2}} \wedge \dblR_0}+
  \sum_l  d_{ n,k,l}  \hat \dblR^l \wedge \bm\omega^{n-2l-k}.
\end{equation*}
\end{lemma}
As usual, constants with half-integer indices are taken to be zero.
\proof

Applying Proposition \ref{prop_diff_basic_forms}, we compute
\begin{align*}
\dg ( \bm\alpha^\vee \wedge \dblR^l \wedge \bm\omega^{n-2l-k-1}) & =(\dg \bm\alpha^\vee) \wedge \dblR^l \wedge \bm\omega^{n-2l-k-1} \\
 & \quad + (n-2l-k-1)\bm\alpha^\vee \wedge \dblR^l \wedge \dg\bm \omega \wedge \bm\omega^{n-2l-k-2}\\
 & =\dblR^l \wedge \bm\omega^{n-2l-k} + (n-2l-k-1)\bm\alpha^\vee \wedge \dblR^l \wedge \dblR_0 \wedge \bm\omega^{n-2l-k-2}.
 \end{align*}
By \eqref{eq_def_R0}, the first term may be expressed as   
\begin{align*}
\dblR^l \wedge \bm\omega^{n-2l-k}&= (\hat \dblR-\bm\alpha^\vee \wedge \dblR_0)^l \wedge\bm \omega^{n-2l-k} 
 \\
 &= \hat \dblR^l \wedge\bm \omega^{n-2l-k} - l \bm\alpha^\vee \wedge \dblR_0 \wedge \dblR^{l-1} \wedge \bm\omega^{n-2l-k}.
\end{align*}

Thus
\begin{align*}
  \dg (\bm\alpha^\vee \wedge \bm\psi_{n-k-1})&=\dg \sum_{l=0}^{\left\lfloor\frac{n-k-1}{2}\right\rfloor} d_{ n,k,l}\bm\alpha^\vee \wedge \dblR^l \wedge \bm\omega^{n-2l-k-1} \\
& =\sum_{l=0}^{\left\lfloor\frac{n-k-1}{2}\right\rfloor} d_{ n,k,l}  \hat \dblR^l \wedge \bm\omega^{n-2l-k}\\
 & \quad + \sum_{l=0}^{\left\lfloor\frac{n-k-1}{2}\right\rfloor-1} {\left[(n-2l-k-1)d_{n,k,l}-(l+1)d_{n,k,l+1}\right]}\bm\alpha^\vee \wedge \dblR_0 \wedge \dblR^l \wedge\bm \omega^{n-2l-k-2} \\
 & \quad + { d_{n,k,\frac{n-k-2}{2}}\bm\alpha^\vee \wedge \dblR^{\frac{n-k-2}{2}} \wedge \dblR_0}.
\end{align*}
Now \eqref{eq_recursion_tilde_c} yields the result.
\endproof


\section{Hermitian double forms}\label{sec_hermitian_double}  The present section 
forms the algebraic nucleus of the paper. We study the rich structure of double forms on a K\"ahler manifold, extending
the more rudimentary apparatus on riemannian manifolds examined in  the last section.
 One decisive 
feature is an additional generator, the {\it Gray double form} $\dblG$, modeled on the curvature tensor of the standard
metric on complex projective space. 

The primary valuation-theoretic objects of study are based on the canonical K\"ahler double forms 
$ {\bm\gamma}_{k,p}, \tilde {\bm\gamma}_{k,p}$.  In Propositions \ref{prop_rumin_in_cartan_calculus} and \ref{prop:polynomiality prime} we show that 
the valuation-theoretic  operations on these objects may be expressed invariantly in an algebraic model. We then show
in Theorem \ref{thm_restrict_gamma} that they are invariant under K\"ahler embeddings. These results will be crucial in the proofs of Theorems 
 A and B in Section \ref{sec_alg structure}.
Finally, in preparation for an argument by genericity
 we show in Theorem \ref{thm_pointwise_lemma} that a full dimensional cone of such tensors actually appear as curvature tensors of complex 
analytic subvarieties of $\CC^n$.
 
\subsection{Algebraic curvature tensors}
An {\it algebraic curvature tensor} on a real vector space $V$ is a symmetric double form
$\dblR= \dblR^\vee \in  \bw^{2,2}(V,V)$
that satisfies the Bianchi identity
\begin{displaymath}
\dblR' =0.
\end{displaymath}
Denote the space of all algebraic curvature tensors by $\calR = \calR(V)$. 
Denote by $\calK=\calK(\CC^n)$ the subspace of $\calR(\CC^n)\subset \bw^{2,2}(\CC^n, \CC^n)$ consisting of all algebraic curvature tensors that satisfy the additional condition
\begin{equation}\label{eq:kahler_symmetry}
\dblR(JX,JY; Z,W)= \dblR(X,Y;JZ,JW)= \dblR(X,Y;Z,W),
\end{equation} where $JX=\sqrt{-1}\cdot X.$
We refer to $\calK$ as
 \emph{the space of algebraic K\"ahler curvature tensors}, as the curvature tensor of a Kähler manifold satisfies \eqref{eq:kahler_symmetry} (see e.g. \cite[p.\ 165]{kobayashi_nomizu_vol2}).

\subsection{Double forms on Kähler manifolds}\label{sect_blocks} 
 Let $M$ be an $n$-dimensional Kähler manifold, with riemannian metric $\dblg\in\Omega^{1,1}(M)$, complex structure $J$, and Kähler form $F\in \Omega^{2}(M)$ related by
 \[
  F(X,Y)= \dblg(JX;Y).
 \]

 Let $\bm\beta \in \Omega^{1,0}(SM,M)$ be given by $$\bm\beta_\xi(X)=\dblg(J\xi;d\pi X).$$  Note that $(\bm\beta^\vee)'=\bm \beta.$

Define the {\it Gray double form} $\dblG\in\Omega^{2,2}(M)$ by 
 \begin{align}\notag
  \dblG(W,X;Y,Z)& :=\dblg(W;Y)\dblg(X;Z)-\dblg(W;Z)g(X;Y)+\dblg(JW;Y)\dblg(JX;Z)\\
  & \quad -\dblg(JW;Z)\dblg(JX;Y)+2\dblg(JW;X)\dblg(JY;Z).\label{eq:def_G}
 \end{align}
 Note that if $M$ is a space of constant holomorphic sectional curvature $4\lambda$, then  $\dblR=\lambda \dblG$ (see e.g.  \cite{gray_book}*{Lemma 6.9}).

\begin{definition}\label{def_GAMMA} Define  for $k<2n$ and  $ 0\leq 2p\leq k$ the elements
\begin{align}
\bm\gamma_{k,p} &=\bm\alpha^\vee \wedge  \dblG^p \wedge \dblg^{k-2p}\wedge \bm\psi_{2n-k-1}\label{eq_def_gamma_kp}\\
\tilde {\bm\gamma}_{k,p} & = \bm \alpha^\vee \wedge \bm \beta\wedge  \bm\beta^\vee
 \wedge  \dblG^p \wedge \dblg^{k-2p-1} \wedge  \tilde {\bm\psi}_{2n-k-1}
\end{align}
of $\Omega^{2n-1,2n} (SM,M)\simeq  \Omega^{2n-1} (SM)$, and put $\bm\gamma_{2n,p}=\widetilde{\bm \gamma}_{2n,p}=0$.  Let also 
\begin{displaymath}
 \bm\varphi_{k,p}= d_{2n,k, n-\frac k 2}
\dblG^p\wedge \dblg^{k-2p}\wedge \dblR^{n-\frac{k}{2}}
\end{displaymath} 
in $\Omega^{2n,2n}(M)\simeq\Omega^{2n}(M)$.

Define the curvature measures
\begin{align}
\label{eq_def_Gamma_interior} \Gamma_{k,p}
&= \left[ \bm \varphi_{k,p},\bm\gamma_{k,p}\right]\\
\notag \tilde \Gamma_{k,p}&:= [0,\tilde{ \bm \gamma}_{k,p}],
\end{align}
where we identify these double forms with ordinary forms via
\eqref{eq_omegan}.
\end{definition}

\remark Note that the $\Gamma_{k,p}$ are {\it not} the curvature measures denoted by this symbol in \cite{bernig_fu_hig,bernig_fu_solanes}.

As we will see below, the $\Gamma_{k,p}, \tilde \Gamma_{k,p}$ represent K\"ahler 
curvature measures that behave well with respect to K\"ahler inclusions, and (in the case
of $\Gamma_{k,p}$) with respect to Alesker multiplication. This is essentially guaranteed by the 
particular linear combinations of canonical double forms that constitute the $\bm \psi_j,
\tilde{\bm \psi}_j$. Our ultimate objects of interest, however, will be the curvature measures
$\Delta_{k,p}^M, B_{k, p}^M$---linear combinations of the $\Gamma_{k,p}, \tilde \Gamma_{k,p}$, respectively, that carry special geometric significance---and the valuations $\mu_{k, p}$ that are their common 
globalizations.

\subsection{More relations among hermitian double forms}
 
 As in \eqref{eq_def_R0}, let $\dblG_{0}  \in \Omega^{2,1}(SM,M)$ be given by 
 \[
  (\dblG_0)_\xi(X_1,X_2;Y)=\dblG_x(d\pi X_1, d\pi X_2;Y,\xi).
 \] Since $\dblG_0'=0$ and $(\dblG_0^\vee)'=\dblG_0$ on complex space forms by  \eqref{eq_primes} and \eqref{eq_rvee_prime}, the same holds on any Kähler manifold $M$.

We will need another operator on double forms on  a K\"ahler manifold.

\begin{proposition}
 Let $\bm\varphi$ be a  spherical double form of type $(p,q)$ on a Kähler manifold $M$. Define a new  spherical double form of the same type by 
 \begin{displaymath}
  J\bm\varphi(X_1,\ldots,X_p;Y_1,\ldots,Y_q):=\left.\frac{d}{dt}\right|_{t=0}\bm \varphi(X_1,\ldots,X_p;e^{it}Y_1,\ldots,e^{it}Y_q), 
 \end{displaymath}
for $X_1,\ldots,X_p \in T_{\xi} SM$ and $Y_1,\ldots,Y_q \in T_{\pi\xi} M$,
 where on the right hand side we use the given complex structure on $TM$.
Then $J$ is a derivation, i.e. for double forms $\bm\varphi,\bm\theta$ we have 
\begin{displaymath}
 J (\bm\varphi \wedge \bm\theta)=J\bm\varphi \wedge\bm \theta+\bm\varphi \wedge J\bm\theta
\end{displaymath}
and $J$ commutes with the exterior derivative,
\begin{displaymath}
 \dg \circ J = J \circ  \dg.
\end{displaymath}
Moreover, we have {for $\dblR \in \calK$}
\begin{equation}\label{eq:J_basic_forms}
 J\bm\alpha^\vee=-\bm\beta^\vee,\quad J\bm\beta^\vee = \bm\alpha^\vee,\quad J \dblR=0,
\end{equation}
and 
\begin{equation} \label{eq:dblG} 
\dblG  = \frac12(  \dblg \wedge \dblg +  J\dblg \wedge J\dblg + 4 F \wedge F^\vee).
\end{equation} 
\end{proposition}

\proof
Straightforward.
\endproof

\begin{proposition} \label{prop_diff_basic_forms_kahler} On $\Omega^{\bullet,\bullet}(SM,M)$ we have the following differential identities 
\begin{equation} \label{eq_gray_diff_basic_forms_kahler}
d_\nabla F^\vee=0, \quad \dg  \dblG =0,\quad d_\nabla \bm\beta^\vee=-J\bm\omega,
\end{equation}
and
\begin{equation} \label{eq_dnabla_G0}
\dg (\dblG_0)^\vee=-\bm\omega \wedge \dblg-J\bm\omega \wedge J\dblg-2\dg\bm\beta \wedge F^\vee.
\end{equation}
\end{proposition}
\begin{proof}
If the orthonormal frame $e_0,\ldots, e_{2n+1}$ is hermitian (i.e. $e_{2i+1}=Je_{2i}$), we have 
\[J(1\otimes \theta_{2i})=-1\otimes \theta_{2i+1},\quad J(1\otimes \theta_{2i+1})=1\otimes \theta_{2i}.\] Moreover, since $M$ is K\"ahler we have $\nabla J=0$, which translates into
\[
 \omega_{2i,2j}=\omega_{2i+1,2j+1},\qquad \omega_{2i,2j+1}=-\omega_{2i+1,2j}.
\]
Since
\begin{displaymath}
 F = \sum_{i=0}^{n-1} \theta_{2i}\wedge \theta_{2i+1} \quad \text{and} \quad \beta = \theta_1,
\end{displaymath}
using these relations and \eqref{eq:d_solder}, the first and last identity in \eqref {eq_gray_diff_basic_forms_kahler} are straightforward. As for $d_\nabla G=0$, it follows immediately from \eqref{eq:dblG} since $dF=0$ on K\"ahler manifolds.

The identity \eqref{eq_dnabla_G0} then follows from 
\begin{displaymath}
\dblG_0^\vee=-\bm\alpha^\vee \wedge \dblg+\bm\beta^\vee \wedge J\dblg-2\bm\beta \wedge F^\vee.
\end{displaymath}
\end{proof}

Given a complex linear isometry $\sigma\colon \C^n\to T_xM$, the isomorphism \eqref{eq:sigma_x}
\begin{displaymath}
 \bar\sigma^*\colon \left.\Omega^{\bullet}(SM)\right|_{S_xM}\to \bw^\bullet (\C^n)^*\otimes \Omega^\bullet(S^{2n-1})\simeq\left.\Omega^{\bullet}(S\C^n)\right|_{\{0\}\times S^{2n-1}}
\end{displaymath}
extends to an algebra isomorphism 
\begin{align} \bar\sigma^*\colon& \left.\Omega^{\bullet,\bullet}(SM,M)\right|_{S_xM}\to \left.\Omega^{\bullet,\bullet}(S\C^n,\C^n)\right|_{\{0\}\times S^{2n-1}} \label{eq:sigma_double}
 \end{align}
given by
\begin{displaymath}
  \bar\sigma^*(\omega\otimes s)=\bar\sigma^*(\omega)\otimes \sigma^*(s),
\end{displaymath}
where $\omega\in \restrict{\Omega^\bullet(SM)}{S_xM}$,  $s$ is a map  ${S_xM}\to \bw^\bullet T_x^*M$, and $\sigma^*(s)$ is the corresponding map  ${S^{2n-1}}\to \bw^\bullet (\C^n)^*$.

\begin{lemma}\label{lem:constant} The images of $\bm\alpha, d\bm\alpha, \bm\beta, d\bm\beta, \dblg, J\dblg,\bm\omega,J\bm\omega,\dblG,\dblG_0$ under $\bar\sigma^*$ do not depend on $M,x$ or $\sigma$. 
\end{lemma}
\begin{proof}
Let us check the statement for $\bm\omega$, the other cases being similar. Given   $X=(X_1,X_2)\in \C^n\oplus v^\bot\simeq T_{(0,v)} S\C^n$ and $Y\in\C^n$ one has
\begin{equation}\label{eq:sigma_omega}
 \bar \sigma^*\bm\omega(X;Y)=\dblg(\sigma X_2;\sigma Y)=\langle X_2,Y\rangle,
\end{equation}
where $\langle\ ,\ \rangle$ is the $\RR$-valued euclidean inner product.
\end{proof}

 
\subsection{The Rumin differential}\label{sec_rumin}

The product of valuations is given by the formula of \cite{alesker_bernig}, which involves the Rumin differential $D$ of the boundary term.
In this section (Proposition \ref{prop:polynomiality prime}) we 
show that the product of any two valuations $\glob(\Gamma^M_{k, p})$ on a K\"ahler manifold $M$ may be expressed in terms of differential forms canonically constructed from the forms defining the two factors. Thus the key 
point is that the Rumin differential of the boundary term $\bm \gamma_{k, p}$ is canonical in the model algebra.

For convenience we extend the definition of the Rumin differential to double 
forms $\bm \omega = \omega \otimes \vol_M\in \Omega^{2n-1,2n} (SM,M)$ on an $n$-dimensional K\"ahler manifold
$M$ by
\begin{equation}\label{eq_D_double}
D\bm \omega := (D\omega) \otimes \vol_M.
\end{equation}
Recalling Lemma \ref{lem_dnabla}, this operation may be defined directly by
\begin{equation}\label{eq_D_double_def}
D\bm \omega = \dg(\bm \omega + \bm \alpha \wedge \bm \xi),
\end{equation}
where $\bm \xi \in \Omega^{2n-2,2n}(SM,M)$ is uniquely determined, modulo $\bm \alpha$,
by the condition that the expression \eqref{eq_D_double_def} is a multiple of $\bm \alpha$.  Recall also that $\calK(\CC^n)$ denotes the space of algebraic K\"ahler curvature tensors and that \eqref{eq:def_tau_zeta} defines a map $\tau\colon\calV(M)\to  \Omega^{2n}(SM)$.

\begin{proposition} \label{prop_rumin_in_cartan_calculus} 
There exist polynomial maps
 \begin{displaymath}
 P_{k,j},\widetilde P_{k,j}:\calK(\CC^n) \to \bw^\bullet (\CC^n)^*\otimes \Omega^\bullet(S^{2n-1}),
 \end{displaymath}
 such that, for any K\"ahler manifold $M$ of dimension $n$, any point $x\in M$, and any complex linear isometry $\sigma\colon \CC^n\to T_xM$,  we have
\begin{align}\label{eq_D_gamma}
 \restrict{D\bm \gamma_{k, j}^M}{S_xM} &= (\bar\sigma^*)^{-1} \widetilde P_{k, j}( \sigma^*\dblR^M_x)\\
\label{eq_tau_Gamma} \restrict{\tau  (\glob\Gamma_{k, j}^M)}{S_xM} &= (\bar\sigma^*)^{-1}  P_{k, j}( \sigma^*\dblR^M_x).
\end{align}
\end{proposition}

\proof
Taking $\bm \omega = \bm \gamma_{k, j}$ in \eqref{eq_D_double_def} above, we will show  that  we may take
\begin{displaymath}
\bm \xi:={ - }j \bm\alpha^\vee \wedge       \dblG^{j-1} \wedge \dblG_0^\vee \wedge \dblg^{k-2j} \wedge \widetilde{\bm \psi}_{2n-k-1}.
\end{displaymath}
Since the $ \widetilde {\bm\psi}_i$ are polynomial in $\dblR$,
in view of Lemma \ref{lem_dnabla}, Propositions \ref{prop_diff_basic_forms} and \ref{prop_diff_basic_forms_kahler}, and Lemma \ref{lem:constant}, we have that ${\bar\sigma^*} (\dg \bm\gamma_{k,j})$ and ${\bar\sigma^*}(\dg\bm\xi)$ depend polynomially on ${\sigma^*}(\dblR_x^M)$, which will yield  \eqref{eq_D_gamma}. Recalling \eqref{eq:zeta_tau}, since ${\bar\sigma^*(\pi^*\bm\varphi_{k,p})}$ is polynomial in $\sigma^* (\dblR_x^M)$, this will give \eqref{eq_tau_Gamma}.

Let us write, similarly as in \eqref{eq_def_R0}
\begin{align*}
\dblg & =  \bm\alpha\wedge\bm \alpha^\vee+\hat \dblg,
\end{align*}
and put $\widehat \dblG  = \dblG + \bm\alpha^\vee \wedge \dblG_0$.
Therefore
\begin{align*}
\dblG^j\wedge\dblg^{k-2j} & = -j\bm\alpha^\vee \wedge \dblG^{j-1} \wedge \dblG_0 \wedge \dblg^{k-2j}\\
& \quad +(k-2j) \bm\alpha \wedge\bm\alpha^\vee \wedge  \dblG^j \wedge \dblg^{k-2j-1}
+ \widehat \dblG^j \wedge \hat \dblg^{k-2j}.
\end{align*}

By Lemma \ref{lem_d_nabla_psi} and \eqref{eq_def_R0},
\begin{align*}
\dg\bm\gamma_{k,j} & = \dg (\bm\alpha^\vee \wedge \dblG^j\wedge\dblg^{k-2j} \wedge \bm\psi_{2n-k-1})\\
& =\dg (\bm\alpha^\vee \wedge\bm \psi_{2n-k-1}) \wedge\dblG^j\wedge\dblg^{k-2j}\\
& = -\bm\alpha^\vee \wedge \sum_{ l\geq 0} j  d_{2n,k,l}   \dblG^{j-1} \wedge \dblG_0 \wedge \dblg^{k-2j} \wedge \dblR^l \wedge \bm\omega^{2n-2l-k}\\
& \quad +\bm\alpha \wedge\bm\alpha^\vee \wedge \sum_{ l \geq 0} (k-2j)d_{2n,k,l}  \dblR^l \wedge \dblG^j \wedge \dblg^{k-2j-1} \wedge \bm\omega^{2n-2l-k}\\
&  \quad +\sum_{l\geq 0} d_{2n,k,l}\widehat \dblR^l\wedge\bm\omega^{2n-2l-k}\wedge\widehat\dblG^j\wedge\widehat\dblg^{k-2j}\\
&\quad+ d_{2n,k,n-\frac{k}2-1}\bm\alpha^\vee  \wedge \dblR^{n-\frac{k}{2}-1} \wedge \dblR_0 \wedge \dblG^j\wedge\dblg^{k-2j}\\ 
& \equiv -\bm\alpha^\vee  \wedge \sum_{ l\geq 0} j  d_{2n,k,l}   \dblG^{j-1} \wedge \dblG_0 \wedge \dblg^{k-2j} \wedge \dblR^l \wedge\bm \omega^{2n-2l-k} \mod \alpha\otimes 1.
\end{align*}
Note that the last term in the next to last line is a multiple of $\vol \otimes \vol$, hence $\equiv 0 \mod \alpha \otimes 1$.  As for the sum in the previous line, it vanishes as
\[
 \widehat\dblR_\xi(\cdot,\cdot\,;\cdot,\xi)=\widehat\dblG_\xi(\cdot,\cdot\,;\cdot,\xi)=0,\qquad \widehat\dblg_\xi(\cdot\,;\,\xi)=\bm\omega_\xi(\cdot\,;\,\xi)=0.
\]

A double form of type $(p,2n+1)$ necessarily vanishes.  Using Proposition \ref{prop_diff_basic_forms} and \eqref{eq_rvee_prime} we thus find 
\begin{align*}
0 & =\left(\bm\alpha^\vee  \wedge    \dblG^{j-1} \wedge \dblG_0^\vee \wedge \dblg^{k-2j} \wedge \sum_{l \geq 0} d_{2n,k,l} \dblR^l \wedge\bm \omega^{2n-2l-k}\right)'\\
& =  \bm\alpha \wedge     \dblG^{j-1} \wedge \dblG_0^\vee \wedge \dblg^{k-2j} \wedge \sum_{l \geq 0} d_{2n,k,l} \dblR^l \wedge \bm\omega^{2n-2l-k} \\
& \quad -\bm\alpha^\vee  \wedge   \dblG^{j-1} \wedge \dblG_0 \wedge \dblg^{k-2j} \wedge \sum_{l \geq 0} d_{2n,k,l} \dblR^l \wedge \bm\omega^{2n-2l-k}\\
& \quad +   \bm\alpha^\vee  \wedge  \dblG^{j-1} \wedge \dblG_0^\vee \wedge \dblg^{k-2j} \wedge \sum_{l \geq 0}  (2n-2l-k) d_{2n,k,l} \dblR^l \wedge \bm \omega^{2n-2l-k-1} \wedge d_\nabla\bm\alpha. 
\end{align*}

It follows that 
\begin{displaymath}
\dg \bm\gamma_{k,j} \equiv - \dg\bm\alpha \wedge \bm\xi \mod \bm \alpha,
\end{displaymath}
as claimed. 
\endproof

The next statement, a direct consequence of 
Proposition~\ref{prop_double_fibration} and Proposition \ref{prop_rumin_in_cartan_calculus},
implies that the multiplication tables for these families of valuations and
curvature measures are universal across  K\"ahler manifolds $M$.
%
Here $\tau $ refers to the expression
\eqref{eq_ab_formula1} in Proposition ~\ref{prop_double_fibration}.
Put $W_\CC=\CC^n\oplus e_0^\bot$,  where $\perp$ refers to the
$\RR$-valued euclidean inner product.

\begin{proposition} \label{prop:polynomiality prime} 
There exist polynomial maps
\begin{displaymath}
 P_{k,p,l,q}\colon \mathcal K(\CC^n)\to \bigwedge\nolimits^{2n} { W^ *_{\CC}},
\end{displaymath}
such that, for any 
 Kähler manifold $M$ of dimension $n$, any $\xi\in SM, x=\pi(\xi)$, and any complex linear  isometry $\sigma\colon\CC^n\to T_xM$ with $\sigma e_0=\xi$,
\begin{equation}\label{eq_poly_product}
 \tau(\glob(\Gamma_{k,p}^M)\cdot \glob(\Gamma_{l,q}^M))_\xi = (\bar\sigma_\xi^*)\inv P_{k,p,l,q}(\sigma^*\dblR^M_x).
\end{equation}

Moreover, there exist polynomial maps $ R_{k,p,l,q},\widetilde R_{k,p,l,q} \colon \mathcal K(\CC^n)\to \bigwedge\nolimits^{2n-1} { W^*_{\CC}}$ 
	such that, for any 
	Kähler manifold $M$  of dimension $n$, any $\xi\in SM, x=\pi(\xi)$,  and any complex linear isometry  isometry $\sigma\colon\CC^n\to T_xM$ with $\sigma e_0=\xi$,
	$$ (\bar\sigma_\xi^*)\inv R_{k,p,l,q}(\sigma^*\dblR^M_x) \quad \text{and}\quad (\bar\sigma_\xi^*)\inv \widetilde R_{k,p,l,q}(\sigma^*\dblR^M_x)$$
	give the boundary terms of $\glob(\Gamma_{k,p}^M)\cdot \Gamma_{l,q}^M$ and $\glob(\Gamma_{k,p}^M)\cdot \widetilde \Gamma_{l,q}^M$ at 
	$\xi$, respectively.
\end{proposition}

The following identity is a special case of Proposition \ref{prop_glob_delta_beta} below.
\begin{proposition} \label{prop_glob_b10}
We have for every K\"ahler manifold $M$
\begin{displaymath}
\glob_M (\Gamma_{1,0})=\glob_M (\tilde \Gamma_{1,0}).
\end{displaymath}
\end{proposition}

\proof 
Put
 \begin{align*}
  \overline{\bm\psi}_{2n-3}&=\sum_{l=0}^{n-2} (2n-2l-2)  d_{2n,1,l} \dblR^l\wedge\bm\omega^{2n-2l-3}\\
  \bm\eta&= \bm\alpha^\vee \wedge\bm\beta^\vee\wedge \overline{\bm\psi}_{2n-3}\wedge J\dblg.\label{eq:def_eta}
 \end{align*}
Recalling the Definition \ref{def_GAMMA}, we claim that
\begin{equation}\label{eq:prop_rumin_beta_exact}
 d \bm\eta \equiv \widetilde{\bm\gamma}_{1,0}-\bm\gamma_{1,0} \mod \bm\alpha,d\bm\alpha.
\end{equation}
Since exact forms and multiples of $\bm\alpha, d\bm\alpha$ induce curvature measures that lie in the kernel of the globalization map, this will imply the statement. 

By Lemma \ref{lem_dnabla}, \eqref{eq_gray_diff_basic_forms}, and \eqref{eq_gray_diff_basic_forms_kahler},
\begin{align}\label{eq:d_nabla_eta}\begin{split}
 d_\nabla \bm\eta=&\bm \omega\wedge\bm\beta^\vee\wedge \overline{\bm\psi}_{2n-3}\wedge J\dblg \\
 \quad & -\bm\alpha^\vee\wedge J\bm\omega\wedge \overline{\bm\psi}_{2n-3}\wedge J\dblg +\bm\alpha^\vee\wedge\bm\beta^\vee\wedge d_\nabla\overline{\bm\psi}_{2n-3}\wedge J\dblg.\end{split}
\end{align}

Note that $J\bm\psi_{2n-2}=J\bm\omega\wedge\overline{\bm\psi}_{2n-3}$ and note also that $J(\rho\otimes \vol)=0$ for any form $\rho\in \Omega^\bullet(SM)$. Therefore
\begin{align*}
 0=J(\bm\alpha^\vee\wedge\bm\psi_{2n-2}\wedge J\dblg)=-\bm\beta^\vee\wedge\bm\psi_{2n-2}\wedge J\dblg+\bm\alpha^\vee\wedge J\bm\omega\wedge\overline{\bm\psi}_{2n-3}\wedge J\dblg-\bm\alpha^\vee \wedge\bm\psi_{2n-2}\wedge \dblg,
\end{align*}
whence
\begin{align}
 d_\nabla \bm\eta & = -\bm\beta^\vee\wedge(\bm\omega\wedge\overline{\bm\psi}_{2n-3}+\bm\psi_{2n-2}+\bm\alpha^\vee\wedge d_\nabla \overline{\bm\psi}_{2n-3})\wedge J\dblg-\bm\alpha^\vee\wedge \bm\psi_{2n-2}\wedge \dblg.\label{eq:d_nabla_bis}
\end{align}

By \eqref{eq_gray_diff_basic_forms} and \eqref{eq_recursion_tilde_c} we get
\begin{align*}
 d_\nabla \overline{\bm\psi}_{2n-3}
 &= \sum_l (l+1)(2n-2l-3) d_{2n,1,l+1} \dblR^l\wedge \dblR_0\wedge\bm\omega^{2n-2l-4}\\
 &= \sum_l l(2n-2l-1)  d_{2n,1,l} \dblR^{l-1}\wedge \dblR_0\wedge\bm\omega^{2n-2l-2}.
\end{align*}

Using the decomposition $\dblR=\dblR_0\wedge\bm\alpha^\vee +\widehat \dblR$, we rewrite
\begin{align*}
\bm \omega\wedge\overline{\bm\psi}_{2n-3}=& \sum_l (2n-2l-2) { d}_{2n,1,l} \dblR^l\wedge\bm\omega^{2n-2l-2}\\
 =&\sum_l l(2n-2l-2) {d}_{2n,1,l} \dblR_0\wedge\bm\alpha^\vee\wedge \dblR^{l-1}\wedge\bm\omega^{2n-2l-2}\\
 & +\sum_l (2n-2l-2) d_{2n,1,l} \widehat \dblR^l\wedge\bm\omega^{2n-2l-2},
\end{align*}
and similarly
\begin{align*}
\bm \psi_{2n-2}&= \sum_l  d_{2n,1,l} \dblR^l \wedge \bm\omega^{2n-2l-2}\\
 &=\sum_l l d_{2n,1,l} \dblR_0\wedge \bm\alpha^\vee \wedge \dblR^{l-1}\wedge \bm\omega^{2n-2l-2} +\sum_l d_{2n,1,l} \widehat \dblR^l\wedge\bm\omega^{2n-2l-2}.
\end{align*}
It follows that
\begin{equation}\label{eq:telescopic}
\bm\omega\wedge\overline{\bm\psi}_{2n-3}+\bm\psi_{2n-2}+\bm\alpha^\vee\wedge d_\nabla\overline {\bm\psi}_{2n-3}=\sum_l (2n-2l-1) d_{2n,1,l}\widehat \dblR^l \wedge \bm\omega^{2n-2l-2}.
\end{equation}

Combining \eqref{eq:d_nabla_bis} and \eqref{eq:telescopic}, and using the identity 
$$J\dblg = \bm\alpha^\vee \wedge \bm\beta+ \widehat{J\dblg},$$
where $\widehat {(J\dblg)}_\xi(\,\cdot\, ;\, \xi)=0$,  we conclude
\begin{align*}
 d_\nabla\bm\eta & = -\bm\beta^\vee\wedge \sum_l (2n-2l-1) d_{2n,1,l} \widehat \dblR^l\wedge \bm\omega^{2n-2l-2}\wedge J\dblg-\bm\alpha^\vee\wedge \bm\psi_{2n- 2}\wedge \dblg\\
 &=- \bm\beta^\vee\wedge \sum_l (2n-2l-1) d_{2n,1,l} \widehat \dblR^l\wedge\bm \omega^{2n-2l-2}\wedge (\bm\alpha^\vee \wedge \bm \beta+\widehat{J \dblg})-\bm\alpha^\vee\wedge \bm\psi_{2n-2}\wedge \dblg\\
 &=\bm\alpha^\vee\wedge\bm\beta\wedge\bm\beta^\vee\wedge\widetilde{\bm\psi}_{2n-2} -\bm\alpha^\vee\wedge \bm\psi_{2n- 2}\wedge \dblg,
\end{align*}
which is our claim \eqref{eq:prop_rumin_beta_exact}.
\endproof


\subsection{Restrictions} 
\label{subsec_restrictions}

The main result of this subsection is 
\label{sec_restrict_gamma}
\begin{theorem}\label{thm_restrict_gamma}
 If $ M \subset \tilde M$ is an inclusion of K\"ahler manifolds then
\begin{align}
\label{eq_restrict_Gamma} \restrict{\Gamma_{k,p}^{\tilde M}}M &= \Gamma^M_{k,p}, \\
\label{eq_restrict_tilde_Gamma} \restrict{\tilde \Gamma_{k,p}^{\tilde M}}M &= \tilde \Gamma^M_{k,p}
\end{align}
\end{theorem}

We will prove this theorem by relating the double forms $\bm \gamma^M_{k, p},\tilde {\bm\gamma}^M_{k, p}, \bm \varphi_{k, p}^M$ defining the curvature measures on $M$ to
the corresponding forms 
 $\bm \gamma^{\tilde M}_{k, p},\tilde {\bm\gamma}^{\tilde M}_{k, p}$ for $\tilde M$. Denote the complex dimensions of $M,\tilde M$ by $m,n$ respectively. The case $m=n$ is trivial so we assume $d:= n-m>0$. Since 
 a subset $A \subset M$ has empty interior as a subset of $\tilde M$, the interior term
 $\bm \varphi^{\tilde M}_{k, p}$ is irrelevant.

 This task reduces to expressing the normal
 cycle $\nc(A,\tilde M)$ of $A$, considered as a subset of $ \tilde M$, in terms of its normal cycle $\nc(A, M)$ as a
 subset of $M$. 
Since the statement is local, 
 we may fix a hermitian orthonormal local frame 
$$b_{2m},\ldots,b_{2n-1}, \quad { Jb_{2r} = b_{2r+1}}
$$ 
of the normal bundle  $NM$ of $M$ in $\tilde M$. Denote the normal space to $M$ at a point $x$ by $N_xM \subset T_x\tilde M$. 

For $(x,y) \in M\times S^{2d-1}$, $y = (y_{2m},\ldots,y_{2n-1})$, put 
$$
yb:= \sum_{r=2m}^{2n-1} y_r b_r(x)
$$
lying in the unit sphere in the normal fiber  $N_xM$ at $x$,
and define similarly $Yb$ for $Y \in T_yS^{2d-1}\subset \CC^d$.

Consider the maps 
 \begin{align*}
 F\colon & SM \times \left(0,\frac{\pi}{2}\right) \times S^{2d-1} \to S\tilde M,\\
  (x,v,\theta,y) & \mapsto \left(x,\cos \theta\, v + \sin \theta\, yb \right)\\
F_0\colon & M \times S^{2d-1}  \to NM \subset S\tilde M\\
(x,y) & \mapsto (x,yb),
\end{align*}
together with the projections
\begin{align*}
P:   SM \times \left(0,\frac{\pi}{2}\right) \times S^{2d-1}&\to SM, \\
P_0:M \times S^{2d-1} &\to M.
\end{align*}
It then holds for every $A \in \mathcal P(M)$ that 
  
\begin{equation}\label{eq_nc_of_inclusion}
\nc(A \subset \tilde M)=(F_0)_*\left( [[A]] \times [[S^{2d-1}]]\right)+F_*\left(\nc(A) \times \left[\left[0,\frac{\pi}{2}\right]\right] \times [[S^{2d-1}]]\right),
\end{equation}
where the double brackets mean that we consider the corresponding oriented manifold as a current. In other words, if $[\phi,\omega] \in \calC(\tilde M)$ then
\begin{equation}\label{eq_restricted_cm}
\calC(M) \owns\restrict {[\phi,\omega] }M= \left[P_{0*}  F_0^*\omega,P_* F^*\omega\right],
\end{equation}
where $P_{0*}, P_*$ denote fiber integration with respect to $P_0,P$.
(Observe that since the codimension $d>0$, any restriction $\restrict{ [\phi,0]}M = 0$.)

We will apply \eqref{eq_restricted_cm} to the forms $\omega$ representing the
boundary terms of $\Gamma^{\tilde M}_{k,p},\tilde \Gamma^{\tilde M}_{k,p}$. Formally,
these are double forms $\bm \gamma^{\tilde M}_{k, p}, \tilde{\bm \gamma}^{\tilde M}_{k, p} \in \Omega^{2n-1,2n}(S\tilde M, \tilde M) $, which we identify with elements of $\Omega^{2n-1}(S\tilde M)$ under the convention \eqref{eq_omegan}. 

Following the
usual formalism of double forms, the pullback and pushforward operations above are
applied to the first factor. This will result in elements of $\Omega^{2m-1,2n}(i\circ \pi)$ and $
\Omega^{2m,2n}(i)$, where $i:M\to \tilde M$ denotes the inclusion map and $\pi:SM\to M$
is the projection. Denote by
$\vol_{NM} \in \Omega^{0,d}(i)$ the section consisting of the volume forms of the normal fibers $N_xM$. Pulling it back in the obvious way, we obtain injections
$$
\Omega^{\bullet,\bullet}(SM,M )\hookrightarrow \Omega^{\bullet,\bullet +d}(i \circ \pi ), \quad 
\Omega^{\bullet,\bullet}(M )\hookrightarrow \Omega^{\bullet,\bullet +d}(i )
$$
by $\bm \gamma \mapsto \bm \gamma \wedge \vol_{NM}$.
In these terms, Theorem 
\ref{thm_restrict_gamma} will follow from the identities
\begin{align}
\label{eq_PF1} P_*F^* \bm \gamma_{k,p}^{\tilde M} &= \bm \gamma_{k,p}^{ M} \wedge\vol_{NM},\\
\label{eq_PF2}P_{0*}F_0^* \bm \gamma_{k,p}^{\tilde M} &= \bm\varphi_{k,p}^M \wedge \vol_{NM},\\
\label{eq_PF3}P_*F^*{ \tilde{\bm \gamma}_{k,p}}^{\tilde M} &= \tilde{\bm \gamma}_{k,p}^M \wedge\vol_{NM},\\
\label{eq_PF4}P_{0*}F_0^*{\tilde{ \bm \gamma}_{k,p}}^{\tilde M} &= 0.
\end{align}

\subsection{Some identities for pullbacks of double forms} 
 Let us first compute
the pullbacks of $\bm \alpha,\bm \alpha^\vee, \bm \beta^\vee, \bm \omega$ under $F,F_0$, which live in the spaces $\Omega^{\bullet,\bullet}(\pi\circ F),
\Omega^{\bullet,\bullet}(\pi\circ F_0)$, respectively.
We may consider $\Omega^{\bullet,\bullet}(\pi\circ F_0)$ to be a subspace of $\Omega^{\bullet,\bullet}(\pi\circ F)$ via pullback under the projection map $SM \times \left[0,\frac \pi 2\right]\times S^{2d-1} \to M \times S^{2d-1} $.  Denote by $K$ the subspace of $\Omega^{1,1}(\pi\circ F_0)$ 
(or of $\Omega^{1,1}(\pi\circ F)$)
consisting of sections of the bundle with fibers $T_x^*M \otimes N_x^*M$.

Define
\begin{displaymath}
\bm \tau:= F_0^* \bm\alpha^\vee \in \Omega^{0,1}(\pi\circ F_0),  
\end{displaymath}
and observe that
\begin{align}\label{eq_F*tau}
F^* \bm\alpha&=\cos \theta\, \bm \alpha, &&
F^* \bm\alpha^\vee =\cos \theta\, \bm \alpha^\vee+\sin \theta\, \bm \tau,\\
\label{eq_F_beta}
 F^*\bm\beta
&=\cos \theta\, \bm \beta, && F^*\bm \beta^\vee=\cos \theta \bm\beta^\vee - \sin\theta\, J\bm\tau. 
\end{align}

Define also
\begin{displaymath}
\bm H, \bm \sigma \in \Omega^{1,1}(\pi\circ F_0) 
\end{displaymath}
by taking for $(x,y) \in M \times S^{2d-1}$ and for $ X \in T_xM, \ Y \in T_yS^{2d-1}, \ Z \in T_x \tilde M$
\begin{displaymath}
 \bm H(x,y) ((X,Y);Z):= \dblg_x(\nabla_X (yb); Z^\top),
\end{displaymath}
where $Z= Z^\top + Z^\perp \in T_xM \oplus N_xM$, a pullback of the second fundamental form of $M$ in $\tilde M$ at the normal vector $yb \in N_xM$, and 
\begin{displaymath}
 \bm \sigma ((X,Y);Z) := \dblg(Yb;Z).
\end{displaymath}

We claim that
\begin{align}\label{eq_F0_pullback_omega}
F_0^*\bm\omega^{\tilde M} & \equiv \bm H + \bm \sigma \quad \mod K.
\end{align}
Indeed, let $X=x'(0)\in T_xM$ and $Y=y'(0) \in T_yS^{2d-1}$ for some smooth curve in $M \times S^{2d-1}$. The tautological section $\xi$ along the curve $F_0(x(t),y(t))$ is given by $\sum y_r(t) b_r(x(t))$. Putting $V:=dF_0(X,Y)$, and using \eqref{eq_charactezitation_pull_back_connection} one gets
\begin{displaymath}
(\pi^*\nabla)_V \xi=\sum_r y_r'(0) b_r + y_r(0) (\pi^*\nabla)_V \pi^*b_r= Yb + \nabla_X yb.
\end{displaymath}	
Equation \eqref{eq_F0_pullback_omega} then follows easily. 

Define $\vol_S\in \Omega^{2d-1,0}(\pi\circ F_0)$ to be the volume form of the $S^{2d-1}$ factor, and observe that
\begin{equation} \label{eq_tau_wedge_sigma}
\bm\tau \wedge \bm \sigma^{2d-1}=(2d-1)! \vol_S \wedge \vol_{NM}.
\end{equation}

Finally, we observe that similarly to \eqref{eq_F0_pullback_omega} we have
\begin{align} \label{eq_pullback_omega} \begin{split}
 F^*\bm\omega^{\tilde M} & \equiv\cos \theta\, \bm\omega^M + \sin \theta \,(\bm H + \bm \sigma  )
  +(d\theta \otimes 1)\wedge(\cos \theta \,\bm \tau - \sin \theta \, \bm\alpha^\vee) 
\quad \mod K
  \end{split}.
\end{align}

\subsubsection{The Weyl lemma} 
Consider the  second fundamental forms associated to the $b_r$ as the double forms 
\begin{displaymath} 
\bm l_r \in \Omega^{1,1}(M), \quad r=2m,\ldots,2n-1, 
\end{displaymath}

\begin{displaymath}
\bm l_r(X ;Y):=\dblg\left(\nabla_X b_r,Y\right),\qquad X,Y\in T_xM.
\end{displaymath}
Thus
\begin{displaymath}
  \bm H(x,y) =  y\cdot (\bm l_1,\ldots,\bm l_r)_x^t=: y \, \bm l_x. 
\end{displaymath}

Define also the complex-valued version
\begin{align*}
\bm \lambda_{s}(X ;Y)&:=   \bm l_{2s}(X;Y) +  \sqrt{-1} \bm l_{2s+1}(X;Y), \quad s = m, \dots ,n-1,
\end{align*}
and observe that at each point $x \in M$ the value $(\bm \lambda_s)_x$ is a symmetric bilinear form over the complex vector space $T_xM$.

Given a double form $\dblC \in \Omega^{\bullet,\bullet}( \tilde M)$, the pullback $i^*\dblC$ by the inclusion map $i$ consists of sections of the bundle over $M$ with fibers $\bw^{\bullet,\bullet} (T_xM, T_x \tilde M), \ x\in M$.
Denote by
\begin{displaymath}
\restrict {\dblC} M \in \Omega^{\bullet,\bullet}(M) 
\end{displaymath}
the image of $i^*\dblC$ under the  pullback $\bw^{\bullet,\bullet} (T_xM, T_x \tilde M)\to \bw^{\bullet,\bullet} (T_xM, T_x  M)$ induced by $i$ on the second factor.

\begin{lemma}[Weyl lemma] \label{lemma_weyl}
For even exponents $2e$ we have for each $x \in M$
\begin{equation}
  \int_{S^{2d-1}} \bm H(x,y)^{2e} dy=
  \frac{2^{e+1}s_{2d+2e-1}}{s_{2e}}(\dblR^M_x-\dblR_x^{\tilde M}|_M)^e,
 \end{equation}
while the integral vanishes for odd exponents.
\end{lemma}

\proof If the exponent is odd then the integrand is odd in $y$. If the exponent is even,
apply \cite[Proposition A1]{bernig_curv_tensors} and use the Gauss equation
\begin{equation}\label{eq_gauss}
 \sum_{s=m}^{n-1} \bm \lambda_s \wedge \overline {\bm\lambda_s} =  \sum_{r=2m}^{2n-1} \bm l_r^2=2(\dblR^M - \dblR^{\tilde M}|_M).
\end{equation}
\endproof

\subsubsection{Proof of \eqref{eq_restrict_Gamma}} This amounts to proving 
\eqref{eq_PF1} and \eqref{eq_PF2}. Since $\dblg, \dblG$, are horizontal,
with $\restrict{\dblg^{\tilde M}}M=  \dblg^M, \restrict{\dblG^{\tilde M}}M=  \dblG^M$, referring to \eqref{eq_def_Gamma_interior} it is enough to show that
\begin{align}\label{eq_pullback_psi}
	\begin{split}
P_* F^*(\bm\alpha^\vee\wedge\bm \psi^{\tilde M}_{2n-k-1})& = \bm\alpha^\vee\wedge\bm \psi^{ M}_{2m-k-1}\wedge \vol_{NM}, \\
P_{0*}F_0^*( \bm\alpha^\vee\wedge \bm \psi^{\tilde M}_{2n-k-1}) &= d_{2m,k, m-\frac k 2} (\dblR^M)^{m-\frac k 2}\wedge \vol_{NM}.
\end{split}
\end{align}

Since the curvature double form $\dblR^{\tilde M}$ is horizontal, 
we observe that 
\begin{displaymath}
P_*F^*(\bm\alpha^\vee \wedge (\bm\omega^{\tilde M})^a\wedge (\dblR^{\tilde M})^b)
= (i^* \dblR^{\tilde M})^b\wedge P_*F^*(\bm\alpha^\vee \wedge (\bm\omega^{\tilde M})^a). 
\end{displaymath}

Since the double forms appearing in \eqref{eq_pullback_psi} are sums of terms of this form,
we first evaluate the expressions $P_*F^*(\bm\alpha^\vee \wedge (\bm\omega^{\tilde M})^a)$. 
Expanding this using \eqref{eq_pullback_omega}, it is clear from \eqref{eq_tau_wedge_sigma} that any 
terms that are {\it not } multiples of 
\begin{displaymath}
F^*\bm \alpha^\vee \wedge (d\theta\otimes 1)\wedge(\cos \theta \,\bm \tau - \sin \theta \, \bm\alpha^\vee)  · \wedge \bm\sigma^{2d-1}
=(d\theta\otimes 1)\wedge \bm\alpha^\vee  \wedge\bm \tau \wedge \bm\sigma^{2d-1} 
\end{displaymath}
map to $0$ under $P_*$. Modulo these irrelevant terms we compute  
\begin{align*}
F^*(\bm\alpha^\vee \wedge (\bm\omega^{\tilde M})^a) & \equiv \bm\alpha^\vee \wedge \sum_ k \frac{a!}{(a- k-2d)! k!(2d-1)!} \\
& \quad (\cos \theta)^{a-k-2d}\bm \omega^{a- k-2d} (\sin \theta)^{ k+2d-1} \bm H^{ k}  \\
& \quad \wedge  ( d\theta  \otimes 1) \wedge\bm \tau \wedge \bm\sigma^{2d-1}\\
& = \bm\alpha^\vee \wedge  \sum_{ k} \frac{a!}{(a- k-2d)!k!}(\cos \theta)^{a-k-2d}  (\sin \theta)^{ k+2d-1} \\
& \quad \bm \omega^{a-k-2d} \wedge \bm H^{ k} \wedge  (d\theta \otimes 1) \wedge  \vol_{S}\wedge \vol_{NM}.
\end{align*}

 Recalling the identity 
\begin{equation} \label{eq_integration_cos}
 \int_0^{\frac{\pi}{2}} \cos^a \theta \sin^b \theta d\theta=    \frac{s_{a+b+1}}{s_as_b}=  \frac{(a+b+2)\omega_{a+b+2}}{(a+1)(b+1) \omega_{a+1}\omega_{b+1}}
\end{equation}
and invoking \eqref{eq:volume_odd} and Lemma \ref{lemma_weyl},  integration yields
\begin{align*}
P_*  F^* (\bm\alpha^\vee \wedge \bm\omega^a) & = \bm\alpha^\vee \wedge \sum_e \frac{1}{2^{e}\pi^e e!} \frac{\omega_{a+1} (a+1)!}{\omega_{a-2e-2d+1}(a-2e-2d+1)!} \cdot \\
& \quad  \cdot \bm\omega^{a-2e-2d} \wedge   (\dblR^M-\dblR^{\tilde M}|_M)^e  \wedge \vol_{NM}.
\end{align*}

Setting $a= 2n-2l-k-1$, recalling that $n-d=m$, and applying \eqref{eq_d_relation_1},
\begin{align*}
P_*  F^* (d_{2n,k,l} \bm\alpha^\vee \wedge \bm\omega^{2n-2l-k-1}) & = \bm\alpha^\vee \wedge \sum_e \frac{d_{2n,k,l}}{2^{e}\pi^e e!} \frac{\omega_{2n-2l-k} (2n-2l-k)!}{\omega_{2m-2l-2e-k}(2m-2l-2e-k)!} \cdot \\
& \quad  \cdot \bm\omega^{2m-2l-2e-k-1} \wedge   (\dblR^M-\dblR^{\tilde M}|_M)^e
\wedge  \vol_{NM}\\
& = \bm\alpha^\vee \wedge \sum_e {d_{2m,k,l+e}}
\binom{l+e} e 
  \cdot \bm\omega^{2m-2l-2e-k-1} \wedge   (\dblR^M-\dblR^{\tilde M}|_M)^e\\
&\quad \wedge  \vol_{NM}.
\end{align*}

Thus
\begin{align} \notag
P_* F^*(\bm\alpha^\vee\wedge\bm \psi^{\tilde M}_{2n-k-1})& = \sum_l d_{ 2n,k,l} (\dblR^{\tilde M}|_M)^l \wedge  P_*F^*(\bm\alpha^\vee \wedge \bm\omega^{2n-2l-k-1})\\
 \label{eq:computation_res_Delta}
 &= \bm\alpha^\vee \wedge \vol_{NM}\\
\notag &\quad \wedge 
\sum_{f} d_{2m,k, f}    \bm \omega^{2m-2f-k-1}\wedge \sum_e \binom  f e
 \wedge (\dblR^{\tilde M}|_M)^{ f-e} \wedge (\dblR^M-\dblR^{\tilde M}|_M)^e\\
\notag& =  \bm\alpha^\vee \wedge \sum_{f} d_{2m,k,f} \bm\omega^{2m-2f-k-1} \wedge (\dblR^M)^f\wedge\vol_{NM}\\
\notag & = \bm\alpha^\vee\wedge \bm\psi^M_{2m-k-1} \wedge\vol_{NM},
\end{align}
which is the first relation of \eqref{eq_pullback_psi}.

We next consider the interior terms for $\Gamma_{k,p}$. 

Since $F^*_0 \bm\alpha^\vee = \bm\tau$, we conclude that, modulo the kernel of the fiber integral map,
\begin{align*}
F_0^*(\bm\alpha^\vee \wedge \bm\omega^a) & \equiv  \binom{a}{2d-1} \bm\tau \wedge \bm H^{a-2d+1} \wedge \bm\sigma^{2d-1} \\
& = (-1)^{a-2d+1} \binom{a}{2d-1} \bm H^{a-2d+1} (2d-1)! \vol_{S}|_y  \wedge \vol_{NM}.
\end{align*}
By Lemma \ref{lemma_weyl}, the fiber  integral vanishes if $a$ is even, hence $P_{0*} F_0^*\bm\psi_{k,p}=0$ if $k$ is odd. If $a$ is odd, we use Lemma \ref{lemma_weyl} to compute 
\begin{align*}
P_{0*}  F_0^*(\bm\alpha^\vee \wedge \bm\omega^a) & = \frac{a!}{(a-2d+1)!} \frac{2^{\frac{a+1}{2}-d+1} s_a}{s_{a-2d+1}} (\dblR^M-\dblR^{\tilde M}|_M)^\frac{a-2d+1}{2} \wedge \vol_{NM}.
\end{align*}

We thus have for even $k$, using \eqref{eq_d_relation_2},
\begin{align*}
&P_{0*} F_0^*(\bm\alpha^\vee \wedge \bm\psi^{\tilde M}_{2n-k-1}) \\ &  =  \sum_l d_{2n,k,l} \frac{(2n-2l-k-1)!}{(2m-2l-k)!} \frac{2^{m-l-\frac{k}{2}+1} s_{2n-2l-k-1}}{s_{2m-2l-k}}  (\dblR^{\tilde M}|_M)^l \wedge (\dblR^M-\dblR^{\tilde M}|_M)^{m-\frac k2 -l}\wedge \vol_{NM} \\
& =   \frac{\pi^k}{(k+1)!\omega_k} \frac{1}{\left(\frac{k}{2}\right)!\left(m-\frac{k}{2}\right)!} \frac{1}{2^{m-\frac{k}{2}} \pi^m }  \dblR^{m-\frac{k}{2}}\wedge \vol_{NM},
\end{align*}
which is the second relation of \eqref{eq_pullback_psi}.

\subsubsection{Proof of \eqref{eq_restrict_tilde_Gamma}} 
Recalling \eqref{eq_F_beta}, and the identity
\begin{displaymath}
\frac{s_{a+2}}{s_{a-2e-2d+2}}= \frac{a-2e-2d+1}{a+1} \frac{s_a}{s_{a-2e-2d}},
\end{displaymath} 
an argument  as above shows that 
\begin{align*}
P_*F^* (\bm\beta\wedge\bm\alpha^\vee \wedge\bm \beta^{ \vee} \wedge \bm\omega^a)  & \equiv \sum_e \frac{1}{2^{e}\pi^e e!} \frac{a-2e-2d+1}{a+1}\frac{s_{a} a!}{s_{a-2e-2d}(a-2e-2d)!} \\
& \quad  \bm\alpha^\vee \wedge\bm \omega^{a-2e-2d} \wedge  \bm \beta \wedge \bm \beta^\vee \wedge (\dblR^M-\dblR^{\tilde M}|_M)^e\wedge  \vol_{NM}
\end{align*}
and therefore, using \eqref{eq_d_relation_1},
\begin{align*}
P_* F^* ( (2n-2l-k) &d_{2n,k,l}\bm\beta\wedge\bm\alpha^\vee \wedge\bm \beta^{ \vee} \wedge \bm \omega^{2n-2l-k-1} )=\\
&= \sum_e \binom{l+e} e (2m-2l-2e-k)d_{2m,k,l+e}\\
&\quad \bm\beta\wedge\bm\alpha^\vee \wedge\bm \beta^{\vee} \wedge \bm \omega^{2m-2l-2e-k-1}
 \wedge (\dblR^M-\dblR^{\tilde M}|_M)^e\wedge  \vol_{NM}. 
\end{align*}
As above, we conclude that 
\begin{displaymath}
 P_* F^* (\bm\beta \wedge \bm\alpha^\vee \wedge \bm \beta^\vee \wedge \tilde \psi^{\tilde M}_{2n-k-1})
= \bm\beta \wedge \bm\alpha^\vee \wedge \bm \beta^\vee \wedge \tilde \psi^{ M}_{2m-k-1}\wedge \vol_{NM},
\end{displaymath}
which yields \eqref{eq_PF3}.

To prove \eqref{eq_PF4}, observe that  $\tilde {\bm\gamma}_{k,p}$ is divisible by $\bm \beta $, and note that $F_0^*\bm\beta =0$.
\endproof

\subsection{Embedded K\"ahler curvature tensors}

It is convenient to record here the following fundamental fact.
\begin{definition}
 An  algebraic K\"ahler curvature tensor $\dblR\in \calK(\CC^m)$ is called \emph{embedded} if there exists a complex submanifold $M^m \subset \mathbb C^{n+m}$, a point $p \in M$, and a $J$-linear isometry $f\colon \CC^m\to T_pM$ such that $\dblR=f^*\dblR_p^M$.
\end{definition}

\begin{theorem}[Pointwise embedding lemma] \label{thm_pointwise_lemma}
The set of embedded algebraic K\"ahler curvature tensors  
is a full-dimensional convex cone in $\calK(\CC^m)$.   
\end{theorem}
\begin{proof}
 Setting $\dblR^{\tilde M} = 0$ since $\tilde M = \CC^n$, the relation \eqref{eq_gauss} characterizes this set explicitly. Let $B$ be the complex vector space of symmetric complex-bilinear forms on $\CC^m$. Note that any finite collection $\bm \lambda_m,\dots,\bm\lambda_{n-1}\in B$ occur as second fundamental forms as above, e.g. for $M = \{\left(\vec z, \frac 1 2 \bm \lambda_{m}(\vec z, \vec z),\dots,\frac 1 2 \bm \lambda_{n-1}(\vec z, \vec z)\right): \vec z = (z_0,\dots,z_{m-1}) \in \CC^m\}\subset \CC^n$. 
 
 Polarizing, in order to prove the Theorem it is enough to show that the map
 $$
\bm \lambda \otimes_\RR \bm \mu \mapsto \bm \lambda \wedge \overline{\bm \mu }+ \bm \mu \wedge \overline{\bm \lambda }
 $$
 constitutes a surjective map  $B\otimes_\RR B \to\calK(\CC^m)$.
 
 By \cite[Theorem~3.6]{sitaramayya}, 
\begin{displaymath}
 \dim_{\RR}\calK(\CC^m)=\binom{m+1}{2}^2,
\end{displaymath}
so it is enough to show that the dimension of the image is at least this large. Selecting a basis $e_1,\dots,e_m$ of $\CC^m$, let $\theta_1,\dots,\theta_m$ be the corresponding dual basis, and observe that the images of 
\begin{displaymath}
 (\theta_i\odot \theta_j)\otimes (\theta_i\odot \theta_j),\quad
(\theta_i\odot \theta_j)\otimes (\theta_k\odot \theta_l),\quad 
(\sqrt{-1} \theta_i\odot \theta_j)\otimes (\theta_k\odot \theta_l),
\end{displaymath}
with $i\leq j$, $k\leq l$, and $(i,j)<(k,l)$ with respect to the lexicographic order,
are linearly independent, as the reader can readily confirm by evaluating them on 4-tuples of vectors selected from
$e_i,e_j,e_k,e_l$ and their images under $J$.
 Therefore the dimension of the image of $B\otimes_\RR B$ under the map above is 
$\ge \binom{m+1}{2}+2 \binom{\binom{m+1}{2}}{2}=\binom{m+1}{2}^2$, as claimed.
\end{proof}

\subsection{Angularity of the $\Gamma_{k,p}^M$}

Recall that among all curvature measures the ones called  angular are distinguished by enjoying the  simplest possible behavior on polytopes. More precisely, in $\RR^n$ a translation invariant curvature measure $\Psi$ is called angular (see \cites{bernig_fu_solanes,wannerer_angular}), if there exist functions $f_k$ on the Grassmannian $\Grass_k(\RR^n)$ of $k$-dimensional linear subspaces of $\RR^n$ such that 
$$\Psi(P,U) = \sum_{k=0}^n \sum_{\dim F = k}  f_k(\overline F) \gamma(F,P) \vol_k(F\cap U)$$
for every polytope $P\subset \RR^n$ and Borel set $U\subset \RR^n$. Here the inner sum is over all $k$-dimensional faces of $P$, $\overline F$ is the linear subspace parallel to the affine hull of $F$, and  $\gamma(F,P)$ denotes the external angle of $P$ at $F$. On  a riemannian manifold $M$ a smooth curvature measure  $\Psi$ is called angular, if $\tau_x \Psi$ is angular for all $x\in M$ {where $\tau_x$ is the transfer map \eqref{eq_transfer_map}.}

\begin{proposition} \label{prop:angular} 
  The curvature measures $\Gamma_{k,p}^M$ are angular.
\end{proposition}

\proof 
It suffices to show that  the curvature measures defined by  the $(2n-1)$-forms 
\begin{equation}\label{eq:homogeneous_terms} 
\bm\gamma_{k,p,l}:= \bm\alpha^\vee \wedge \dblG^p \wedge \dblg^{k-2l-2p}\wedge  \dblR^l \wedge \bm \omega^{2n-k-1}, \quad  l,p\geq 0,\quad 2p+2l\leq k
\end{equation}
are angular.

 Fix a point $x\in M$ and a $k$-dimensional linear subspace $E\subset T_xM$. Let $\xi \in S_xM$ be perpendicular to $E$. {Let
 \[
  \bar\tau_x\colon \restrict{\Omega^{\bullet,\bullet}(SM,M)}{S_xM}\to \Omega^{\bullet,\bullet}(ST_xM,T_xM)^{T_xM}
 \]
be the natural extension of \eqref{eq:bar_tau} to double forms.} At a point $(v,\xi)\in S T_xM$, the restriction of $\bar\tau_x  \bm \gamma_{k,p,l}$ to $E\oplus (\xi^\perp\cap E^\perp)\subset T_{(v,\xi)} ST_xM$  is given by
\begin{displaymath}
 { c_{n,k}}\bar\tau_x(\dblG^p \wedge \dblg^{k-2l-2p} \wedge \dblR^l)|_{E}  \wedge (1\otimes \vol_{E^\perp} )
 \wedge  (\vol_{S(E^\perp)}|_\xi \otimes 1) = c(E)\wedge  (\vol_{S(E^\perp)}|_\xi \otimes 1),
\end{displaymath}
where $c_{n,k}$ is a constant and $c(E)\in \largewedge^k E^*\otimes \largewedge^{2n} T_x^*M$ does not depend on $\xi \in E^\perp$. This shows that $\bar\tau_x \bm\gamma_{k,p,l}$ defines an angular translation invariant curvature measure on $T_xM$ and concludes the proof.
\endproof

\section{The Gray algebra}\label{sect_gray}
In Theorem \ref{thm:StructureGray} below we determine the structure  of the {\it Gray algebra} defined below. The 
principal consequences that we will carry forward to the concluding discussion in Section \ref{sec_alg structure} are
Propositions \ref{prop:g_kp_zero} and \ref{prop_poincare_gray}.

\subsection{Canonical elements of $\mathcal G(n)$}
\begin{definition}  The \emph{Gray algebra} {of degree $n$} is the complex subalgebra
 $\mathcal G(n)\subset \bw^{\bullet,\bullet}(\CC^n,\CC^n)\otimes_{\RR}\CC$ generated by $\dblg, \dblG$. { Clearly $\mathcal G(n)$ is a graded commutative algebra over $\C$:}
 \begin{displaymath}
\mathcal G(n)=\bigoplus_{k=0}^{2n} \mathcal G_k(n),  
 \end{displaymath}
{ where $\mathcal G_k(n)$ denotes the subspace of  elements of type $(k,k)$.}
 \end{definition}

The  elements 
\begin{equation}\label{eq_def_phikp}
\bm\phi_{k,p} := \sum_{j \geq 0} c_{k,p,j}\dblG^j \wedge \dblg^{k-2j} =g_{k,p}(\dblG,\dblg) \in \mathcal G_k(n),\qquad  2p\leq k
\end{equation}
will play an important role,
where 
\begin{align*}
 g_{k,p}(s,t) & : =\sum_{j\ge 0} c_{k,p,j} s^j t^{k-2j} \in \CC[s,t]\\
 c_{k,p,j} & :=(-1)^{p+j}2^j \sum_{l\geq 0} \binom{k+1}{2p+2l+1}   \binom{p+l}{l}\binom{p+l}{j}.
\end{align*}

An alternative description of the constants in terms of generating functions is given by the following lemma.

\begin{lemma} \label{lemma:GenFunc2}
$$ G(x,y,z):= \sum_{k,p,j\geq 0} c_{k,p,j} x^j y^p z^k = \bigg[(1-z)^{2}-   (1-2x)(1-y)z^{2}\bigg]^{-1}.$$
\end{lemma}

\begin{proof}
Using the binomial series
\begin{displaymath}
\sum_{n\geq k} \binom{n}{k} x^n = \frac{x^k}{(1-x)^{k+1}}
\end{displaymath}
we compute 
\begin{align*}
\sum_{k,p,j\geq 0} c_{k,p,j} x^j y^p z^k & =  \sum_{k,p,j\geq 0}  (-1)^{p+j} 2^j  \sum_{l\geq 0} \binom{k+1}{2p+2l+1} \binom{p+l}{l}\binom{p+l}{j} x^j y^p z^k \\
 & =  \sum_{l,p,j\geq 0}(-1)^{p+j} 2^j   \binom{p+l}{l}\binom{p+l}{j} x^j y^p \frac{z^{2p+2l}}{(1-z)^{2p+2l+2}}\\
 & = \sum_{l,p\geq 0}(-1)^{p}    \binom{p+l}{l} (1-2x)^{p+l} y^p \frac{z^{2p+2l}}{(1-z)^{2p+2l+2}}\\
& =(1-z)^{-2} \sum_{r,l\geq 0} (-1)^{r-l} \binom{r}{l} (1-2x)^{r} y^{r-l} \frac{z^{2r}}{(1-z)^{2r}}\\
& = (1-z)^{-2} \sum_{r\geq 0}   \bigg[  \frac{(1-2x)(1-y)z^{2}}{(1-z)^{2}}\bigg]^r\\
& =\bigg[(1-z)^{2}-   (1-2x)(1-y)z^{2}\bigg]^{-1}.
\end{align*} 
\end{proof}

The generating function $G$ yields the following relations among the $g_{k,p}$:

\begin{lemma}
 \begin{equation}\label{eq:recRelation}  
     (k+2) t g_{k,p}(s,t) = 2(p+1) g_{k+1,p+1}(s,t)+ (k-2p+1) g_{k+1,p}(s,t),
   \end{equation} 
  \begin{equation} \label{eq:recRelation2}  
  (j+1)c_{k,p,j+1}+(k-2j)c_{k,p,j}=(k+1)c_{k-1,p,j}.
 \end{equation}
 \end{lemma}
 
 \begin{proof}
 Note that
 \begin{displaymath}
 \sum_{k,p\geq 0} g_{k,p}(s,t) y^pz^k = G\left(\frac s{t^2}, y, tz\right),\quad g_{0,p} = 0, \ p>0. 
 \end{displaymath}
Thus the relations \eqref{eq:recRelation} are equivalent to the functional identity
$$
\left(2t + t z\frac\partial{\partial z}\right)G(s/t^2, y, tz)  = \left( \frac {2(1-y)} z\frac\partial{\partial y}+   \frac\partial{\partial z} \right)G(s/t^2, y, tz) ,
$$ 
which is easily checked.
 
Relations \eqref{eq:recRelation2} follow by equating  coefficients in the identity
\begin{displaymath}
 (1-2x) \frac{\partial G}{\partial x}  =z(z-1) \frac{\partial G} {\partial z}  +2z G.
\end{displaymath}
\end{proof}

The Gray algebra interacts nicely with K\"ahler curvature tensors:

\begin{theorem}\label{thm_int_term} Suppose $\dblR \in \calK(\CC^n)$. 
Then, for $p+q\leq n$,
\begin{displaymath}
 \dblG^p \wedge  \dblR^q \wedge \dblg^{2(n-p-q)} =  2^{p} \frac{ C_{n-p-q}}{ C_{n-q}} \dblR^{q} \wedge \dblg^{2(n-q)},
\end{displaymath}
 where 
 \begin{displaymath}
   C_n = \frac{(2n)!}{n!(n+1)!}
 \end{displaymath}
 denotes the $n$-th Catalan number.
\end{theorem}

\begin{proof}
Note that  by Lemma~\ref{lemma:contraction}
\begin{align*} 
\dblG^p \wedge  \dblR^q \wedge \dblg^{2(n-p-q)} & =(-1)^p \frac{q!}{(p+q)!}\left.\frac{d^p}{d\lambda^p}\right|_{\lambda=0} (\dblR-  \lambda \dblG)^{p+q} \wedge \dblg^{2(n-p-q)} \\
 & = (-1)^p \frac{q!(2n-2p-2q)!}{(p+q)!(2p+2q)!}\left.\frac{d^p}{d\lambda^p}\right|_{\lambda=0} C^{2(p+q)}((\dblR-  \lambda \dblG)^{p+q} ).
\end{align*}

Given $\dblR \in \calK(\CC^n)$, the {\it total Chern form} $\gamma\in \bw^{\bullet}(\CC^n)^*$ of $\dblR$ is defined by
\begin{displaymath}
\gamma:= \det \left( I -\frac{1}{2\pi\sqrt{-1}} \Xi\right),
\end{displaymath}
where 
\begin{displaymath}
 \bw^2 (\CC^n)^* \owns \Xi_{ij}:=\dblR( \cdot,\cdot\,;e_{2i},e_{2j})-\sqrt{-1} \dblR(\cdot ,\cdot\, ;e_{2i},e_{2j+1}), \quad i,j=0,\ldots,n-1,
\end{displaymath}
see \cite{gray_book}*{p. 88}. Denote the  degree $2k$ component of $\gamma$ by $\gamma_k$. 
According to \cite[Lemma~7.11]{gray_book} we have 
\begin{displaymath}
\frac{(n-p-q)!}{(p+q)!(2p+2q)!} C^{2(p+q)}((\dblR-  \lambda \dblG)^{p+q} )=  (2\pi)^{p+q} ( \widetilde \gamma_{p+q} \wedge F^{n-p-q} )(e_1,\ldots, e_{2n}),
\end{displaymath}
where 
\begin{displaymath}
\widetilde \gamma_{p+q}= 
\sum_{k=0}^{n} \binom{n -k+1}{p+q-k} \left( - \frac{\lambda}{\pi} F\right)^{p+q-k} \wedge \gamma_k
\end{displaymath}
by \cite[Lemma~7.15]{gray_book}.

Therefore
\begin{align*}\dblG^p &\wedge  \dblR^q \wedge \dblg^{2(n-p-q)}\\ & = (-1)^p (2\pi)^{p+q} \frac{q!(2n-2p-2q)!}{(n-p-q)!}\left.\frac{d^p}{d\lambda^p}\right|_{\lambda=0} ( \widetilde \gamma_{p+q} \wedge F^{n-p-q} )(e_1,\ldots, e_{2n})  \\ 
 & = 2^{p+q}  \pi^q \frac{q!(2n-2p-2q)!(n-q+1)!}{(n-p-q)!(n-p-q+1)!}(\gamma_q\wedge  F^{n-q})(e_1,\ldots, e_{2n}).
\end{align*}
Since
\begin{align*}
(2\pi)^{q} (\gamma_q\wedge  F^{n-q})(e_1,\ldots, e_{2n}) & =  \frac{(n-q)!}{q! (2q)!} C^{2q}(\dblR^q)\\
 & =  \frac{(n-q)!}{q! (2n-2q)!} \dblR^{q} \wedge \dblg^{2(n-q)} 
\end{align*}
by \cite[Lemma~7.6]{gray_book} and Lemma~\ref{lemma:contraction}, the claim follows.
\end{proof}


\subsection{Structure of the algebra $\mathcal G(n)$}\label{sub:gray_algebra}

 We will need to determine the structure of the Gray algebra. 
  For every   finite-dimensional real vector space $U$,	denote by 
  $K^p(U)$ the kernel of the linear map 
 \begin{displaymath}
 \ \Sym^2(\largewedge^p U^* ) \to \largewedge^{p+1} U^* \otimes \largewedge^{p-1} U^*, \quad A\mapsto A'.
 \end{displaymath}
 Thus $K^2(U) = \mathcal R(U)$. Put $K_\CC^p(U) := K^p(U) \otimes_\RR \CC$ 

Put 
\begin{align}\label{def:g}
g_{k}(s,t):= g_{k,0}(s,t) &= \sum_{i \geq 0} (-1)^i 2^{k-i} \binom{k-i}{i}s^it^{k-2i},\\
\label{eq:GenFunc1}\sum_{k\geq 0} g_k(s,t)x^k  &= \left(1- 2 tx +2s x^2 \right)\inv.
\end{align}
 These identities follow by setting $x=s/t^2$, $y=0$, and  $z=tx$  in Lemma~\ref{lemma:GenFunc2}. Note that the polynomials $g_k$ satisfy
 $$ g_0=1, \quad g_1= 2t,$$ 
 and the relation
 \begin{equation}
  \label{eq:relation_gk}  2s  g_k -2 t g_{k+1} + g_{k+2}=0
 \end{equation}
 for $k\geq 0$.

 \begin{theorem} \label{thm:StructureGray} Let $p=0,1,\ldots, 2n$. 
  \begin{enumerate}
   \item[a)]  The  { component $\mathcal G_p(n)$} of the Gray algebra coincides with the subspace of $\U(n)$-invariant elements of $K_\CC^p(\CC^n)$, { with}
 \begin{equation}\label{eq:dimG}\dim \mathcal G_p(n) =  1 + 
 \left\lfloor \frac{ \min(p, 2n-p)}{2}\right\rfloor.
 \end{equation}
   \item[b)]  The multiplication pairing $P\colon \mathcal G_p(n)\times \mathcal G_{2n-p}(n)\to \mathcal G_{2n}(n)
  \simeq \CC$ is perfect.
  \item[c)]  Consider the graded polynomial algebra $\C[s,t]$, where $\deg t=1$ and $\deg s=2$. Then the map $\varphi\colon \C[s,t]\to \mathcal G(n)$ of graded algebras given by 
 \begin{displaymath}
 \varphi(h(s,t) )= h(\dblG,\dblg)
 \end{displaymath}
 covers an isomorphism
 \begin{displaymath}
  \C[s,t]/(g_{n+1}, g_{n+2}) \cong \mathcal G(n).
 \end{displaymath}
 \end{enumerate}
 \end{theorem}

The rest of this section is devoted to the proof.
 
 \begin{remark} 
 	It follows from c) above that the  Gray algebra $\mathcal G(n)$ is isomorphic to the  cohomology algebra of the Grassmann manifold $\mathrm{Gr}_2(\C^{n+2})$:
by \cite[p.\ 152]{fulton97}, \cite{fulton98} the cohomology algebra of $\mathrm{Gr}_2(\C^{n+2})$ is isomorphic to $\C[\sigma_2,\sigma_1]/(\tilde g_{n+1},\tilde g_{n+2})$, where 
\begin{displaymath}
\sum_{k\geq 0} \tilde g_k(\sigma_2,\sigma_1)x^k =\frac{1}{1-x\sigma_1+x^2\sigma_2}. 
\end{displaymath}
Hence $t \mapsto \frac12 \sigma_1, s \mapsto \frac12 \sigma_2$ covers an isomorphism $\mathcal G(n) \to H^*(\mathrm{Gr}_2(\C^{n+2}))$.
\end{remark}

Let $\SO(n,F)$ and $\SL(n,F)$ denote the orthogonal and special linear groups over a field $F$.  Recall that the finite dimensional representations of  $\SO(n,\RR)$ and $\SL(n,\RR)$ are in bijection with the finite dimensional analytic representations of  $\SO(n,\CC)$ and $\SL(n,\CC)$, since the  latter groups are the complexifications of the former (see, e.g., \cite[Chapter 24]{bump13}).
In the following we denote by $\varpi_p$, $p=0,1,\ldots,n-1$ the fundamental weights of $\SL(n,\C)$, i.e., the highest weights of the irreducible representations $\largewedge^p \C^n$.  

\begin{lemma} 
For $p=1,\ldots, n-1$, $K^p_\CC(\RR^n)$ is irreducible as a representation of $\SL(n,\C)$ and has highest weight $2 \varpi_{ n-p}$.
\end{lemma}

\begin{proof} Fix a basis $e_1,\ldots, e_n$ of $W=\RR^n\otimes_\RR \C$ and let $\alpha_1,\ldots \alpha_n$ denote the dual basis.
Consider the diagram 
\begin{center}
\begin{tikzpicture}
  \matrix (m) [matrix of math nodes,row sep=3em,column sep=4em,minimum width=2em]
  {
    \largewedge^p W^* \otimes \largewedge^p W^*   &  \largewedge^{p+1} W^* \otimes \largewedge^{p-1} W^* \\
    \largewedge^p W^* \otimes \largewedge^{n-p} W  & \largewedge^{p+1} W^* \otimes \largewedge^{n-p+1} W\\};
  \path[-stealth]
    (m-1-1) edge node [right] {$f_0$}  (m-2-1)
            edge node [above] {${}'$} (m-1-2)
    (m-1-2) edge node [right] {$f_1$}  (m-2-2) 
     (m-2-1) edge node [above] {$(-1)^{p+1} L$}(m-2-2);
\end{tikzpicture}
\end{center}
where the vertical arrows are induced by the $\SL(n,\C)$-equivariant isomorphisms $\largewedge^k W^* \cong \largewedge^{n-k} W$ and 
\begin{displaymath}
 L(\xi \otimes x)=   \sum_{i=1}^n (\alpha_i\wedge \xi) \otimes (e_i\wedge x),
\end{displaymath}
{analogous to multiplication by the symplectic form.}
We claim that the  above diagram is commutative.

Let $\phi\otimes \psi \in \largewedge^p W^* \otimes \largewedge^p W^*$. A direct computation shows that
\begin{equation}\label{eq:contr_basis}
 (\phi \otimes \psi)'= \sum_{i=1}^n  (\alpha_i \wedge \phi) \otimes (\iota_{e_i} \psi).
\end{equation}

Let $I=(i_1,\ldots, i_p)$ and $J=(j_1,\ldots, j_p)$ be $p$-tuples with $1\leq i_1<\cdots < i_p\leq n$ and 
$1\leq j_1<\cdots < j_p\leq n$ and let us write $\alpha_I$ for $\alpha_{i_1}\wedge \cdots \wedge  \alpha_{i_p}$. 
We denote  by  $J^c$ the complement of $J$, ordered in such a way that $(J,J^c)$ is an even permutation. Then 
\begin{align*} f_1( (\alpha_I \otimes \alpha_J)') &  = \sum_{i=1}^n  f_1( \alpha_i \wedge \alpha_I \otimes \iota_{e_i} \alpha_J)\\
& = (-1)^{p+1} \sum_{i=1}^n  \alpha_i \wedge \alpha_I \otimes {e_i}\wedge e_{J^c} \\
& = (-1)^{p+1} L(  \alpha_I \otimes  e_{J^c}) \\
& = (-1)^{p+1}  L(f_0(  \alpha_I \otimes  \alpha_{J})).
\end{align*}
This proves that the diagram is commutative.

By  symplectic linear algebra, the map $L$ is surjective. As $\SL(n,\C)$-representations, the spaces
$\largewedge^p W^* \otimes \largewedge^{n-p} W$ and $\largewedge^{p+1} W^* \otimes \largewedge^{n-p+1} W$ 
differ by precisely one irreducible representation, namely the one with highest weight $2 \varpi_{n- p}$, see, e.g., \cite{wannerer_angular}*{Lemma 2.9} for a proof. It follows that the kernel of $L$  coincides with this irreducible subrepresentation. The commutativity of the diagram now implies the lemma.
\end{proof}

\begin{remark}
The kernel of the map ${}': \largewedge^p W^* \otimes \largewedge^p W^*  \to  \largewedge^{p+1} W^* \otimes \largewedge^{p-1} W^*$ is irreducible and contains the kernel of the restriction 
 ${}': \Sym^2 \largewedge^p W^*   \to  \largewedge^{p+1} W^* \otimes \largewedge^{p-1} W^*$, hence both kernels must coincide. This 
shows that if a double form $\omega\in \largewedge{}^{p,p}(\RR^n,\RR^n)$ satisfies $\omega'=0$, then it is also symmetric, i.e.\ $\omega^\vee =\omega$. For a different proof of this fact see \cite{gray69}*{Proposition (2.3)}.
\end{remark}

\begin{lemma} \label{lemma:Kp_iso} For $p=1,\ldots, n-1$,  
$K^{p}_\CC(\RR^n)$ and $K^{n-p}_\CC(\RR^n)$ are isomorphic  as representations of $\SO(n,\CC)$. 
\end{lemma}
\begin{proof}
Fix a basis $e_1,\ldots, e_n$ of $W=\RR^n \otimes_\RR \C$ and let $\alpha_1,\ldots \alpha_n$ denote the dual basis.
Let $*\colon \largewedge^{k} W^*\to \largewedge^{n-k} W^*$ denote the Hodge star operator, normalized so that $\phi \wedge {*}\psi = \langle \phi,\psi\rangle \vol$. Moreover, let $I_{a,b} \colon \largewedge^a W^* \otimes \largewedge^b W^*\to 
\largewedge^b W^* \otimes \largewedge^a W^*$ be the map that interchanges the first and second factor. 
Since 
\begin{align*}* (\phi\wedge  \alpha_i ) & = \iota_{e_i} {*}\phi,\\ 
* (\iota_{e_i} \psi ) & = (-1)^{n-1} {*}\psi\wedge \alpha_i,
\end{align*}
a comparison with \eqref{eq:contr_basis} shows that  
the diagram
\begin{center}
\begin{tikzpicture}
  \matrix (m) [matrix of math nodes,row sep=3em,column sep=4em,minimum width=2em]
  {
    \largewedge^p W^* \otimes \largewedge^p W^*   &  \largewedge^{p+1} W^* \otimes \largewedge^{p-1} W^* \\
    \largewedge^{n-p} W^* \otimes \largewedge^{n-p} W^*  & \largewedge^{n-p+1} W^* \otimes \largewedge^{n-p-1} W^*\\};
  \path[-stealth]
    (m-1-1) edge node [right] {$(*\otimes *) \circ I_{p,p}$}  (m-2-1)
            edge node [above] {${}'$} (m-1-2)
    (m-1-2) edge node [right] {$(-I_{n-p-1,n-p+1})\circ (*\otimes *)$}  (m-2-2) 
     (m-2-1) edge node [above] {$'$}(m-2-2);
\end{tikzpicture}
\end{center}
commutes. Hence $(*\otimes *) \circ I_{p,p} \colon K^p_\CC(\RR^n)\to K^{n-p}_\CC(\RR^n)$ is an isomorphism intertwining the action of $\SO(n,\CC)$.
\end{proof}

\subsection{Proof of  Theorem  \ref{thm:StructureGray}}

\begin{proof}[Proof of Theorem~\ref{thm:StructureGray}, items (a) and (b)]
Let  $\operatorname{Res}^{\SL(n,\C)}_{\SO(n,\C)}$ denote the restriction of representations functor.  We denote by $\Gamma_\nu$ the irreducible representation with highest weight $\nu$. It is well-known, see, e.g., \cite{wannerer_angular}*{Lemma 2.10}, that, whenever $p\leq n/2$,
\begin{displaymath}
 \operatorname{Res}^{\SL(n,\C)}_{\SO(n,\C)} \Gamma_{2\varpi_p}  = \bigoplus_\lambda \Gamma_\mu,
\end{displaymath}
where the sum is over all highest weights of $\SO(n,\C)$ of the form
\begin{equation}\label{eq:weights}
 \mu= (\underbrace{2,\ldots, 2}_{k \text{ times}},0,\ldots ,0)
\end{equation}
with $0\leq k\leq p$. Using Lemma~\ref{lemma:Kp_iso}, we conclude that $K^p_\CC(\RR^n)$ decomposes as a representation of $\SO(n,\CC)$ into irreducible representations of weight \eqref{eq:weights} with $0\leq k\leq \min(p, n-p)$.

Next we recall the fact that the $\SO(2n,\RR)$-representation $\Gamma_{\mu}$ contains a non-zero vector that is fixed under the action of the unitary group $\U(n)\subset \SO(2n,\RR)$ if and only if $k$ is even. This can be deduced from the Cartan-Helgason theorem for the symmetric pair $(\SO(2n,\RR),\U(n))$  or from the branching rule for this pair. We refer the reader to \cite[p. 269]{alesker_mcullenconj01} or \cite[Section 2]{boroczky_domokos_solanes} for further discussion and earlier applications of this fact to valuation theory.

Combining the above two facts   we can thus  conclude that
\begin{displaymath}
  \dim K^p_\CC({\C^{n}}) ^{\U(n)}= 1+ \left\lfloor \frac{ \min(p, 2n-p)}{2}\right\rfloor.
\end{displaymath}

Since  $\mathcal G_p(n) \subset K^p_\CC(\C^{n}) ^{\U(n)}$ as $\dblG'=0,g'=0$, we obtain that the right-hand side of \eqref{eq:dimG} is an upper bound on the dimension of $\mathcal G_p(n)$. 

Assume first that $p=2m\leq n$, and consider the pairing matrix for the elements
{ 
\begin{displaymath}
2^{-m}\dblG ^m, \ 2^{1-m} \dblG^{m-1} \dblg^{2}, \ldots,\ \dblg^{2m}\in\mathcal G_{2m}(n),
\end{displaymath}
\begin{displaymath}
 2^{m-n}\dblG ^{n-m},\ 2^{m+1-n}  \dblG^{n-m-1} \dblg^2, \ldots,\  \dblG^{n-2m}\dblg^{2m}\in\mathcal G_{2n-2m}(n),
\end{displaymath}
}
viz. 
\begin{displaymath} 
 C_n\inv \begin{pmatrix} C_0 & C_1 & \cdots & C_m\\
          C_1 & C_2 & \cdots & C_{m+1}\\
          \vdots  &  & \cdots& \vdots \\
          C_{m} & C_{m+1} & \cdots & C_{2m} 
         \end{pmatrix}
 \end{displaymath}
 by the case $q=0$ of Theorem~\ref{thm_int_term}. This matrix is nonsingular; in fact, dropping the factor $C_n\inv$,
 it is well known to have determinant 1 \cite{aigner99}. The perfectness assertion (b) and the linear independence of the considered elements of $\mathcal G_{p}(n)$ follows.

The case of odd $p$ is treated in the same way using the identity
  \begin{displaymath}
   \det \begin{pmatrix} C_1 & C_2 & \cdots & C_m\\
          C_2 & C_3 & \cdots & C_{m+1}\\
          \vdots  &  & \cdots& \vdots \\
          C_{m} & C_{m+1} & \cdots & C_{2m-1} 
         \end{pmatrix}=1.
  \end{displaymath}

\end{proof}

To complete the proof of Theorem~\ref{thm:StructureGray}, item (c), we will need the following combinatorial identity.
 
\begin{lemma}\label{lem:sum_catalan}
For all  non-negative integers $k,n$,
\begin{displaymath}
  \sum_{i\geq 0} (-1)^{i} \binom{n+1-i}{i} C_{n-k-i} =(-1)^{n-k}\binom{k}{n-k}.
\end{displaymath}
In particular, the left-hand side is zero if $n>2k$. 
 \end{lemma}
 
 \begin{proof}
  In terms of generating functions, 
  we wish to show that 
  \begin{displaymath}
   \sum_{n,k,i\geq 0}(-1)^{i} \binom{n+1-i}{i} C_{n-k-i}  x^n y^k= \frac{ 1}{1-xy(1-x)}.
  \end{displaymath}
  Indeed, using that  
   \begin{displaymath}
     \sum_{n\geq0} C_n x^n = \frac{1-\sqrt{1-4x}}{2x},
   \end{displaymath}
   the left-hand side may be expressed as 
  \begin{align*}
   \sum_{n,k,i\geq 0}(-1)^{i} \binom{n-i+1}{i} C_{n-k-i}  x^n y^k & =\sum_{m,k,i\geq 0}(-1)^{i} \binom{m+k+1}{i} C_{m}  x^{m+k+i} y^k \\
   & =(1-x)\sum_{m,k\geq 0} C_{m}  (1-x)^{m+k} x^{m+k} y^k \\
   & = \frac{1-\sqrt{1-4x(1-x)}}{2x(1-xy(1-x))}\\
    & = \frac{1}{1-xy(1-x)}.
  \end{align*}
 \end{proof}

\begin{proof}[Proof of Theorem~\ref{thm:StructureGray}, item (c)] Let us first show that $g_{n+1}(\dblG,\dblg)=g_{n+2}(\dblG,\dblg)=0$ in $\mathcal G(n)$. 
By \eqref{def:g}, Theorem \ref{thm_int_term},  and Lemma \ref{lem:sum_catalan} we have
 \begin{align*} P( \dblG^k  \dblg^{n-2k-1}, g_{n+1}( \dblG, \dblg)) & = 2^{n+1} \sum_{i\geq 0} (-2)^{-i} \binom{n+1-i}{i} P(
 \dblG^k \dblg^{n-2k-1},  \dblG^i  \dblg^{n+1-2i}) \\
 & = \frac{2^{n+k+1}}{C_n} \left( \sum_{i\geq 0} (-1)^{i} \binom{n+1-i}{i} C_{n-k-i}\right) \dblg^{2n} \\
 &= 0
 \end{align*}for $n>2k$. 
 Since the pairing is perfect and the double forms $ \dblG^k \dblg^{n -2k-1}$, $0\leq 2k<n$, span $\mathcal G_{ n-1}(n)$, the relation $g_{n+1}( \dblG, \dblg)=0$ follows. The same argument shows $g_{n+2}( \dblG, \dblg)=0$. 
 
 Since $t\cdot g_{n+1}\neq g_{n+2}$, repeating the argument of \cite{fu06}*{Lemma 3.7} shows that the ideal of relations in $\mathcal G (n)$ is generated by $g_{n+1}( \dblG, \dblg)$ and $g_{n+2}( \dblG, \dblg)$. 
  \end{proof}

 Theorem \ref{thm:StructureGray} has the following crucial consequences.  
 
   \begin{proposition}\label{prop:g_kp_zero} If $k-p> n$ then, as an element of $\mathcal G(n)$, 
  $$\bm \phi_{k,p} =0.$$ 

 \end{proposition}
 \begin{proof}
 Relation \eqref{eq:relation_gk}  and part c) of Theorem~\ref{thm:StructureGray}  imply  the statement for $p=0$. The case $p>0$ is proved by induction. Assume $p>0$ and $k-p>n$  and that the proposition has been proved for $p-1$. Relation \eqref{eq:recRelation} and the inductive assumption imply
  \begin{displaymath}
  \bm \phi_{k,p}=g_{k,p}(\dblG,\dblg) = \frac{k+1}{2p} \dblg g_{k-1,p-1}( \dblG, \dblg) -\frac{k-2p+2}{2p} g_{k,p-1}( \dblG, \dblg) =0.
  \end{displaymath} 
 \end{proof}

\begin{proposition} \label{prop_poincare_gray} Let $\dblR \in \calK(\C^n)$. Then, for even $k$,
\begin{displaymath}
  \bm\phi_{k,p}\wedge \dblR^{n-\frac{k}2} = 
  \begin{cases}
                      (2p+1) \dblR^{n-p}\wedge \dblg^{2p}, &  k=2p,\\
                      0, &  k >2p.
                     \end{cases}
 \end{displaymath}
\end{proposition}

\begin{proof}
By Theorem \ref{thm_int_term},
 \begin{align}
    \bm\phi_{k,p}\wedge \dblR^{n-\frac{k}2} &= \sum_{j=0}^{ \frac{k}2} c_{k,p,j} \dblG^j\wedge \dblg^{k-2j}\wedge \dblR^{n-\frac{k}2}\\
  &=\left(\sum_{j=0}^{ \frac{k}2} c_{k,p,j} 2^j \frac{C_{\frac{k}2-j}}{C_{\frac{k}2}} \right)\dblg^{k}\wedge \dblR^{n-\frac{k}2}.\label{eq:phi_R}
 \end{align}
The numerical factor in the last line  is independent of $n$. Denote it by $A_{k,p}$. We claim that $A_{k,p}= 0 $ if $k>2p$ and $= 2p+1$ if $k=2p$. To prove this it is enough to consider the case $n= \frac k 2$, so that the final expression above is $A_{k,p} \dblg^{2n}$, where $\dblg^{2n} \ne 0$.

If $k>2p$ then Proposition \ref{prop:g_kp_zero} implies that then $\bm \phi_{k,p}=0$ as an element of $\mathcal G(n)$, so $A_{k , p} = 0$.

For the case $k=2p=2n$ we recall that  $\bm\phi_{2p,q}=0$  for $q<p$ by Proposition \ref{prop:g_kp_zero}. Since
\begin{equation} \label{eq:sum_p}
\sum_p c_{k,p,j}= \begin{cases}
                      0, &  j>0,\\
                      k+1, & j=0
                  \end{cases}
\end{equation}
by Lemma~\ref{lemma:GenFunc2}, we conclude
\begin{align*}
  \bm \phi_{2p,p}&=\sum_{q=0}^p   \bm\phi_{2p,q}\notag  \\
 &= \sum_{q,j=0}^p c_{2p,q,j} \dblG^j \wedge \dblg^{2(p-j)}\\
 &= (2p+1) \dblg^{2p}.
 \end{align*} 
\end{proof}

\section{The K\"ahler-Lipschitz-Killing algebra}\label{sec_alg structure}

We assemble the ingredients above to prove Theorems A and B by giving the claimed isomorphisms in explicit terms, using
the general hermitian intrinsic volumes  $\mu^M_{k, p}$ defined in \eqref{eq_def_mu_kq} below.  
Theorem A is a direct consequence of Proposition \ref{thm_weyl} and Theorem \ref{thm:alg_isom}.  Theorem B follows from Proposition \ref{thm_weyl} and the results of Section \ref {sec_module}.

\subsection{ The Kähler-Lipschitz-Killing curvature measures and the hermitian intrinsic volumes}

\begin{definition}\label{def:hivs}
Given a Kähler manifold $M$ we define the {\em  { Kähler-Lipschitz-Killing} curvature measures} on $M$ by
\begin{align*}
\Delta^M_{k,p}  & := \sum_j c_{k,p,j}\Gamma_{k,j}= [\bm \kappa_{k,p},{\bm\omega}_{k,p}], \quad 0 \leq p \leq \frac{k}{2} \leq n\\
B^M_{k,p} & :=\sum_j c_{k-1,p,j} \frac{k+1}{k-2p}\tilde \Gamma_{k,j}= [0,\tilde{\bm \omega}_{k,p}], \quad 0 \leq p <\frac{k}{2}  < n
\end{align*}
where
\begin{align}
\bm\omega_{k,p} &:=\bm\alpha^\vee \wedge  \bm\phi_{k,p} \wedge  \bm\psi_{2n-k-1}\label{eq_def_omega_kp}\\
\tilde {\bm\omega}_{k,p} &: =\frac{k+1}{k-2p}\bm\alpha^\vee \wedge \bm\beta \wedge\bm\beta^\vee 
\wedge  \bm\phi_{k-1,p} \wedge  \tilde {\bm\psi}_{2n-k-1}\\\label{eq_def_kappa_k}
\bm \kappa_{k,p}&:=d_{2n,k,n-\frac{k}2}\bm\phi_{k,p}\wedge\dblR^{n-\frac{k}2}=\delta_{k, 2p}\cdot \bm \kappa_k
\end{align}
 by Proposition \ref{prop_poincare_gray}.

Define also the \emph{hermitian intrinsic volumes} on $M$ to be the valuations
\begin{equation}\label{eq_def_mu_kq}
\mu_{k,p}^M:=[\Delta_{k,p}^M] \in \mathcal V(M).
\end{equation}
We denote by $\mathcal{KLK}(M) \subset \mathcal V(M)$ the vector space spanned by the $\mu_{k,p}^M$,  and by $\widetilde{\mathcal{KLK}}(M)\subset \calC(M)$ the space spanned by the $\Delta_{k,p}^M$ and the $B_{k,p}^M$.
\end{definition}

Theorem \ref{thm_restrict_gamma} implies that the $\Delta_{k, p}$ and $B_{k, p}$ are invariant under restriction.

\begin{proposition} \label{thm_weyl}
 Let $M \subset \tilde M$ be an isometric holomorphic embedding of K\"ahler manifolds. Then 
 \begin{align*}
  \Delta_{k,p}^{\tilde M}|_M & =\Delta_{k,p}^M\\
  B_{k,p}^{\tilde M}|_M & =B_{k,p}^M.
 \end{align*}
\end{proposition}

We will see  (Proposition \ref{prop_glob_delta_beta})  that  $[B_{k,p}^M]=[\Delta_{k,p}^M]=\mu_{k,p}^M$, and (Theorem \ref{thm:alg_isom}) that  $\mathcal{KLK}(M)$ is closed with respect to Alesker multiplication. The resulting subalgebra of $\calV(M)$ is the {\it  K\"ahler-Lipschitz-Killing algebra} of $M$.

These curvature measures fit the previously developed schema from \cite{bernig_fu_hig, bernig_fu_solanes}:

\begin{proposition}\label{prop_B_eq_B}
Taking $M=\C^n$,
$$
\Delta_{k,p}^{\CC^n} = \Delta_{k,p},\quad B_{k,p}^{\CC^n}= B_{k,p},
$$ 
where the right hand sides are the curvature measures  of $\Curv_k^{\U(n)}$ as defined in \cite{bernig_fu_hig, bernig_fu_solanes}.
\end{proposition}

\begin{proof}
	
We show first that $\Delta_{k,p}=\Delta_{k,p}^{\CC^n}$. Consider in complex projective space $\CP_\lambda^N$ of constant holomorphic curvature $4\lambda\neq 0$ the image of $\Delta_{k,p}\in\Curv^{\U(n)}$ under the transfer map, which we denote again by $\Delta_{k,p}$.
By \cite[Eqs. (75), (82)]{fu_wannerer} we have 
\begin{align*}
			\Delta_{k,p} & = (-1)^p \left(\frac \pi 2\right)^{\lceil \frac k 2 \rceil} \frac{1 }{(k+1)!!} \sum_{j,l\geq 0}   \binom{k+1}{2p+2l+1} \binom{p+l}{l}\binom{p+l}{j}  (-\lambda)^{-j}C_{k,j}^\lambda\\
			& =  \frac{\pi^k}{(k+1)!\omega_k} \sum_{j\geq 0} c_{k,p,j}  (2\lambda)^{-j}  C_{k,j}^\lambda,
\end{align*}
where, under the identification \eqref{eq_omegan},
\begin{align*}
			C_{k,j}^\lambda&=\left[0,\frac{\omega_k}{\pi^k(2n-k)!\omega_{2n-k}}\bm\alpha^\vee\wedge (2\lambda \dblG)^j\wedge \dblg^{k-2j}\wedge\bm\omega^{2n-k-1}\right] & \mbox{for }k<2n,\\
			C_{2n,j}^\lambda&=\left[\frac{\omega_{2n}}{\pi^{2n}} (2\lambda \dblG)^j\wedge \dblg^{2n-j},0\right].
\end{align*}
		
Hence, for $k<2n$, { and recalling \eqref{eq_def_phikp},}
		\begin{equation}\label{eq:double_forms_delta_boundary}
			\Delta_{k,p}  =\frac{1}{(k+1)!(2n-k)!\omega_{2n-k}}\left[0,\bm\alpha^\vee\wedge\bm\phi_{k,p}\wedge \bm\omega^{2n-k-1}\right],
		\end{equation}
		while
\begin{equation}\label{eq:double_forms_delta_interior}
			\Delta_{2n,p} = \frac{1}{(2n+1)!} \left[ \bm\phi_{2n,p},0\right]=\frac{\delta_{p,n}}{(2n)!} \left[ \dblg^{2n},0\right],
\end{equation}
by Proposition \ref{prop_poincare_gray}.
		
On the other hand,   since $\dblR^{\CP_\lambda^{ n}}= \lambda \dblG$, 
		\begin{align*}
			\bm \kappa_{k,p}&=\delta_{2p,k} {\bm \kappa}_k=\frac{\delta_{2p,k}\delta_{k,2n}}{(2n)!}\dblg^{2n}+O(\lambda),\\
			\bm\omega_{k,p}&=\frac{1}{(k+1)!(2n-k)!\omega_{2n-k}} \bm\alpha^\vee\wedge\bm\phi_{k,p}\wedge\bm\omega^{2n-k-1}+O(\lambda).
		\end{align*}

		It follows for all $k,p$ that 
		\[
	 	\Delta_{k,p}^{\CP_\lambda^{ n}}  =\Delta_{k,p}+O(\lambda).
		\]
		To prove that $\Delta_{k,p}=\Delta_{k,p}^{\CC^n}$, it suffices to prove this identity inside the unit ball of $\CC^{n}$, where  we may approximate the flat metric by a family of metrics $g_\lambda$ compatible with the complex structure and with constant holomorphic curvature $4\lambda$. Letting $\lambda\to 0$ we have $\Delta_{k,p}^{\CP_\lambda^{ n}}\to\Delta_{k,p}^{\CC^n}$ which yields the claim.

Now we  prove $B_{k,p}^{\CC^n}=B_{k,p}$  for $2p<k<2n$. Consider on $S\CC^n=\CC^n\times S^{2n-1}$ the horizontal vector field $X$ defined by $ X_{(x,\xi)}=(J\xi,0)$.
Recall that $B_{k,p},\Delta_{k,p}\in \mathrm{Curv}^{\U(n)}$ are given by (cf. \cite[\S 3.1]{bernig_fu_solanes})
\begin{align*}
B_{k,p}&=[0,\bm\beta_{k,p}]\\
\Delta_{k,p}&=\frac{1}{2n-k}(2(n-k+p)[0,\bm\gamma_{k,p}]+(k-2p)[0,\bm\beta_{k,p}]),
\end{align*}
where $\bm\beta_{k,p}$ is a multiple of $\bm\beta$ and $\contr{X} \bm\gamma_{k,p}\equiv0$ modulo $\bm\alpha$. Comparing with \eqref{eq:double_forms_delta_boundary} it follows that,  under the identification \eqref{eq_omegan}
\begin{equation}\label{eq_beta_to_xi}
 \frac{1}{(k+1)! (2n-k-1)!\omega_{2n-k}}\bm\alpha^\vee \wedge \bm\phi_{k,p} \wedge \bm\omega^{2n-k-1}=2(n-k+p)\bm\gamma_{k,p}+(k-2p)\bm\beta_{k,p}+\bm\rho,
\end{equation} 
where $\bm\rho\in \Omega^{2n-1}(S\CC^n)^{\overline{\U(n)}}$ is in the ideal $I$ generated by $\bm\alpha,d\bm\alpha$. Since
\begin{displaymath}  
\contr{X}\bm\alpha=0,\qquad \contr{X} d\bm\alpha=-\omega_{1,0}=:-\bm\gamma
\end{displaymath}
we see that, 
\begin{displaymath}
\bm\beta\wedge( \contr{X} \bm\rho) \equiv \bm\beta\wedge \bm\gamma\wedge \bm T \quad \mod I,
\end{displaymath}
where $\bm T$ belongs to the space of translation invariant $\U(n)$-invariant elements of $\Omega^{2n-3}(S\CC^n)$. Since this space is generated by $\bm\alpha,\bm\beta,\bm\gamma$ and four forms of degree $2$,  while $(2n-3)$ is odd, we deduce $\bm\beta\wedge( \contr{X} \bm\rho)\in I$.

By \eqref{eq:dblG} we have
\begin{displaymath}
 \contr{X} \dblg=\bm\beta^\vee,\qquad \contr{X} \dblG\equiv\bm \beta^\vee\wedge \dblg\qquad \contr{X}\bm \omega=0,
\end{displaymath}
modulo $\bm\alpha, \bm\alpha^\vee$. Hence,
\begin{align*}
 \contr{X}\bm\phi_{k,p} &=\sum_j c_{k,p,j} (j\bm \beta^\vee \wedge \dblG^{j-1} \wedge \dblg^{k-2j+1} +(k-2j)\bm\beta^\vee \wedge \dblG^j \wedge \dblg^{k-2j-1})\\
 &=\sum_j ((j+1)c_{k,p,j+1}+(k-2j)c_{k,p,j})\bm\beta^\vee \wedge \dblG^j \wedge  \dblg^{k-2j-1}\\
 &=(k+1)\sum_j c_{k-1,p,j}\bm\beta^\vee \wedge \dblG^j \wedge \dblg^{k-2j-1}\\
 &=(k+1)\bm\beta^\vee \wedge \bm\phi_{k-1,p},
\end{align*}
where the third equality  follows from \eqref{eq:recRelation2}.

Applying $\bm\beta\wedge\contr{X}(\cdot)$ to both sides of \eqref{eq_beta_to_xi}, and using the previous relation, yields $B_{k,p}=B_{k,p}^M$.
\end{proof}

\subsection{Basic properties of the Kähler-Lipschitz-Killing curvature measures}
\label{sec_hermtian_intrinsic_volumes}

The two localizations $\Delta_{k,p}^M$ and $ B_{k,p}^M$ of the hermitian intrinsic volumes have different, particularly desirable properties with respect to the Alesker product. As will be proved in Proposition \ref{prop_module} below, multiplication by the first intrinsic volume $t = \frac 2 \pi \mu_{1,0}^M$ stabilizes the  family of $\Delta_{k,p}^M$, while multiplication by $\sigma_M:=\frac{1}{2\pi}(\mu_{2,0}^M+2\mu_{2,1}^M)$ stabilizes the $B_{k,p}^M$.

Initially defined for all $2p\leq k\leq 2n$ and $2p<k< 2n$, respectively, as a consequence of the relations in the Gray algebra, some of the curvature measures $\Delta^M_{k,p}$ and
 $B^M_{k,p}$  are actually trivial.

\begin{proposition}\label{prop:triviality}
On an $n$-dimensional K\"ahler manifold $M$, 
\begin{displaymath}
 \Delta_{k,p}^M=0 \quad \text{and} \quad B_{k,p}^M=0
\end{displaymath}
whenever $k-p>n$.
\end{proposition}

\begin{proof}
 Since $\bm\phi_{k,p}= { g_{k,p}(\dblG,\dblg)}=0$ whenever $k-p>n$ by Proposition~\ref{prop:g_kp_zero}, the first assertion is clear from the definition of $\Delta_{k,p}^M$. To prove the second assertion, observe that 
\begin{displaymath}
\bm\alpha \wedge \big(\bm\alpha^\vee \wedge \bm\beta\wedge\bm\beta^\vee \wedge \bm\phi_{k-1,p}\big)=  (\bm\alpha \wedge \bm\alpha^\vee) \wedge  \bm\beta \wedge  \bm\beta^\vee \wedge (\bm\phi_{k-1,p}|_{H})=0,
\end{displaymath}
 where $H\subset T_x M$ denotes the complex hyperplane $\ker\bm\alpha\cap\ker\bm\beta$. Thus $\bm\alpha^\vee \wedge  \bm\beta \wedge  \bm\beta^\vee \wedge \bm\phi_{k-1,p}$ is divisible by the contact form $\bm\alpha$ and consequently $B_{k,p}^M=0$.
\end{proof}

We may now clarify the relation between the $\Delta_{k,p}^M$ and the Lipschitz-Killing curvature measures $\Lambda_k^M$ of Lemma \ref{lem:LKF CMs}.
\begin{proposition} \label{prop:Delta_LK}
 $$ \sum_{p=0}^{\left\lfloor \frac{k}2\right\rfloor} \Delta^M_{k,p} =\Lambda_{k}^M.$$
\end{proposition}

\begin{proof} 
By \eqref{eq:sum_p}, for $k<2n$,
\begin{displaymath}
	\bm\alpha^\vee \wedge  \sum_p \bm\phi_{k,p} \wedge \bm\psi_{2n-k-1}  =  (k+1) \bm\alpha^\vee \wedge  \dblg^k \wedge \bm\psi_{2n-k-1},
 \end{displaymath}
which is the boundary term of $\Lambda_k^M$ from Lemma \ref{lem:LKF CMs}. Comparison of the 
interior terms is trivial.
\end{proof}

 \subsection{Structure of $\KLK(M)$}
In this section we prove our main result, Theorem \ref{thm:alg_isom}. The following will be helpful. 

\begin{lemma}\label{lem:filtration_delta} Given $x\in M$, let $\tau_x:\calC(M)\to\Curv(T_xM){ \simeq \Curv(\C^n)}$ denote the transfer map \eqref{eq_transfer_map}. Let $\pi_k\colon \Curv\to\Curv_k$  be the projection onto the $k$-th degree component. Then,  for $0,k-n\leq p\leq \frac{k}2\leq n$,
\[ \pi_k\circ \tau_x(\Delta_{k,p}^M)=\Delta_{k,p},\qquad \pi_k\circ \tau_x(B_{k,p}^M)=B_{k,p}.\]
\end{lemma}
\begin{proof}
Let us first treat the case $k=2n$. By  \eqref{eq_def_kappa_k},  $\Delta^M_{2n,n}$ is the riemannian volume measure on $M$, and thus $\tau_x(\Delta_{2n,n})=\Delta_{2n,n}$. Since $B^M_{2n,n}=0$, this finishes the case $k=2n$.

Let us now assume $k<2n$. In this case $\tau_x(\Delta_{k,p})=[(\bm \kappa_{k,p})_x,\bar\tau_x(\bm\omega_{k,p})]$, where $
  \bar\tau_x$ is the composition of \eqref{eq:sigma_double} with the identification 
  $\Omega^{\bullet,\bullet}(S\C^n,\C^n)|_{\{0\}\times S^{2n-1}}\simeq\Omega^{\bullet,\bullet}(S\C^n,\C^n)^{\C^n}.$ 

By \eqref{eq:sigma_omega} we have that $\bar\tau_x(\bm\omega)$ is the connection double form on $\CC^n$. Likewise, $\bar\tau_x(\dblg),\bar\tau_x(\dblG)$ are the metric and the Gray double forms on $\C^n$.
 By Proposition \ref{prop_B_eq_B}, it follows that $\Delta_{k,p}\in \Curv_k$ is represented by $d_{n,k,0}\bar\tau_x(\bm\alpha^\vee \wedge \bm \phi_{k,p} \wedge  \bm\omega^{2n-k-1})$.
Therefore, since $[0,\bar\tau_x(\bm\alpha^\vee \wedge \bm\phi_{k,p} \wedge \dblR^l \wedge \bm\omega^{2n-2l-k-1})]\in \Curv_{k+l}$, we have
 \begin{align*}
  \pi_k\circ\tau_x(\Delta_{k,p}^M)&=[0,d_{n,k,0}\bar\tau_x(\bm\alpha^\vee \wedge \bm\phi_{k,p} \wedge  \bm\omega^{2n-k-1})]=\Delta_{k,p}.
 \end{align*}
The proof of the second  stated equality is completely analogous.
\end{proof}

 In the following we consider the linear map $$r_M\colon \Val^{\U(n)} \to \mathcal{KLK}(M)$$ given by $r_M(\mu_{k,p})=\mu_{k,p}^M$. In Theorem \ref {thm:alg_isom} we will see that $r_M$ is an isomorphism of algebras. A first step is the following.

\begin{proposition}\label{prop:r_M_bijection}
 $r_M$ is a bijection.
\end{proposition}
\begin{proof}
Surjectivity of $r_M$ follows from the definition of $\KLK(M)$.
To prove that $r_M$ is injective, suppose that
 \begin{displaymath}
  \sum_{k= l}^{2n}\sum_qa_{k,q}\mu_{k,q}^M=0
 \end{displaymath}
with  $l$ maximal and $0,k-n\leq q\leq \frac{k}2$. With respect to the filtration $\mathcal W_0\supset \mathcal W_1\supset\dots\supset \mathcal W_{2n}$ of $\mathcal V(M)$ introduced in \cite{alesker_val_man2}, we have $\mu_{k,p}^M\in \mathcal W_k$ by \cite[Proposition 5.2.5]{alesker_val_man1}. By  Lemma \ref{lem:filtration_delta} and \cite[Corollary 3.6]{solanes_wannerer}, the projection $\mathcal W_l\to \mathcal W_l/\mathcal W_{l+1}=\Val_l(TM)$ maps $\mu_{l,q}^M\mapsto \mu_{l,q}$, while $\mu^M_{k,p}\mapsto 0$ for $k>l$. We deduce that $\sum_q a_{l,q}\mu_{l,q}=0$ in $\Val_l^{\U(n)}$, and thus all $a_{l,q}$ vanish, contradicting the maximality of $l$.
\end{proof}

As a second step toward Theorem \ref {thm:alg_isom}, we prove it for K\"ahler submanifolds of $\CC^N$.

\begin{proposition}[Algebra structure for embedded K\"ahler manifolds]\label{prop:embedded}
Let $M \subset \CC^N$ be a K\"ahler submanifold of complex dimension $n$. Then $\KLK(M)\subset \calV(M)$ is a subalgebra,  and  the linear map $r_M\colon \Val^{\U(n)}\to \mathcal{KLK}(M)$  defined by $r_M(\mu_{k,q})=\mu_{k,q}^M$, is an isomorphism of algebras. Moreover, restriction of valuations on $\CC^N$ to $M$ factors as
  \begin{center}
\begin{tikzpicture}
  \matrix (m) [matrix of math nodes,row sep=3em,column sep=4em,minimum width=2em]
  {
   &   \Val^{\U(N)}   &   \\
\Val^{\U(n)}    &   & \mathcal{KLK}(M)\\};
  \path[-stealth]
(m-1-2) edge node [above right] {$|_M$}  (m-2-3)
(m-1-2) edge node [above left] {$|_{\CC^n}$}   (m-2-1)
(m-2-1) edge node [above] {$r_M$} (m-2-3);
\end{tikzpicture}
\end{center}
\end{proposition}

\begin{proof}
 For $p\geq k-n$, we have $\mu_{k,p}|_{\C^n}=\mu_{k,p}$, and $\mu_{k,p}|_M=\mu_{k,p}^M$ by Proposition \ref{thm_weyl}. For $p<k-n$, both $\mu_{k,p}|_{\C^n}$ and $\mu_{k,p}|_M=\mu_{k,p}^M$ vanish by Proposition \ref{prop:triviality}.   Hence the previous diagram commutes.

Since restriction and Alesker product commute, the diagonal maps in the previous diagram are algebra morphisms.
Since the restriction $\Val^{\U(N)}\to \Val^{\U(n)}$ is surjective, it follows that  $r_M$ is also a morphism of algebras. Indeed, given $\phi,\nu\in\Val^{\U(N)}$ we have
\begin{align*}
 r_M(\phi|_{\C^n}\cdot\nu|_{\C^n})&=r_M((\phi\cdot \nu)|_{\C^n})\\
 &=(\phi\cdot \nu)|_{M}\\
 &=\phi|_M \cdot \nu|_M\\
 &=r_M(\phi|_{\C^n})\cdot r_M(\nu|_{\C^n}).
\end{align*}
Together with Proposition \ref{prop:r_M_bijection}, this concludes the proof.
\end{proof}

The pointwise embedding lemma (Theorem \ref {thm_pointwise_lemma}) implies that this is enough to prove
Theorem \ref{thm:alg_isom} in full generality.

\begin{theorem}[Structure of the K\"ahler-Lipschitz-Killing algebra]\label{thm:alg_isom}
Given any Kähler manifold $M$ of dimension $n$, the subspace $\KLK(M) \subset \calV(M)$ is a subalgebra, and the map 
$ r_M\colon \Val^{\U(n)}\to \KLK(M)$, given by 
$$r_M(\mu_{k,q})=\mu_{k,q}^M,$$
 is  an  isomorphism of algebras.
\end{theorem}

\begin{proof}
Let $\dblR\in\mathcal K(\C^n)$ be an embedded curvature tensor, and take a complex submanifold $S\subset\C^N$ such that $\dblR=\sigma^*\dblR^S_x$ for some $x\in S$ and some linear isometry $\sigma\colon \C^n\to T_xS$.  

By Proposition \ref{prop:embedded}, if $a_i$ are constants such that 
\[\mu_{k,p}\cdot\mu_{l,q}=\sum_i a_i\mu_{k+l,i}\] in $\Val^{\U(n)}$, then also
\begin{align*}
  \mu_{k,p}^S\cdot\mu_{l,q}^S&= \sum_i a_i \mu_{k+l,i}^S.
\end{align*} 
 
Recalling that $\Delta_{k,p}=\sum_j c_{k,p,j} \Gamma_{k,j}$, put
\begin{align*}
Q_{k,p}=\sum_j c_{k,p,j} P_{k,p},\qquad Q_{k,p,l,q}=\sum_{i,j} c_{k,p,i}c_{l,q,j} P_{k,i,l,j}
\end{align*}
where $P_{k,p}$ are the polynomials given in Proposition \ref{prop_rumin_in_cartan_calculus}, and $P_{k,p,l,q}$ are those from Proposition \ref{prop:polynomiality prime}. By these propositions, at $\xi=\sigma(e_0)$, we have
\begin{align*}
(\bar\sigma_\xi^*)^{-1} Q_{k,p,l,q}(\sigma^*\dblR_x^S)&= \tau(\mu_{k,p}^S\cdot \mu_{l,q}^S)_\xi\\
&= \sum_i a_i \tau(\mu_{k+l}^S)_\xi\\
&=\sum_i a_i  (\bar\sigma_\xi^*)^{-1}Q_{k+l,i}(\sigma^*\dblR_x^S)|_{\xi},
\end{align*} 
where $\tau$ is given by \eqref{eq:zeta_tau}. Hence
\begin{displaymath}
 Q_{k,p,l,q}(\dblR)=\sum_i a_i \left.Q_{k+l,i}(\dblR)\right|_{e_0}.
\end{displaymath}
By Theorem \ref{thm_pointwise_lemma}, the previous relation holds for all $\dblR\in\mathcal K$. In particular, again by Propositions \ref{prop:polynomiality prime} and \ref{prop_rumin_in_cartan_calculus}, for every $n$-dimensional K\"ahler manifold $M$,
\begin{displaymath}
 \tau(\mu_{k,p}^M\cdot\mu_{l,q}^M)=\sum_i a_i \tau(\mu_{k+l,i}^M).
\end{displaymath}
Since $\zeta(\mu^M_{k,p}\cdot\mu^M_{l,q})_x$ does not depend on $\dblR_x^M$ by \eqref{eq_ab_formula2}, we have  \[\zeta(\mu_{k,p}^M\cdot\mu_{l,q}^M)=\sum_i a_i \zeta(\mu_{k+l,i}^M),\] and thus
$ \mu_{k,p}^M\cdot\mu_{l,q}^M=\sum_i a_i \mu_{k+l,i}^M.$ This shows that $\mathcal{KLK}(M)$ is a subalgebra of $\calV(M)$ and $r_M$ is an algebra morphism.  By Proposition \ref{prop:r_M_bijection} it is an isomorphism.
\end{proof}

\subsection{Structure of $\widetilde{\mathcal{KLK}}$ as a $\KLK$-module}\label{sec_module}
Recall that $\widetilde{\mathcal{KLK}}(M) \subset \mathcal C(M)$ is the vector space spanned by the curvature measures $\Delta_{k,q}^M, B_{k,q}^M$.
In the present section we show that $\widetilde{\mathcal{KLK}}(M)$ is a module over $\mathcal{KLK}(M)$, canonically isomorphic to $\Curv^{\U(n)}$ with respect to
the isomorphism $\Val^{\U(n)} \simeq \KLK(M)$.

\begin{proposition}\label{coro_mod_isom} 
The linear map $R_M\colon \Curv^{\U(n)}\to \widetilde{\mathcal{KLK}}(M)$ defined by $R_M(\Delta_{k,p})=\Delta_{k,p}^M$ and $R_M(B_{k,p})=B_{k,p}^M$ is a bijection.
\end{proposition}

\begin{proof} The surjectivity of $R_M$ is  clear by the definition of $\widetilde{\mathcal{ KLK}}(M)$.

 Let $\mathcal C_{2n} \subset\cdots\subset\mathcal C_0$ be the filtration of $\mathcal C(M)$ introduced in \cite[\S 3]{solanes_wannerer}. By \cite[Propositions 3.2 and 3.5]{solanes_wannerer} and Lemma~\ref{lem:filtration_delta}, there is an isomorphism $\mathcal C_k/\mathcal C_{k+1}\to \Curv_k$ which maps $\Delta_{k,p}^M\mapsto \Delta_{k,p}$ and $B_{k,p}^M\mapsto B_{k,p}$. The injectivity of $R_M$ follows easily, as in the proof of Proposition~\ref{prop:embedded}. 
\end{proof}

\begin{proposition}\label{prop_module}
There is a $\mathcal{KLK}(M)$-module structure on $\widetilde{\mathcal{KLK}}(M)$  such that the following diagram commutes:
\begin{center}
\begin{tikzpicture}
  \matrix (m) [matrix of math nodes,row sep=3em,column sep=4em,minimum width=2em]
  {
   \Val^{\U(n)} \otimes \Curv^{\U(n)}   &  \Curv^{\U(n)}  \\
    \KLK(M) \otimes \widetilde{\KLK}(M)  & \widetilde{\KLK}(M)\\
    \mathcal V(M) \otimes \mathcal C(M) &  \mathcal C(M)\\};
    \path[-stealth]
   (m-1-1) edge node [] {}  (m-1-2)
   (m-2-1) edge node [] {}  (m-2-2)
   (m-3-1) edge node[]{} (m-3-2);
    \path[draw,->] 
   (m-1-1) edge node [right]{$r_M \otimes R_M$} (m-2-1)
     (m-1-2) edge node [right]{$R_M$} (m-2-2);
      \path[draw,right hook->] 
   (m-2-1) edge node []{} (m-3-1)
     (m-2-2) edge node []{} (m-3-2);
\end{tikzpicture}
\end{center}
where the horizontal maps are the module structures. Explicitly, 
\begin{equation}\label{eq:compatiblity_r_R}
 R_M(\mu\cdot \Phi)=r_M(\mu)\cdot R_M(\Phi),\qquad \forall\mu\in \Val^{\U(n)},\forall \Phi\in\Curv^ {\U(n)}.
\end{equation}
The structure of $\widetilde{\mathcal{KLK}}(M)$ as a module over $\mathcal{KLK}(M)$ is thus explicitly determined by Propositions 5.7 and 5.10 of \cite{bernig_fu_solanes}.
\end{proposition}

\begin{proof}
For $M\subset \CC^N$, the relation \eqref{eq:compatiblity_r_R} follows from Propositions \ref{prop_B_eq_B} and \ref{thm_weyl}, and the compatibility of the module product with restrictions (cf. \cite{bernig_fu_solanes}*{Propositions 2.3 and 2.4}). The general case follows by polynomiality, similarly to the proof of Theorem \ref{thm:alg_isom}, using { Proposition~\ref{prop_double_fibration} and Theorem \ref{thm_pointwise_lemma}}. 
\end{proof}

We will denote by $\glob_M\colon \mathcal C(M)\to  \calV(M)$  and $\glob_0\colon \Curv^{\U(n)} \to \Val^{\U(n)}$ the respective globalization maps of $M$ and $\C^n$. The following proposition shows that $R_M$ is an isomorphism of modules lifting the algebra isomorphism $r_M$.

\begin{proposition} \label{prop_glob_delta_beta} 
The following diagram commutes:
\begin{center}
\begin{tikzpicture}
  \matrix (m) [matrix of math nodes,row sep=3em,column sep=4em,minimum width=2em]
  {
   \Curv^{\U(n)}   &  \widetilde{\KLK}(M) \\
    \Val^{\U(n)}  & \KLK(M)\\};
  \path[draw,->]
    (m-1-1) edge node [above] {$R_M$} (m-1-2)
    (m-2-1) edge node [above] {$r_M$}(m-2-2);

    \path[draw,->>]        
    (m-1-1) edge node [right] {$\glob_0$}  (m-2-1)
    (m-1-2) edge node [right] {$\glob_M$}  (m-2-2);
   
\end{tikzpicture}
\end{center}

In particular $\glob_M(B_{k,q}^M)=\glob_M(\Delta_{k,q}^M)$.
\end{proposition}

\begin{proof}
We have $\glob_M(R_M(\Delta_{k,q}))=\mu_{k,q}^M=r_M(\glob_0 \Delta_{k,q})$ by the definitions.  For a general $\Phi\in\Curv^{\U(n)}$ we have by \cite[Theorem 5.13]{bernig_fu_solanes}
\[
 \Phi=\mu\cdot \Delta_{0,0}+\varphi\cdot N_{1,0}
\]
for some $\mu,\varphi\in\Val^{\U(n)}$ and $N_{1,0}=\Delta_{1,0}-B_{1,0}$. By \eqref{eq:compatiblity_r_R} and the compatibility of module product and globalization we get
\begin{displaymath}
\glob_M( R_M(\Phi))=r_M(\mu)+r_M(\varphi)\cdot\glob_M(\Delta_{1,0}^M-B_{1,0}^M)
\end{displaymath}
whereas
\begin{displaymath}
 r_M(\glob_0(\Phi))=r_M(\mu).
\end{displaymath}

The proof is completed by noting that $\glob_M(\Delta_{1,0}^M)=\glob_M(B_{1,0}^M)$ by Proposition~\ref{prop_glob_b10}, 
 since $\Delta_{1,0}^M =2 \Gamma_{1,0}^M, B_{1,0}^M =2 \tilde \Gamma_{1,0}^M$.

\end{proof}

\section{Integral geometry of complex space forms}\label{sec_space forms}

In this section we specialize our results to the complex space forms $M=\mathbb CP^n_\lambda$ with metric of constant holomorphic curvature $4\lambda$, viz. the complex projective space for $\lambda>0$, the flat euclidean space for $\lambda=0$, and the complex hyperbolic space when $\lambda<0$. For $\lambda\neq 0$ we let $G_\lambda$ be the isometry group of $\mathbb CP^n_\lambda$, while $G_0$ will be the group generated by $\U(n)$ and the translations of $\C^n$. We abbreviate
\begin{displaymath}
r_\lambda:= r_{\CC P^n_\lambda},\quad  R_\lambda:= R_{\CC P^n_\lambda}.
\end{displaymath}

The algebra $\mathcal{KLK}(M)$ coincides with the algebra of $G_\lambda$-invariant valuations which we denote by $\mathcal V_\lambda^n$. Indeed, the former is clearly included in the latter, and by Theorem \ref{thm:alg_isom} they have the same dimension. Similarly, the module $\widetilde{\mathcal{KLK}}(M)$ equals the module of $G_\lambda$-invariant curvature measures $\calC_\lambda^n$, which can be identified with $\Curv^{\U(n)}$ via the transfer map, as in \cite{bernig_fu_solanes}. 

A curious feature of Hermitian integral geometry is the existence of multiple automorphisms of the relevant spaces of valuations and curvature measures. Section \ref{sect_intertwine} presents a new instance of this phenomenon.

\subsection{The isomorphisms $r_\lambda$ and $R_\lambda$}
The algebra $\calV^n_\lambda$ is generated by two elements $t,s$ (see \cite[eq. (3.1)]{bernig_fu_solanes}),  which we do not distinguish notationally from the generators $t,s$ of $\Val^{\U(n)}$. Let us put
\[
 \sigma_\lambda:= r_\lambda(s).
\]
Our first goal is to describe the algebra isomorphism $r_\lambda$ in these terms.

Let $\mathrm{Beta}^n\subset\Curv^{\U(n)} \simeq \calC_\lambda^n $ denote the span of the curvature measures $B_{k,p}$, and denote $\mathrm{Beta}_k^n=\mathrm{Beta}^n \cap \Curv_k$. Equivalently, $\mathrm{Beta}^n$ is the space of curvature measures represented by invariant differential forms  divisible by $\bm \beta$.

Recall the first variation map $\delta\colon\calV(N) \to \calC(N)$ defined for a riemannian manifold $N^m$ by $\delta(\mu)=[0,\contr{T} (\tau(\mu))]$, where $T$ is the Reeb vector field of $SN$, and $\tau(\mu)$ is the $m$-form associated to $\mu$ by \eqref{eq:zeta_tau}.

\begin{lemma} 
Let $\varphi\in \calV_\lambda^n$. Then $\delta\varphi\in\mathrm{Beta}^n$ if and only if $\varphi$ is a polynomial in $s$.
\end{lemma}

\begin{proof}
	That $\delta \varphi\in \mathrm{Beta}^n$ if $\varphi$ is a polynomial in $s$ was proved in  \cite{abardia_gallego_solanes}*{Corollary 4.5}. To prove the converse, we proceed as follows. 
By \cite[Prop. 4.6]{bernig_fu_hig} the composition 
\begin{displaymath}
p:\Val_{k}^{\U(n)}\stackrel{\delta}\rightarrow \Curv_{k-1}^{\U(n)}\to \Curv_{k-1}^{\U(n)}/\mathrm{Beta}_{k-1}^n 
\end{displaymath}
is surjective. Since
\begin{align*}
 \dim\Val_k^{\U(n)} & =\left\lfloor \frac{\min(k,2n-k)}{2}\right\rfloor +1,\\
 \dim \Curv_{k-1}^{\U(n)}/\mathrm{Beta}_{k-1}^n & = \left\lfloor \frac{\min(k,2n-k)-1}2\right\rfloor +1,
\end{align*}
it follows that $p$ is injective for odd $k$ and has one-dimensional kernel if $k$ is even. By \cite{abardia_gallego_solanes}*{Corollary 4.5} this kernel is spanned by $s^{\frac{k}2}$. 

Suppose $\varphi\in \mathcal{W}_k$ and let $\varphi_0\in\Val^{\U(n)}$ be its image under the projection  $\mathcal{W}_k\rightarrow \mathcal{W}_k/\mathcal{W}_{k+1}\cong\Val_k^{\U(n)}$. Modulo $\oplus_{i\geq k}\Curv_i^{\U(n)}$ one has $\delta\varphi \equiv\delta\varphi_0$, e.g., by  \cite[Prop. 3.7]{abardia_gallego_solanes}. Hence $\delta\varphi_0\in \mathrm{Beta}_{k-1}^n$, by the assumption. It follows from the previous paragraph that  $\varphi_0=a s^{\frac{k}2}$ for some $a\in \RR$  if $k$ is even and hence  $\varphi -as^{\frac{k}2}\in \mathcal{W}_{k+1}$ in this case.  If $k$ is odd, then $\varphi_0=0$ and hence $\varphi \in \mathcal{W}_{k+1}$. In either case we can proceed by induction.
\end{proof}

\begin{proposition} \label{prop:sigma}
 The algebra isomorphism $r_\lambda\colon \Val^{\U(n)}\to \mathcal V_\lambda^n$ is given by 
 \begin{align}
r_\lambda(t) & =t, \nonumber \\
r_\lambda(s) & =\sigma_\lambda=\frac{s}{1-\lambda s}=\sum_{i=1}^n \lambda^{i-1}s^i. \label{eq_sigma_to_s}
 \end{align}
\end{proposition}

\begin{proof}  
By \cite[Prop. 4.6]{bernig_fu_hig}, on $\CC^m$ we have $\delta s\in \mathrm{Beta}^m$.
 This means that $\contr{T} \tau(s)\equiv\bm\beta\wedge\eta$ modulo the ideal $(\bm\alpha,d\bm\alpha)$ for some $\eta\in \Omega^{2m-2}(S\CC^m)$. Therefore
\begin{displaymath}
 \tau(s)= \bm\alpha\wedge \contr{T} \tau(s)=\bm\alpha\wedge\bm\beta\wedge \eta+\bm\alpha\wedge d\bm\alpha\wedge\rho
\end{displaymath}
for some $\rho\in\Omega^{ 2m-3}(S\CC^m)$.

Let now $M\subset \CC^m$ be a complex submanifold,  and put $\sigma_M=r_M(s)=s|_M$. By \cite[eq. (3.5.4)]{alesker_intgeo}, and using the notation of Subsection \ref{subsec_restrictions}, we have
\begin{displaymath}
 \tau(\sigma_{M}) =P_* F^* (\tau(s))=P_* F^*( \bm\alpha\wedge\bm\beta\wedge\eta +\bm\alpha\wedge d\bm\alpha\wedge \rho).
\end{displaymath}
Since (cf. \eqref{eq_F*tau}, \eqref{eq_F_beta})
\[F^*\bm\beta =\cos(\theta)\bm\beta,\quad  F^*\bm\alpha = \cos(\theta)\bm\alpha,\quad F^*d\bm\alpha=\cos(\theta) d\bm\alpha - \sin(\theta) d\theta\wedge \bm\alpha,\] we deduce that $\tau(\sigma_{M})$ belongs to the ideal  $I=(\bm\alpha\wedge\bm\beta, \bm\alpha\wedge d\bm\alpha)$.

 Let now  $M$ be an arbitrary K\"ahler manifold $M$ and put again $\sigma_M=r_M(s)=\frac1{2\pi}(\mu_{2,0}^M+2\mu_{2,1}^M)$. By Proposition \ref{prop_rumin_in_cartan_calculus}, the form $\tau(\sigma_M)$ depends polynomially on $\dblR^M$. By
Theorem \ref{thm_pointwise_lemma}, and the previous paragraph, this form vanishes  modulo $I$ for all  $M$. It follows that $i_T(\tau(\sigma_M))$ is a multiple of $\bm\beta$  modulo the ideal $(\bm\alpha, d\bm\alpha)$. Taking $M=\mathbb CP^n_\lambda$ it follows that $\delta\sigma_\lambda\in \mathrm{Beta}^n$.

We can therefore apply the previous lemma to deduce that
\begin{equation}\label{eq_sigma_ai}
\sigma_\lambda=\sum_{i=1}^n a_i s^i,\qquad a_i\in\mathbb R.
\end{equation} 
Since $\sigma_\lambda$ and $s^i$ are invariant under totally geodesic embeddings $\mathbb CP^n_\lambda\to\mathbb CP^{n+d}_\lambda$, the coefficients $a_i$ are independent of $n$.

Assume $\lambda>0$. Since $\mu_{2,0}$ has zero interior term, we have $\mu_{2,0}^M(M)=0$. Recalling that $\sigma_\lambda=\frac{1}{2\pi}(\mu_{2,0}^M+2\mu_{2,1}^M)$ and $t^2=\frac{2}{\pi}(\mu_{2,0}^M+\mu_{2,1}^M)$ we deduce
\begin{align*}
 \sigma_\lambda(\CP_\lambda^n)&=\frac12t^2(\CP_\lambda^n)=\frac{1}{\lambda}{n+1\choose2},
\end{align*} 
by Corollary 2.34 of \cite{bernig_fu_solanes}. On the other hand we have (cf. \cite[eq. (3.2)]{bernig_fu_solanes})
\begin{displaymath}
 s^i(\CP_\lambda^n)=\frac{n-i+1}{\lambda^i},\qquad 0 \leq i\leq n.
\end{displaymath}
Plugging these values into \eqref{eq_sigma_ai} gives
\begin{displaymath}
 \frac1\lambda{n+1\choose 2}=\sum_{i=1}^{n} a_i\frac{n-i+1}{\lambda^i}.
\end{displaymath}
Induction on $n$ yields $a_i=\lambda^{i-1}$. This proves the statement for $\lambda>0$. 

Since the curvature tensor of $\mathbb CP^n_\lambda$ is $R=\lambda \dblG$, every curvature measure $\Delta_{k,p}^M$ yields an element of $\Curv^{\U(n)}$  depending polynomially on $\lambda$. In particular $\sigma_\lambda=\glob(\frac1{2\pi}(\Delta_{2,0}^M+2\Delta_{2,1}^M))$  is the globalization of a polynomial in $\lambda$ with values in $\Curv^{\U(n)}$. By \cite[Lemma 5.4]{bernig_fu_solanes}, also the right hand side of \eqref{eq_sigma_to_s} is the  globalization of a polynomial family in  $\Curv^{\U(n)}$. It follows from \cite{bernig_fu_solanes}*{Lemma 2.7} that \eqref{eq_sigma_to_s} holds for all $\lambda$.
\end{proof}

\begin{remark} It is amusing to compare Proposition~\ref{prop:sigma} with a surprisingly similar fundamental relation arising in spherical integral geometry. Let $S^n_\lambda$ be the $n$-dimensional sphere of curvature $4\lambda$ and define 
$$ \phi = \frac{1}{\sqrt \lambda} \int_{\SO(n+1)} \chi(\bullet \cap g S^{n-1}_\lambda) dg,$$
where $S^{n-1}_\lambda$ is a totally geodesic subsphere, and $dg$ is the Haar probability measure. Then the algebra of $\SO(n+1)$-invariant valuations on $S^n_\lambda$ is generated by $\phi$. At the same time this algebra is generated by the first Lipschitz-Killing valuation $t$. The relation between these two generators is  
$$ t= \frac{\phi}{\sqrt{1 -\lambda \phi^2}}.$$
\end{remark}

Next we express $r_\lambda$  in terms of the bases  $\tau_{k,p},\tau_{k,p}^\lambda$ of $\Val^{\U(n)},\calV_\lambda^n$ introduced in \cite{bernig_fu_hig,bernig_fu_solanes} respectively.
\begin{proposition}Given $k$, put $r=\left\lfloor \frac{k}2 \right\rfloor$, and set $\epsilon=0$ for $k$ even and $\epsilon=1$ for $k$ odd. Define the function 
\begin{displaymath}
 g_{k,q}(\xi,\eta)= \frac1{r!} {r\choose q} \left(\frac\pi\lambda\right)^{r} \xi^{r-q}(1-\xi)^{q-r-\epsilon-\frac12}  \eta^q(1-\eta)^{-q-\frac12}.
\end{displaymath}Then
\begin{equation}\label{eq:egf_expansion_tau}
   r_\lambda(\tau_{k,q})=\sum_{i,j} \left(\frac\lambda\pi\right)^{i+j} \left.\frac{\partial^{i+j}g_{k,q}}{\partial \xi^i\partial \eta^j}\right|_{\xi=\eta=0}  \tau_{2i+2j+\epsilon,j}^\lambda.
  \end{equation}
\end{proposition}

\proof Let
\begin{align*}
v & = t^2(1-\lambda s)=\frac{t^2}{1+\lambda \sigma_\lambda},\\
u & = 4s-v=\frac{4\sigma_\lambda-t^2}{1+\lambda \sigma_\lambda}.
\end{align*} 

By \cite[Proposition 3.7]{bernig_fu_hig}, and reading $\sqrt{v}$ as $t\sqrt{1-\lambda s}$,
\begin{align}
r_\lambda(\tau_{k,q}) &=  \frac{\pi^k}{\omega_k (k-2q)!(2q)!} t^{k-2q}(4\sigma_\lambda-t^2)^q\\
&= \frac{\pi^k}{\omega_k (k-2q)!(2q)!} v^{\frac{k-2q}2} u^q(1+\lambda\sigma_\lambda)^{\frac{k}2}\\ 
&= \frac{\pi^k}{\omega_k (k-2q)!(2q)!} (1-\lambda s)^{-\frac{k}2} v^{\frac{k-2q}2} u^q.\label{eq:tau_to_vu}
\end{align}

Consider now the operator $\mathcal O\colon \RR[[\xi,\eta]]\to\RR[[v,u]]$ given by
\begin{displaymath}
\mathcal O\left(\sum_{m,p}{c_{m,p}}\xi^m\eta^p\right)=\sum_{m,p}{c_{m,p}}\frac{{m+p\choose m}}{{2m\choose m}{2p\choose p}}v^m u^p.
\end{displaymath}
By \cite[Lemma 3.16]{bernig_fu_solanes} we have
\begin{displaymath}
\mathcal O(g_{2i,0})=\frac{i!}{(2i)!}\left(\frac\pi\lambda\right)^i v^i\left(1-\frac{v+u}4\right)^{-1-i}. 
\end{displaymath}
Using 
\[
 g_{2r,q}= \left(\frac{4\pi}\lambda\right)^q\frac1{(2q)!}\eta^q\frac{\partial^q}{\partial\eta^q} g_{2r-2q,0}
\]
and 
\[
 \mathcal O\circ\left(\eta^i \frac{\partial^i}{\partial \eta^i}\right)= \left(u^i \frac{\partial^i}{\partial u^i}\right)\circ \mathcal O
\]
yields
\begin{align*}
 \mathcal O(g_{2r,q})&=\frac{r!}{(2r-2q)!(2q)!}\left(\frac\pi \lambda\right)^r u^q v^{r-q} \left( 1-\frac{u+v}4\right)^{-r-1}.
\end{align*}
By  \eqref{eq_tau_in_uv},
\begin{align*}
 \sum_{i,j} \left(\frac\lambda\pi\right)^{i+j} \left.\frac{\partial^{i+j}g_{2r,q}}{\partial \xi^i\partial \eta^j}\right|_{\xi=\eta=0}  \tau_{2i+2j,j}^\lambda&=(1-\lambda s)\sum_{i,j} \lambda^{i+j} \left.\frac{\partial^{i+j}g_{2r,q}}{\partial \xi^i\partial \eta^j}\right|_{\xi=\eta=0}  \frac{(i+j)!}{(2i)!(2j)!} v^i u^j\\
 &=(1-\lambda s)\sum_{i,j} \frac{1}{i!j!}\left.\frac{\partial^{i+j}g_{2r,q}}{\partial \xi^i\partial \eta^j}\right|_{\xi=\eta=0}  \frac{{i+j\choose i}}{{2i\choose i}{2j\choose j}} (\lambda v)^i (\lambda u)^j\\
 &=(1-\lambda s) \mathcal O(g_{2r,q})(\lambda v,\lambda u) \\
 &=\frac{r!}{(2r-2q)!(2q)!}\pi^r u^q v^{r-q} ( 1-\lambda s)^{-r},
\end{align*}
since $s=\frac{u+v}4$. Comparing with \eqref{eq:tau_to_vu}, this proves the proposition for $k=2r$. For odd $k$, the proof is similar using the operator $\mathcal P$ of \cite[Lemma 3.16]{bernig_fu_solanes}.
\endproof

\begin{proposition}\label{prop:egf_mu_mu}
Given $l$, put $r=\left\lfloor \frac{l}2 \right\rfloor$, and set $\epsilon=0$ for $l$ even and $\epsilon=1$ for $l$ odd. Define the function 
 \begin{displaymath}
 f_{l,q}(\xi,\eta) :=\frac{1}{r!} \binom{r}{q} \left(\frac{\pi}{\lambda}\right)^{r} \xi^{r-q} \eta^q(1-(\xi+\eta))^{-r-\epsilon-\frac{1}{2}+q} (1-\eta)^{-r-\frac12}.
 \end{displaymath}
 Then
\begin{displaymath}
   r_\lambda(\mu_{l,p})=\sum_{i,j} \left(\frac\lambda\pi\right)^{i+j} \left.\frac{\partial^{i+j}f_{l,p}}{\partial \xi^i\partial \eta^j}\right|_{\xi=\eta=0}  \mu_{2i+2j+\epsilon,j}^\lambda.
  \end{displaymath}
 \end{proposition}

\begin{proof}
Recalling that $\mu_{k,p}=\sum_q (-1)^{q+p}{q\choose p} \tau_{k,q}$ (cf. \cite{bernig_fu_hig}*{(39)}) we consider
\begin{align*}
 \sum_q (-1)^{p+q}{q\choose p}g_{l,q}(\xi,\eta)&=\frac1{r!}\left(\frac\pi\lambda\right)^{r}\sum_q (-1)^{p+q}{q\choose p}  {r\choose q} \xi^{r-q}(1-\xi)^{q-\epsilon-r-\frac12}  \eta^q(1-\eta)^{-q-\frac12}\\
 &=\frac1{r!}\left(\frac\pi\lambda\right)^{r}\xi^{r}(1-\xi)^{-r-\epsilon-\frac12}(1-\eta)^{-\frac12}(-1)^{p}\sum_q {q\choose p}  {r\choose q} x^{ q},
\end{align*}
where we have put $x=-\frac{\eta(1-\xi)}{\xi(1-\eta)}$. Using the identity \[\sum_q{q\choose p}{r\choose q} x^{q-p}={r\choose p}(1+x)^{r-p},\] which can be proved expanding $(1+x+y)^r$ in two different ways, and using also $1+x=\frac{\xi-\eta}{\xi(1-\eta)}$, we get
 \begin{align*}
\sum_q (-1)^{p+q}{p\choose q}g_{l,q}(\xi,\eta)
&=\frac1{r!}\left(\frac\pi\lambda\right)^{r}{r\choose p}(1-\xi)^{p-r-\epsilon-\frac12}\eta^p(1-\eta)^{-r-\frac12}(\xi-\eta)^{r-p}\\
&=f_{l,p}(\xi-\eta,\eta).
\end{align*}  

The trivial identity 
\begin{displaymath}
 \restrict{\frac{\partial^{i+j}}{\partial \xi^i\partial \eta^j}}{\xi=\eta=0} f(\xi-\eta,\eta)=
 \sum_r (-1)^r \binom j r \restrict{\frac{\partial^{i+j}}{\partial \xi^{i+r}\partial \eta^{j-r}}}{\xi=\eta=0} f(\xi,\eta).
\end{displaymath}
combined with \eqref{eq_def_tau} thus gives 
\begin{align*}
 r_\lambda(\mu_{l,p})&=\sum_{i,j} \left(\frac\lambda\pi\right)^{i+j} \sum_p (-1)^{p+q}{p\choose q}\left.\frac{\partial^{i+j}g_{l,q}}{\partial \xi^i\partial \eta^j}\right|_{\xi=\eta=0}  \tau_{2i+2j{ +\epsilon},j}^\lambda\\
 &=\sum_{i,j} \left(\frac\lambda\pi\right)^{i+j} \left.\frac{\partial^{i+j}f_{l,p}(\xi-\eta,\eta)}{\partial \xi^i\partial \eta^j}\right|_{\xi=\eta=0}  \tau_{2i+2j{ +\epsilon},j}^\lambda\\
 &=\sum_{i,j} \left(\frac\lambda\pi\right)^{i+j} \left.\frac{\partial^{i+j}f_{l,p}(\xi,\eta)}{\partial \xi^i\partial \eta^j}\right|_{\xi=\eta=0}  \mu_{2i+2j{ +\epsilon},j}^\lambda,
\end{align*}
as stated.
\end{proof}

We next establish an analogue of Proposition \ref{prop:egf_mu_mu} for the K\"ahler-Lipschitz-Killing curvature measures.

\begin{theorem}
Define the function 
 \begin{displaymath}
 h_{l,q} :=\frac{1}{1-\eta} f_{l,q}.
 \end{displaymath}
 Then 
 \begin{align}
  \label{eq:conjecture_Delta}\Delta_{l,q}^M & =\sum_{k,p\geq 0} \left(\frac{\lambda}{\pi}\right)^{k+p}\left.\frac{\partial^{k+p} h_{l,q}} {\partial\xi^k \partial\eta^p}\right|_{\xi=\eta=0} \Delta_{2k+2p+\epsilon,p}\\
  \label{eq:conjecture_N}N_{l,q}^M & =\sum_{l,p\geq 0} \left(\frac{\lambda}{\pi}\right)^{k+p}\left.\frac{\partial^{k+p} (\eta h_{l,q})} {\partial\xi^k \partial\eta^p}\right|_{\xi=\eta=0} \Delta_{2k+2p+\epsilon,p}\\
  & \quad +\sum_{k,p\geq 0} \left(\frac{\lambda}{\pi}\right)^{k+p}\left.\frac{\partial^{k+p} ((1-\eta) h_{l,q})} {\partial\xi^k \partial\eta^p}\right|_{\xi=\eta=0} N_{2k+2p+\epsilon,p}.\notag
 \end{align}
 \end{theorem}
 
 \begin{proof}
By Proposition \ref{prop:angular}, $\Delta_{l,q}^M \in \spann\{\Delta_{k,p}\}_{k,p} \subset\Curv^{\U(n)}$.
Let $\Phi_{l,q}$ be the right hand side of \eqref{eq:conjecture_Delta}.
By  \eqref{eq_glob_delta},
\begin{displaymath}
 \glob_\lambda(\Phi_{l,q})=\sum_{k,p\geq 0} \left(\frac{\lambda}{\pi}\right)^{k+p}\left.\frac{\partial^{k+p} } {\partial\xi^k \partial\eta^p}h_{l,q}(\xi,\eta)(1-\eta)\right|_{\xi=\eta=0} \mu_{2k+2p+\epsilon,p}^\lambda.
\end{displaymath}
Since $h_{l,q}(\xi,\eta)(1-\eta)=f_{l,q}(\xi,\eta)$, we deduce with Proposition \ref{prop:egf_mu_mu} that 
\begin{displaymath}
\glob_\lambda(\Phi_{l,q})=r_\lambda(\mu_{l,q})=\glob_\lambda(R_\lambda(\Delta_{l,q})). 
\end{displaymath}
Since $\glob_\lambda$ is injective on  $\spann\{\Delta_{k,p}\}_{k,p}\subset\Curv^{\U(n)}$, the relation \eqref{eq:conjecture_Delta} follows.

 Let $\Psi_{l,q}$ be the right hand side of \eqref{eq:conjecture_N}. Since $B_{k,p}=\Delta_{k,p}-N_{k,p}$ we have $R_\lambda(B_{l,p})=R_\lambda(\Delta_{l,p}-N_{l,p})$ and 
 \begin{displaymath}
  \Phi_{l,q}-\Psi_{l,q}=\sum_{k,p\geq 0} \left(\frac{\lambda}{\pi}\right)^{k+p}\left.\frac{\partial^{k+p} ((1-\eta) h_{l,q})} {\partial\xi^k \partial\eta^p}\right|_{\xi=\eta=0} B_{2k+2p+\epsilon,p}.
 \end{displaymath}
Proposition \ref{prop:egf_mu_mu}   and \eqref{eq_glob_delta} yield
\begin{displaymath}
 \glob_{\lambda}(\Phi_{l,q}-\Psi_{l,q})=r_\lambda(\mu_{l,q})=\glob_{ \lambda}(R_\lambda(B_{l,p})).
\end{displaymath}
Since $R_\lambda(B_{l,p}) \in \mathrm{Beta}^n$ and $\glob_\lambda$ is injective on $\mathrm{Beta}^n$, it follows that $R_\lambda(B_{l,p})=\Phi_{l,q}-\Psi_{l,q}$. We conclude that
\begin{displaymath}
 R_\lambda(N_{l,p})=R_\lambda(\Delta_{l,p})-R_\lambda(B_{l,p})=\Psi_{l,q},
\end{displaymath}
as stated.
 \end{proof}

\subsection{Isomorphisms intertwining the  kinematic operators} \label{sect_intertwine}

Recall  that the kinematic formulas in $\CP_\lambda^n$ are encoded by a co-algebra structure $k_\lambda\colon {\calV_\lambda^n\to \calV_\lambda^n\otimes \calV_\lambda^n}$ (cf. \cite{bernig_fu_solanes}*{\S 2.3.3}).
Let $\pd_\lambda\colon \calV_\lambda^n\to (\calV_\lambda^n)^*$ be the Alesker-Poincaré duality (see \cite{bernig_fu_solanes}*{Definition 2.14}).  As in \cite{bernig_fu_solanes}*{Eq. (3.45)}, $r_\lambda$ being an isomorphism implies that 
\begin{displaymath}
 J_\lambda=\pd_\lambda^{-1}\circ (r_\lambda^{-1})^*\circ\pd_0: \Val^{\U(n)} \to \mathcal V_\lambda^n
\end{displaymath}
is an isomorphism of co-algebras. This means in particular that $J_\lambda$ intertwines the kinematic operators of $\C^n$ and $\CP^n_\lambda$: 
\begin{center}
\begin{tikzpicture}
  \matrix (m) [matrix of math nodes,row sep=3em,column sep=4em,minimum width=2em]
  {
   \Val^{\U(n)} & \Val^{\U(n)} \otimes \Val^{\U(n)}    \\
   \mathcal V_\lambda^n & \mathcal V_\lambda^n \otimes \mathcal V_\lambda^n\\};
  \path[-stealth]
(m-1-1) edge node [above] {$k_0$}  (m-1-2)
(m-1-1) edge node [left] {$ J_\lambda$}   (m-2-1)
(m-2-1) edge node [above] {$k_\lambda$} (m-2-2)
(m-1-2) edge node [right] {$ J_\lambda \otimes  J_\lambda$} (m-2-2);
\end{tikzpicture}
\end{center}

Since $ J_\lambda$ admits a simple explicit description, we obtain a corresponding relation between 
between the flat and the curved kinematic operator.
\begin{proposition}
 \begin{displaymath}
  J_\lambda=(1-\lambda s)^{n+1} r_\lambda.
 \end{displaymath}
 \begin{displaymath}
 	k_\lambda(r_\lambda(\phi))= (r_\lambda\otimes r_\lambda) \circ k_0((1+\lambda s)^{-(n+1)} \phi).
 \end{displaymath} 
 \end{proposition}

\begin{proof}
 As in \cite[Prop. 3.18]{bernig_fu_solanes},  to completely determine $J_\lambda$ it is enough to show $ J_\lambda(\chi)=(1-\lambda s)^{n+1}$. This boils down to checking
 \begin{equation}\label{eq_check_coalgebra}
  \langle \pd_\lambda((1-\lambda s)^{n+1}),\varphi\rangle =\langle \pd_0(\chi), r_\lambda^{-1}(\varphi)\rangle
 \end{equation}
for all $\varphi=s^j t^{2i}$ with $i,j\geq 0$. Using 
\begin{displaymath}
 r_\lambda^{-1}(\varphi)=\left(\frac{s}{1+\lambda s}\right)^j t^{2i}=\sum_l(-\lambda)^l {j-1+l\choose l}s^{l+j}t^{2i}
\end{displaymath}
and
\cite[(3.48)]{bernig_fu_solanes}, the relation \eqref{eq_check_coalgebra} translates into the identity
\begin{displaymath}
\sum_k(-1)^k {n+1\choose k}{n-j-k+1\choose i+1}=(-1)^{n-j-i}{n-i-1 \choose n-i-j},
\end{displaymath}
which follows from \cite[(5.25)]{graham_knuth_patashnik}.  The second statement follows immediately from \cite{bernig_fu_solanes}*{Theorem 2.19} and 
\begin{equation}\label{eq:r_M} (1-\lambda s)^{n+1} r_\lambda(\phi)= (1+\lambda \sigma_\lambda)^{-(n+1)} r_\lambda(\phi)=  r_\lambda((1+\lambda s)^{-(n+1)}\phi).\end{equation}
\end{proof}

 A similar relation between $k_0,k_\lambda$ was given in \cite{bernig_fu_solanes}:
 For the  linear isomorphism $F_\lambda\colon \calV^n_0\to \calV^{n}_\lambda$, $F_\lambda(\mu_{k,q})= \mu^\lambda_{k,q}$, equation (3.50) of \cite{bernig_fu_solanes} shows that 
 $$ k_\lambda(F_\lambda(\phi))= (F_\lambda\otimes F_\lambda) \circ k_0((1-\lambda s) \phi).$$
Also the local kinematic operators of $\C^n$ and $\CP_\lambda^n$ are isomorphic; in fact, as was shown in \cite{bernig_fu_solanes}*{Theorem 2.23}, the transfer map $\tau\colon \calC_\lambda^n\to \Curv^{\U(n)}$ is an isomorphism of co-algebras. However, no relation between the semi-local kinematic operators $\bar k_0,\bar k_\lambda$ (see \cite{bernig_fu_solanes}*{\S 2.3.3}) was previously known. As a consequence of Proposition \ref{prop_module}, we obtain that also these kinematic operators are isomorphic.

\begin{theorem} 
The following diagram commutes
	\begin{center}
		\begin{tikzpicture}
			\matrix (m) [matrix of math nodes,row sep=3em,column sep=4em,minimum width=2em]
			{
				\Curv^{\U(n)} & \Curv^{\U(n)} \otimes \Val^{\U(n)}    \\
				\mathcal C_\lambda^n & \mathcal C_\lambda^n \otimes \mathcal V_\lambda^n\\};
			\path[-stealth]
			(m-1-1) edge node [above] {$\bar k_0$}  (m-1-2)
			(m-1-1) edge node [left] {$ R_\lambda$}   (m-2-1)
			(m-2-1) edge node [above] {$\bar k_\lambda$} (m-2-2)
			(m-1-2) edge node [right] {$ R_\lambda \otimes  J_\lambda$} (m-2-2);
		\end{tikzpicture}
	\end{center}
As a consequence, the semi-local kinematic operators are related through
 \begin{displaymath}
 \bar k_\lambda (R_\lambda(\Phi))=(R_\lambda\otimes r_\lambda)\circ \bar k_0((1+\lambda s)^{-(n+1)}\Phi).
\end{displaymath}
\end{theorem}

\begin{proof}
Let $\bar m_\lambda\in \mathrm{Hom}(\calV_\lambda^n\otimes\Curv^{\U(n)},\Curv^{\U(n)})=\mathrm{Hom}(\Curv^{\U(n)},\Curv^{\U(n)}\otimes(\calV_\lambda^n)^*)$ correspond to the module structure.
Given $\varphi\in \calV_\lambda^n$ and $\Phi\in\Curv^{\U(n)}$ we have
\begin{align*}
 \langle \bar m_\lambda(R_\lambda\Phi),\varphi\rangle&= \varphi\cdot R_\lambda\Phi\\
 &=R_\lambda(r_\lambda^{-1}(\varphi)\cdot \Phi)\\
 &=\langle (R_\lambda\otimes \mathrm{id})\circ \bar m_0 (\Phi), r_\lambda^{-1}(\varphi)\rangle\\
 &=\langle (R_\lambda\otimes (r_\lambda^{-1})^*)\circ \bar m_0 (\Phi),\varphi\rangle.
\end{align*}
Therefore
\begin{displaymath}
 \bar m_\lambda\circ R_\lambda= (R_\lambda\otimes(r_\lambda^{-1})^*)\circ \bar m_0.
\end{displaymath}
By \cite[Corollary 2.20]{bernig_fu_solanes}, this gives
\begin{align*}
 \bar k_\lambda\circ R_\lambda&=(\mathrm{id}\otimes\pd_\lambda^{-1})\circ \bar m_\lambda\circ R_\lambda\\
 &=(\mathrm{id}\otimes\pd_\lambda^{-1})\circ (R_\lambda\otimes(r_\lambda^{-1})^*)\circ (\mathrm{id}\otimes \pd_0)\circ \bar k_0\\
 & =  (R_\lambda \otimes J_\lambda) \circ \bar k_0.
\end{align*}
The second statement follows immediately from \cite{bernig_fu_solanes}*{Theorem 2.19} and \eqref{eq:r_M}.
\end{proof}

\section{Open problems}\label{sec_problems}

\begin{enumerate}
 \item Motivated by the results in \cite{fu_wannerer}, it seems a promising task to study the space of valuations of the form $[[0,\omega]]$, where $\omega$ is a differential form built from coframe, connection and curvature forms on a K\"ahler manifold such that the result does not depend on the choice of frame and that $d\omega$ is vertical. With respect to the Alesker product, they form a canonical algebra associated to a K\"ahler manifold. We conjecture that it equals the K\"ahler-Lipschitz-Killing algebra. Even for low-dimensional K\"ahler manifolds, it is difficult to prove or disprove the conjecture, since the space of invariant differential forms is of large dimension.  
 
 \item It is easy to show that the intrinsic volumes on riemannian manifolds are (up to linear combinations) the only valuations that satisfy the Weyl principle, i.e.\ are invariant under pullback along isometric immersions. The same statement holds true on pseudoriemannian manifolds \cite{bernig_faifman_solanes_part2}. Is it true that any valuation that can be associated to K\"ahler manifolds and that satisfies the Weyl principle is a linear combination of the $\mu_{k,p}$?  
 \item Can the Weyl principle for K\"ahler manifolds be extended to a larger class of manifolds including submanifolds of K\"ahler manifolds? Note that a submanifold of a K\"ahler manifold is endowed with a metric and a closed $2$-form, which may by degenerate.   
 \item  A natural question is whether the main results of this paper remain valid in the quaternionic K\"ahler case. It turns out that the analogue of Theorem \ref{thm_pointwise_lemma} is not true in the quaternionic K\"ahler case. If $M$ is a quaternionic K\"ahler manifold of real dimension $4n$  which is embedded in some quaternionic vector space $\mathbb{H}^{n+m}$, then each complex structure $I$ of $M$ is parallel. It follows that $R^M_p(IX,IY,Z,W)=R^M_p(X,Y,Z,W)$ for all $X,Y,Z,W \in V$. The curvature tensor of a general quaternionic K\"ahler manifold does not have this symmetry, as can be seen in the case of quaternionic projective space.  Hence the set of embedded algebraic quaternionic K\"ahler curvature tensors spans a subspace of positive codimension in the space of all algebraic quaternionic K\"ahler curvature tensors. Nevertheless, it could still be true that there is a canonical algebra of valuations associated to each quaternionic K\"ahler manifold. This would imply that the algebras of smooth isometry invariant valuations on quaternionic projective space $\operatorname{\mathbb{H}P}^n$ and the algebra of continuous, $\mathrm{Sp}(n)\cdot \mathrm{Sp}(1)$ and translation invariant valuations on $\mathbb{H}^n$ are isomorphic. It is not known, even in small dimensions, if this is true.
 \item More generally, can we associate a canonical algebra of valuations to each manifold with a (sufficiently flat) $G$-structure? An additional hint in this direction is provided by the recent results in \cite{solanes_wannerer}, where it is shown that the algebras of valuations on the isotropic spaces $(S^6,G_2)$ and $(\C^3,\overline{\mathrm{SU}(3)}))$ are isomorphic, as well as on the isotropic spaces $(S^7,G_2)$ and $(\RR^7,\overline{ G_2})$. 
\end{enumerate}

%

\def\cprime{$'$}


\end{document}